\documentclass[final]{siamart190516}

\usepackage{amssymb,amsmath,mathrsfs,graphicx,subcaption}
\usepackage{multirow}

\usepackage{lipsum}
\usepackage{amsfonts}
\usepackage{graphicx}
\usepackage{epstopdf}
\usepackage{algorithmic}
\usepackage{xcolor}
\ifpdf
  \DeclareGraphicsExtensions{.eps,.pdf,.png,.jpg}
\else
  \DeclareGraphicsExtensions{.eps}
\fi

\usepackage{enumitem}
\setlist[enumerate]{leftmargin=.5in}
\setlist[itemize]{leftmargin=.5in}
\usepackage{multirow}
\usepackage{hyperref}
\usepackage{cleveref}

\usepackage{array} % <- Preamble

\newtheorem{rem}{Remark}

\newtheorem{Lemma}{Lemma}

\graphicspath{{figure/}{fig/}}

\title{An Alternating \emph{Rank-k} Nonnegative Least Squares Framework (ARkNLS) \\ for Nonnegative Matrix Factorization}

\author{Delin Chu \footnotemark[1]
\and Wenya Shi \footnotemark[2]
\and Srinivas Eswar \footnotemark[3]
\and  Haesun Park \footnotemark[4] }

\begin{document}
\date{today}

\maketitle
\renewcommand{\thefootnote}{\fnsymbol{footnote}}
\footnotetext[1]{Department of Mathematics, National University of Singapore, Singapore
                  119076. E-mail: matchudl@nus.edu.sg.  This author was
                  supported in part by NUS Research Grant R-146-000-187-112. }
\footnotetext[2]{School of Mathematics, China University of Mining and Technology,
                   Xuzhou, 221116, Jiangsu, P.R. China. E-mail: shiwenyaer@163.com.}
\footnotetext[3]{School of Computational Science and Engineering, Georgia Institute of Technology, Atlanta,
                  GA 30332-0765, USA. E-mail: seswar3@cc.gatech.edu}
\footnotetext[4]{School of Computational Science and Engineering, Georgia Institute of Technology, Atlanta, GA 30332-0765, USA. E-mail: hpark@cc.gatech.edu.}

\renewcommand{\thefootnote}{\arabic{footnote}}

%\maketitle

% REQUIRED

\begin{abstract}
Nonnegative matrix factorization (NMF) is a prominent technique for data dimensionality reduction that has been widely used for text mining,
computer vision, pattern discovery, and bioinformatics. In this paper, a framework called ARkNLS (Alternating \emph{Rank-k} Nonnegativity 
constrained Least Squares) is proposed for computing NMF. First, a recursive formula for the solution of the \emph{rank-k}
nonnegativity-constrained least squares (NLS) is established. This recursive formula can be used to derive the closed-form solution
for the \emph{Rank-k NLS} problem for any integer $k\geq 1$. As a result, each subproblem for an alternating \emph{rank-k} nonnegative least squares framework can be obtained based on this closed form solution. Assuming that all matrices involved in \emph{rank-k NLS} in the context of NMF computation are of full rank, two of the currently best NMF algorithms HALS (hierarchical alternating least squares) and ANLS-BPP (Alternating NLS based on Block Principal Pivoting) can be considered as special cases of ARkNLS with $k=1$ and $k=r$ for rank $r$ NMF, respectively.
This paper is then focused on the framework with $k=3$, which leads to a new algorithm for NMF via the closed-form solution of the
\emph{rank-3 NLS} problem. Furthermore, a new strategy that efficiently overcomes the potential singularity problem in \emph{rank-3 NLS} 
within the context of NMF computation is also presented. Extensive numerical comparisons using real and synthetic data sets demonstrate that the proposed 
algorithm provides state-of-the-art performance in terms of computational accuracy and cpu time.
\end{abstract}

% REQUIRED
%

%REQUIRED

\section{ Introduction}
\setcounter{equation}{0}
\smallskip

Nonnegative matrix factorization (NMF) \cite{LS}, which performs a constrained low rank approximation of a matrix,
is a commonly used effective  method for data dimensionality reduction and other related tasks.
Given $A\in \mathbb{R}^{m\times n}$ and a desired low rank $r<\min\{m,n\}$ for approximation, NMF aims at finding
two low-rank nonnegative matrices ${U}^*\in \mathbb{R}^{m\times r}$ and ${V}^* \in \mathbb{R}^{n\times r}$ such that
\begin{equation} \label{nmf}
     (U^*, V^*)= {\rm arg} \{ \min \|A-U V^T\|_F^2,~  U \in \mathbb{R}^{m\times r},~ V\in \mathbb{R}^{n\times r},  ~U\geq0,~V \geq0 \},
\end{equation}
where $X\geq0$ means that all elements of a matrix $X$ are nonnegative. NMF problem was first proposed in  \cite{PT} as positive matrix
factorization and popularized due to \cite{LS}. By now it has become a powerful tool for data dimensionality reduction and has found
important applications in many fields such as clustering \cite{CLD,KDP,KYP,LPGF,DDP,GKH}, data mining \cite{PSBP,XLG,DHB+17}, signal
processing \cite{B}, computer vision \cite{BSJJZ,H,Ret}, bioinformatics \cite{BTGM,D,KP}, blind source separation \cite{CZPA}, spectral data analysis \cite{PPP}, and many others.

NMF problem (\ref{nmf}) has been studied extensively and many numerical methods are currently available.  Some of the successful methods
alternatingly compute the unknown low rank factors $U$ and $V$ iteratively, partitioning the unknowns ($U$, $V$) into two blocks. These
existing methods include the projected gradient method \cite{L,ZC1}, the interior point method \cite{MZ}, the projected quasi-Newton method
\cite{KSD,zc}, the active-set method \cite{BJ,BK,KP1,LH}, the active-set-like method \cite{KP2,kp11}, the alternating nonnegative least
squares based on block principal pivoting (ANLS-BPP) method \cite{kp11}. There also exist many variants of NMF \eqref{nmf} that add
constraints and/or penalty terms on $U$ and $V$ for better interpretation and representation of the characteristics of the tasks
\cite{AG,WZ} including sparse NMF \cite{CZPA,SWSL,KP}, orthogonal NMF \cite{LWP}, semi-NMF \cite{PCX}, Joint NMF \cite{DDP}, nonnegative
tensor factorization \cite{CP,CZPA,K,KHP,SBF}, manifold NMF \cite{WWZCX},  kernel NMF \cite{ZH}, regularized NMF \cite{SWSL,WSZ},
Symmetric NMF \cite{HY,TKCJ}, integer constrained~\cite{DLP18}, and so on. A comprehensive review of solving NMF can be found in \cite{KHP,WZ}. Some other NMF algorithms
that compute the solution by partitioning the unknowns into vector blocks include the multiplicative updates (MU) method \cite{LS} and
the hierarchical alternating least squares (HALS) method \cite{CP,CZA,CZPA} (which is also called the \emph{rank-1} residue iteration (RRI)
\cite{Ho}). More recently, random shuffling \cite{CWY} and randomized sampling techniques \cite{EMWK}  were used to accelerate HALS/RRI,
respectively.

In \cite{KHP}, it has been shown that most existing NMF algorithms can be explained using the block coordinate descent (BCD) framework.
Among the BCD framework-based algorithms, ANLS-BPP method \cite{kp11} and HALS/RRI method \cite{CP,CZA,CZPA,Ho}
have been shown to be the most effective in most situations \cite{KHP}. A \emph{rank-2} residue iteration method (RTRI) is proposed,
%for NMF (\ref{nmf}) in \cite{LZ}, 
which is similar to HALS/RRI. 
In these methods, all subproblems that the NMF algorithm encounters
are NLS with a matrix with one column in case of HALS/RRI or a matrix with two columns in case of RTRI. 

In this paper, we establish a new framework for computing NMF where the low rank factors $U$ and $V$ are partitioned into  blocks where
each block consists $k$ columns where $k$ can be any integer with $1 \leq k \leq r$.
We also present a recursive formula for the solution of the \emph{rank-k} nonnegativity-constrained
least squares (NLS). This recursive formula can be used to derive the closed-form solution of the \emph{rank-k NLS} problem for any
integer $k\geq 1$. As a result, we provide a framework called ARkNLS (alternating \emph{rank-k} nonnegative least squares) for NMF.
Based on the framework with $k=3$, we present a new algorithm for NMF via the closed-form solution for the \emph{rank-3 NLS} problem.
When $k=1$  our framework produces the HALS/RRI. When $k=r$, the framework can be reduced to ANLS method (Alternating NLS),
where the subproblems can be solved using the closed form solution based on recursion. However, as will be seen in the next section,
as $k$ becomes larger, the recursion for the closed form solution gets significantly complicated incurring high computational demand.
Accordingly, we conclude that computation of the NMF based on ARkNLS stays efficient when $k$ is relatively small, such as $k=1, 2$ or 3,
and for larger $k$, methods like ANLS-BPP is much more efficient. 

In the BCD based methods, the matrices that appear in all NLS subproblems are assumed to have full rank. In case any of these is rank deficient, then it requires a special remedy. In this paper, we will call this problem {\em  singularity} problem. 
It is well-known that the closed-form solution of the \emph{rank-k NLS} problem may run into a singularity problem in the HALS/RRI.
Typically some small values are added to avoid zero columns in a NLS subproblem in HALS,
but numerically  the solution produced by HALS/RRI has been known to be very sensitive to this 'small' value. 
%P3
In order to solve this
singularity problem, in \cite{LZ} instead of using the cyclic strategy of updating two adjacent columns in $U$ or $V$,
%${\bf u}_{i}$ and ${\bf u}_{i+1}$
%(and also ${\bf v}_{i}$ and ${\bf v}_{i+1}$) for $i=1, 2, \cdots$, 
two columns of
%indices $s$ and $t$ are selected so that ${\bf u}_s$ and ${\bf u}_{t}$
$U$ which most violate the optimality conditions are selected in terms of the reduced gradients
\[ H_{\rho}(U):=U-[U-\rho (UV^TV-AV)]_{+}, \quad H_{\rho}(V):=V-[V-\rho(V-\rho(VU^TU-A^TU)]_{+}, \]
%of $\|A-UV^T\|_F$: 
as follows: let
\[ h=\left[ \begin{array}{ccc} \|H_{\rho}(U)(:, 1)\|^2 & \cdots & H_{\rho}(U)(:, r)\|^2 \end{array}\right], \]
set $(\hat h, s)=\max(h)$,  $h(s)=0$, and $(\tilde h, t)=\max(h)$,
then ${\bf v}_s$ and ${\bf v}_{t}$ are updated.  The  details we refer to \cite{LZ}.
However, theoretically this new strategy cannot completely overcome the singularity problem,
since $\left[ \begin{array}{cc} {\bf u}_s & {\bf u}_t \end{array}\right]$ or
$\left[ \begin{array}{cc} {\bf v}_s & {\bf v}_t \end{array}\right]$ can still be rank deficient in some stage of iterations.
In addition, a parameter $0<\rho\leq 1$ is involved. It is not clear how this $\rho$ can be selected
appropriately and how it affects the computed results since indices $s$ and $t$ depend on the value of this parameter $\rho$.
Moreover, the idea used in \cite{LZ} cannot be used to develop methods for NMF based on the \emph{rank-k} residue iteration for $k\geq 3$.
We present a new strategy that efficiently overcomes the potential singularity problem within the context of NMF computation for $k= 1, 2,$ or 3.

Some notations and definition used in this paper are as follows.
A nonnegative constrained least square problem where the coefficient matrix has $k$ columns and of full rank will be called \emph{rank-k NLS}.
A lowercase letter, such as $x$, denotes a scalar; a boldface lowercase letter, such as ${\bf x}$, denotes a vector;
a boldface uppercase, such as $X$, denotes a matrix. For a matrix $X$, $X(i,:)$, $X(:,j)$ and $X(i,j)$ denote its $i$-th row, $j$-th column
and $(i,j)$-th element of $X$, respectively. We also let ${\bf x}(i)$ denote the $i$-th element of ${\bf x}$. For simplicity,
$X\geq0$ indicates that all the elements of X are nonnegative, $[X]_{+}=\max\{X,{\bf 0}\}$, ${\rm det}(X)$ is the determinant of $X$.

This paper is organized as follows. In Section \ref{sec2} an alternating \emph{rank-k} nonnegative least squares framework for NMF is
developed. The recursive formula for \emph{rank-k NLS} problem is established in Section \ref{sec3}. Then in Section \ref{sec4},
this framework with $k=3$ is specifically highlighted which leads to an new algorithm AR3NLS for NMF.
Numerical experiments are provided in Section \ref{sec5} on some synthetic as well as real data sets to illustrate the numerical behavior of our new
algorithms compared with HALS/RRI, RTRI and  ANLS-BPP. Finally some concluding remarks are given in Section \ref{sec6}.

     % Section 1
%\input{ranknnls} % 2 - this file is renamed and separated out into two files

\section{ARkNLS: A Rank-$k$ NLS based NMF Framework}\label{sec2}
\setcounter{equation}{0}
\smallskip

In this section, we present a framework called ARkNLS (alternating \emph{rank-k} nonnegative least squares) for NMF (\ref{nmf}). ARkNLS represents a set of block coordinate descent methods for NMF where a block consists of $k$ columns of $U$ or $k$ columns of $V$.
Specifically, for NMF (\ref{nmf}),
suppose ${U} \in \mathbb{R}^{m\times r}$ and ${V} \in \mathbb{R}^{n\times r}$
are partitioned  into $q$ blocks each as follows:
\begin{equation}\label{UV}
 U=  \left[ \begin{array}{ccc} U_1 & \cdots & U_q  \end{array}\right], \qquad
   V=  \left[ \begin{array}{ccc} V_1 & \cdots & V_q  \end{array}\right], \qquad
   U_1, \cdots, U_q\in \mathbb{R}^{m\times k}, \quad
   V_1, \cdots, V_q\in \mathbb{R}^{n\times k}.
\end{equation}
For simplicity of discussion, let us assume that $r/k=q$ is an integer, for now.
Later, we will show how the cases can be handled when $r$ is not divisible by $k$.
We have
\[ f(U, V)=f(U_1, \cdots, U_q, V_1, \cdots, V_q):= \|A-UV^T\|^2_F
    =\| U_1V_1^T+\cdots+U_qV_q^T-A\|_F^2. \]
Following the BCD scheme \cite{KHP}, $f$ can be minimized by iteratively solving the following problems:\\
for $i=1, \cdots, q$,
\begin{equation}\label{V}
  V_i=\arg \min_{\mathcal{Y} \geq 0} f(U_1, \cdots, U_q, V_1, \cdots, V_{i-1}, \mathcal{Y},  V_{i+1}, \cdots, V_q)
     =\arg \min_{\mathcal{Y}\geq 0} \| U_i\mathcal{Y}^T-(A-\sum_{l\not= i} U_lV_l^T)\|_F^2
\end{equation}
and for $i=1, \cdots, q$,
\begin{equation}\label{U}
  U_i=\arg \min_{\mathcal{Y}\geq 0} f(U_1, \cdots, U_{i-1}, \mathcal{Y}, U_{i+1}, \cdots, U_q, V_1, \cdots, V_q)
     =\arg \min_{\mathcal{Y}\geq 0} \| V_i \mathcal{Y}^T-(A-\sum_{l\not=i} U_lV_l^T)^T\|_F^2 .
\end{equation}
The above yields the alternating \emph{rank-k} nonnegative
least squares framework ARkNLS for NMF (\ref{nmf}) which is summarized
in Algorithm \ref{alg1-1}.
The proposed ARkNLS is a general framework where we can choose any integer $k$.
When $k=1$, it represents a $2r$ block BCD and is reduced to HALS/RRI.
When $k=r$, it represents a 2 block BCD and ANLS-BPP is one of such algorithms.

\begin{algorithm}[h]
\caption{{\bf ARkNLS}: Alternating \emph{rank-k} Nonnegative Least Squares Framework for NMF} \label{alg1-1}
\begin{algorithmic}
\STATE 1.~Assume $A \in \mathbb{R}^{m\times n}$ and $r \leq \min(m,n)$ are given.\\
Initialize $U\in \mathbb{R}^{m\times r}$ and $V\in \mathbb{R}^{n\times r}$ with $\{U, V\}\geq 0$ \\
Partition as in (\ref{UV}) where each block $U_i$ and $V_i$ has $k$ columns for some integer $k$,
such that $qk = r$. \\
Normalize the columns of $U$.
\STATE 2.~\textbf{Repeat}
\STATE 3. \hspace{0.1in} For $i=1, \cdots, q$, \\
\hspace{0.3in} update $V_i$ by solving \emph{rank-k NLS} problems (\ref{V})
\STATE 4. \hspace{0.1in} For $i=1, \cdots, q$, \\
\hspace{0.3in} update $U_i$ by solving \emph{rank-k NLS} problems (\ref{U})
\STATE 5.~\textbf{Until} a stopping criterion is satisfied
\end{algorithmic}
\end{algorithm}

The subproblems (\ref{V}) and (\ref{U}) are NLS with multiple right hand sides, of the form
\begin{equation}\label{gnmf}
 \min_{Y\geq 0} \|GY- B\|_F^2
\end{equation}
%with $Y=\mathcal{Y}^T$, and
where
$ G=U_i$ and $B=A-\sum_{l\neq i}U_lV_l^T$ for $i=1, \cdots, q$
in (\ref{V}).
Likewise
$G=V_i$ and $B=(A-\sum_{l\neq i}U_lV_l^T)^T$ for $i=1, \cdots, q$
in (\ref{U}).
Therefore, an NLS with multiple right hand side vectors (\ref{gnmf})
is the core problem in ARkNLS, which is our focus in this section.

As illustrated in the following theorem,
a significant aspect of ARkNLS framework is that the NLS (\ref{gnmf}),
accordingly, the subproblems (\ref{V}) and (\ref{U}) have closed-form solutions.

\begin{theorem}\label{thm1-coro}
Assume that $ {B} \in \mathbb{R}^{m\times n}$, $ {G} \in \mathbb{R}^{m\times k}$ and ${\bf g}_{k+1}\in \mathbb{R}^m$ are given,
where  $rank({G})=k$ and $rank(\left[ \begin{array}{cc} G & {\bf g}_{k+1}\end{array}\right])=k+1$.
Denote the unique solution  of the \emph{rank-k NLS} problem (\ref{gnmf}) by $ {S}({G}, {B}) \in \mathbb{R}^{k\times n}$.
Then the unique solution of the \emph{rank-(k+1) NLS} problem
\begin{eqnarray}\label{gn5}
\left[ \begin{array}{c}  {Y}^{\star} \\  {\bf y}_{k+1}^{\star}\end{array}\right]=\arg
\min_{ {Y} \geq 0, {\bf y}_{k+1}\geq 0}
    \| \left[ \begin{array}{cc} {G} & {\bf g}_{k+1} \end{array}\right]
     \left[ \begin{array}{c} {Y} \\ {\bf y}_{k+1} \end{array}\right]-{B} \|_F^2
\end{eqnarray}
is given by
\begin{eqnarray*}
\begin{cases}
{\bf y}_{k+1}^{\star}&=\frac{1}{\|{\bf g}_{k+1}\|^2}  \big[  {\bf g}^T_{k+1}
\big( {B} - {G} \cdot
      {S}({G}-\frac{{\bf g}_{k+1}{\bf g}_{k+1}^T}{\|{\bf g}_{k+1}\|^2} {G},
      {B}-\frac{ {\bf g}_{k+1} {\bf g}_{k+1}^T }{\|{\bf g}_{k+1}\|^2} {B}) \big)
      \big]_+ \in \mathbb{R}^{1\times n}, \\
 {Y}^{\star}&={S}({G}, {B}-{\bf g}_{k+1} {\bf y}^*_{k+1}).
\end{cases}
\end{eqnarray*}
\end{theorem}

\begin{proof} Theorem \ref{thm1-coro} follows trivially from Theorem \ref{thm1} which is proved in the next section. \end{proof}

Theorem \ref{thm1-coro} enables us to derive the closed form solution of the \emph{rank-k NLS} problem (\ref{gnmf})
and accordingly for the subproblems (\ref{V}) and (\ref{U}). In addition, according to Theorem 1 in  \cite{KHP}, which characterizes the convergence property of the BCD scheme for NMF, the convergence property
of ARkNLS can be stated as follows.

\begin{theorem}\label{thm-convergence} If $U_i$ and $V_i$, for $i=1, \cdots, q$, are of full column rank  throughout
all the iterations and the unique minimums in (\ref{V}) and (\ref{U}) are attained at each updating step, every limit point
of the sequence $\{ (U,~V)^{(i)} \}$ generated by ARkNLS algorithm is a stationary point of the NMF (\ref{nmf}). Note that this uniqueness cnodition is not needed when $k=r$, i.e., $q=1$.
\end{theorem}

It is important to note that for the above theorem to be applicable, we need to have unique
solution for each subproblem when $q>2$ \cite{LH}. When any $U_i$ or $V_i$ is rank deficient,
then the uniqueness of the solution will be violated, and therefore, the above theorem cannot
be applied for proof of convergence to a stationary point.
In case of HALS, $q = r$ and the uniquenss of the solution for all subproblems cannot be guaranteed
when a block $U_i$ or $V_i$ (in case of HALS, these will consist of one vector) becomes
rank deficient (zero vectors).
For NMF algorithms based the BCD scheme with 2 blocks like ANLS-BPP
(where updates are alternated between blocks
$U$ and $V$), the uniqueness is not required and the convergence result
of the above theorem is applicable.
However, in case of NMF, due to the nonuniqueness of the NMF solution $(U^*, V^*)$, we can modify the subproblems
so that the subproblems are always of full rank and therefore, the solution for NLS is unique.
More detailed discussions are presented in Subsection \ref{singularity}.
  %Section 2
\section{Recursive Formula for the Solution of \emph{Rank-k NLS} Problem}\label{sec3}
\setcounter{equation}{0}
\smallskip

Problem  (\ref{gnmf})
can be decoupled into independent NLS problems with single right-hand side vector as
\[ \min_{{Y}(:, j) \in \mathbb{R}^{k \times 1},
                                     {Y}(:, j)\geq 0}
                       \|G  {Y}(:, j)-{B}(: j)\|_F^2.
\]
Accordingly, in order to derive the closed-form solution of the \emph{rank-k NLS} problem (\ref{gnmf}) (and so problems (\ref{V}) and (\ref{U})),
we first establish the recursive formula for the solution of the following \emph{rank-k NLS} problem
\begin{eqnarray}\label{n4}
\min_{{\bf y}\geq{\bf 0}} \|G{\bf y}-{\bf b}\|,
\end{eqnarray}
where ${\bf b}\in \mathbb{R}^m$, $G\in \mathbb{R}^{m\times k}$, and ${\rm rank} (G) =k$.

\begin{Lemma}\label{lem0}
Given a continuous and  convex function $f({\bf z})$, and two nonempty closed convex sets $\mathcal{T}$ and $\mathcal{C}$
satisfying $\mathcal{T}\cap\mathcal{C}\neq\emptyset$, assume
\[ \widetilde{\bf z}={\rm arg} \min_{{\bf z}\in \mathcal{T}} f({\bf z}),
\]
and $\widetilde{\bf z}$ is finite.
Assume further that the following constrained optimization problem
\begin{eqnarray}\label{n1}
\min_{{\bf z}\in \mathcal{T}\cap\mathcal{C}} f({\bf z})
\end{eqnarray}
has a finite solution.\\
\hspace{0.1in}
1) If $\widetilde{\bf z}\in \mathcal{C}$, then
${\bf z}^*=\widetilde{\bf z}={\rm arg} \min_{{\bf z}\in \mathcal{T}\cap\mathcal{C}} f({\bf z})$. \\
\hspace{0.1in}
2) If $\widetilde{\bf z}\notin \mathcal{C}$, then there exists a ${\bf {z}}^*\in\mathcal{T}\cap\mathcal{C}_{edge}$ satisfying
${\bf {z}}^*={\rm arg} \min_{{\bf z}\in \mathcal{T}\cap\mathcal{C}} f({\bf z})$,
where $\mathcal{C}_{edge}$ denotes the boundary of $\mathcal{C}$.
\end{Lemma}

\begin{proof} Part 1) is obvious. In the following we  prove Part 2).

Let ${\bf \widehat{z}}\in\mathcal{T}\cap\mathcal{C}$ be finite and
\begin{equation} {\bf \widehat{z}}={\rm arg} \min_{\mathcal{T}\cap\mathcal{C}} f({\bf z}). \label{*}
\end{equation}
It is clear that Part 2) follows with ${\bf {z}}^*={\bf \widehat{z}}$ if
${\bf \widehat{z}}\in\mathcal{T}\cap\mathcal{C}_{edge}$.
Otherwise, suppose that
${\bf \widehat{z}}\in\mathcal{T}\cap\mathcal{C}_{int}$,  where $\mathcal{C}_{int}$ is the interior of $\mathcal{C}$.
 Note that $\widetilde {\bf z}\notin \mathcal{C}$, $\widetilde{\bf z}\in \mathcal{T}$ and
 ${\bf \widehat{z}}\in \mathcal{T}\cap\mathcal{C}_{int}$.
Therefore, there exists
${\bf {z}}^* \in\mathcal{T}\cap\mathcal{C}_{edge}$
such that for some $t\in(0,1)$, we have
\[ {\bf {z}}^*=(1-t) \widetilde{\bf z}+t{\bf \widehat{z}}\in \mathcal{T}\cap\mathcal{C}_{edge}. \]
Furthermore, $ f(\widetilde{\bf z})\leq f(\widehat{\bf z})$.
Hence,
$f({\bf {z}}^*)\leq(1-t)f(\widetilde{\bf z})+t f({\bf \widehat{z}})\leq(1-t)f({\bf \widehat{z}})+tf({\bf \widehat{z}}) =f({\bf \widehat{z}})$,
which together with (\ref{*}) yields that
$ f({\bf {z}}^*)=f({\bf \widehat{z}})$
and
${\bf {z}^*}={\rm arg} \min_{{\bf z}\in \mathcal{T}\cap\mathcal{C}} f({\bf z})$.
\end{proof}

\begin{Lemma}\label{lem*}
Assume $G\in \mathbb{R}^{m\times k}$, ${\rm rank} (G) =k$,
and ${\bf b}\in \mathbb{R}^m$.
Then the  solution  of the \emph{rank-k NLS} problem (\ref{n4}) is unique.
\end{Lemma}

\begin{proof} The proof is trivial and we refer it to  \cite{LH}.
\end{proof}

Now, we establish the recursive formula for the solution of \emph{rank-k NLS} problem (\ref{n4}).

\begin{theorem}\label{thm1}
Assume $G\in \mathbb{R}^{m\times k}$, ${\bf g}_{k+1}\in \mathbb{R}^m$,
${\bf b}\in \mathbb{R}^m$ are given, where
$G$ and $\left[ \begin{array}{cc} G & {\bf g}_{k+1}\end{array}\right]$ are of full column rank.
Denote the unique solution  of the \emph{rank-k NLS} problem (\ref{n4}) by $s(G, {\bf b}) \in \mathbb{R}^k$.
Then the unique solution of the \emph{rank-(k+1) NLS} problem
\begin{eqnarray}\label{n5}
\left[ \begin{array}{c} {\bf y}^{\star} \\  y_{k+1}^{\star}\end{array}\right]= \arg
    \min_{{\bf y}\geq{\bf 0},y_{k+1}\geq0} \| \left[ \begin{array}{cc} G & {\bf g}_{k+1} \end{array}\right]
                                              \left[ \begin{array}{c} {\bf y} \\ y_{k+1}\end{array}\right]-{\bf b}\|
\end{eqnarray}
is given by
\begin{eqnarray*}
\begin{cases}
y_{k+1}^{\star}&=\frac{1}{\|{\bf g}_{k+1}\|^2}\big[{\bf g}^T_{k+1}\big({\bf b}-G\cdot
s(G-\frac{{\bf g}_{k+1}{\bf g}_{k+1}^T}{\|{\bf g}_{k+1}\|^2} G, {\bf b}-\frac{{\bf g}_{k+1} {\bf g}_{k+1}^T }{\|{\bf g}_{k+1}\|^2} {\bf b}))
\big]_+, \\
{\bf y}^{\star}&=s(G, {\bf b}-{\bf g}_{k+1}y^*_{k+1}).
\end{cases}
\end{eqnarray*}
\end{theorem}

\begin{proof}
First let us consider the optimization problem
\begin{eqnarray}\label{n6}
\left[ \begin{array}{c} \widetilde {\bf y} \\ \widetilde y_{k+1} \end{array}\right]=\arg
 \min_{{\bf y}\geq{\bf 0},y_{k+1}\in \mathbb{R}} \| \left[ \begin{array}{cc} G & {\bf g}_{k+1} \end{array}\right]
                                              \left[ \begin{array}{c} {\bf y} \\ y_{k+1}\end{array}\right]-{\bf b}\|.
\end{eqnarray}
For any given ${\bf y}$, the solution $y_{k+1}$ to the optimization problem
\[  \min_{y_{k+1}\in \mathbb{R}} \| {\bf g}_{k+1} y_{k+1} -({\bf b}-G {\bf y})\| \]
is uniquely given by
\begin{equation}\label{N1} y_{k+1}=\frac{{\bf g}^T_{k+1}({\bf b}-G{\bf y})}{\|{\bf g}_{k+1}\|^2}.
\end{equation}
Accordingly, the optimization problem \eqref{n6} can be reduced to
\begin{eqnarray}\label{n7}
\widetilde {\bf y} =\arg
\min_{{\bf y}\geq{\bf 0}} \|(G{\bf y}+{\bf g}_{k+1}\frac{{\bf g}^T_{k+1}({\bf b}-G{\bf y})}{\|{\bf g}_{k+1}\|^2})-{\bf b}\|
      =\arg \min_{{\bf y}\geq{\bf 0}}\|\widetilde{G}{\bf y}-{\bf \widetilde{b}}\|,
\end{eqnarray}
where
\[ {\bf \widetilde{b}}={\bf b}-\frac{{\bf g}_{k+1}{\bf g}^T_{k+1}}{\|{\bf g}_{k+1}\|^2}{\bf b}, \quad
 \widetilde{G}=G-\frac{{\bf g}_{k+1}{\bf g}^T_{k+1}}{\|{\bf g}_{k+1}\|^2}G.
\]
Moreover, it holds that
\[ k+1=\emph{rank}(\left[ \begin{array}{cc} G & {\bf g}_{k+1} \end{array}\right])
   =\emph{rank} ({\bf g}_{k+1}) +\emph{rank}((I-\frac{{\bf g}_{k+1}{\bf g}^T_{k+1}}{\|{\bf g}_{k+1}\|^2})G)
   =1+\emph{rank}(\widetilde{G}), \]
i.e., $\emph{rank}(\widetilde{G})=k$ and $\widetilde{G}$ is of full column rank.
Note that according to our notation we have
\[ s(\widetilde{G}, {\bf \widetilde{b}})=\widetilde {\bf y}=\arg \min_{{\bf y}\geq{\bf 0}}\|\widetilde{G}{\bf y}-{\bf \widetilde{b}}\|, \]
thus, it follows from (\ref{N1}) that
\begin{equation}\label{N2}
\widetilde y_{k+1}=\frac{ {\bf g}^T_{k+1} ({\bf b}-G \widetilde {\bf y})} {\|{\bf g}_{k+1}\|^2}
   =\frac{{\bf g}^T_{k+1}({\bf b}-G\cdot s(\widetilde{G}, {\bf \widetilde{b}}))}{\|{\bf g}_{k+1}\|^2}.
\end{equation}
Consequently, the optimization problem \eqref{n6} becomes
\[  \widetilde {\bf y}=\arg
 \min_{{\bf y}\geq{\bf 0}} \|  G {\bf y}-({\bf b}-{\bf g}_{k+1} \widetilde y_{k+1})\|, \]
which, according to our notation again,  can be rewritten as
\begin{equation}\label{N3} \widetilde {\bf y} = s(G, {\bf b}-{\bf g}_{k+1} \widetilde y_{k+1}).
\end{equation}
Let
\[ \mathcal{T}:=\{ \left[ \begin{array}{c} {\bf y} \\ y_{k+1} \end{array}\right] |~{\bf y}\in \mathbb{R}^k, ~{\bf y}\geq0\}, \qquad
\mathcal{C}:=\{ \left[ \begin{array}{c} {\bf y} \\ y_{k+1} \end{array}\right] |~{\bf y}\in \mathbb{R}^k, ~{\bf y}\geq0, ~y_{k+1}\geq 0\}. \]
By Lemma \ref{lem0}, (\ref{N2}) and (\ref{N3}),
\begin{itemize}
\item if $\widetilde y_{k+1}\geq0$, then $\left[ \begin{array}{c} {\bf y}^{\star} \\  y_{k+1}^{\star}\end{array}\right]$ with
\begin{eqnarray*}
y^*_{k+1} &=& \widetilde y_{k+1} = \frac{{\bf g}^T_{k+1}({\bf b}-G\cdot s(\widetilde{G}, {\bf \widetilde{b}}))}{\|{\bf g}_{k+1}\|^2}
=\frac{[{\bf g}^T_{k+1}({\bf b}-G\cdot s(\widetilde{G}, {\bf \widetilde{b}}))]_+}{\|{\bf g}_{k+1}\|^2}, \\
{\bf y}^{\star} &=& \widetilde {\bf y} = s(G, {\bf b}-{\bf g}_{k+1} \widetilde y_{k+1})
   =s(G, {\bf b}-{\bf g}_{k+1} y^*_{k+1}),
\end{eqnarray*}
is the unique solution of the problem \eqref{n5};
\item if $\widetilde y_{k+1}<0$, then the unique solution $y_{k+1}$ of the problem  \eqref{n5} is
$$ y^{\star}_{k+1}=0=[\widetilde y_{k+1}]_{+}=
\frac{[{\bf g}^T_{k+1}({\bf b}-G\cdot s(\widetilde{G}, {\bf \widetilde{b}}))]_+}{\|{\bf g}_{k+1}\|^2}, $$
and consequently, the unique solution ${\bf y}$ of the problem \eqref{n5} is given by
\[ {\bf y}^{\star}=\arg \min_{{\bf y}\geq 0}\| G{\bf y}-{\bf b}\|=\arg \min_{{\bf y}\geq 0}\| G{\bf y}-({\bf b}-{\bf g}_{k+1}y^{\star}_{k+1}) \|, \]
which yields that
\[ {\bf y}^{\star}=s(G, {\bf b}-{\bf g}_{k+1}y^*_{k+1}). \]
\end{itemize}

\end{proof}

Theorem \ref{thm1} can be used to derive the closed-form solution of the \emph{rank-k NLS} problem (\ref{n4}) for any integer $k\geq 1$.
We derive the closed-form solutions
of the \emph{rank-2 NLS} and \emph{rank-3 NLS} (without recursion) in the following two corollaries.

\begin{corollary}\label{cor1}
Assume that $G=\left[ \begin{array}{cc} {\bf g}_1 & {\bf g}_2\end{array}\right]\in \mathbb{R}^{m\times 2}$ and $rank(G)=2$.
Then the unique solution of the  \emph{rank-2 NLS} problem
\begin{align}\label{n10}
\left[ \begin{array}{c} y_1^* \\  y_2^*\end{array}\right]=\arg \min_{{\bf y}\geq0}\|G{\bf y}-{\bf b}\|=
{\rm arg} \min_{\{y_1,y_2\}\geq0}\|y_1{\bf g}_1+y_2{\bf g}_2-{\bf b}\|
\end{align}
is given by
\begin{eqnarray}\label{n9}
\begin{cases}
y_2^*&=\frac{1}{\|{\bf g}_2\|^2}\left[{\bf b}^T{\bf g}_2-{\bf g}^T_2{\bf g}_1
       \left[\frac{\|{\bf g}_2\|^2{\bf b}^T{\bf g}_1-{\bf b}^T{\bf g}_2\cdot{\bf g}^T_2{\bf g}_1}{ \|{\bf g}_1\|^2\|{\bf g}_2\|^2
        -({\bf g}_1^T{\bf g}_2)^2  } \right]_{+}\right]_{+}\\
y_1^*&=\frac{1}{\|{\bf g}_1\|^2}[{\bf b}^T{\bf g}_1-({\bf g}^T_2{\bf g}_1)y_2^*]_{+}.
\end{cases}
\end{eqnarray}
\end{corollary}

\begin{proof}
Since  the solution of the \emph{rank-1 NLS} problem
$ \min_{y_1\geq0}\|y_1{\bf g}_1-{\bf b}\| $
is given by
$s({\bf g}_1, {\bf b})= \frac{[{\bf g}^T_1{\bf b}]_+}{\|{\bf g}_1\|^2}$,
according to Theorem \ref{thm1}, the solution $y_2^*$ to \emph{rank-2 NLS} problem \eqref{n10} is
$$y^*_2=\frac{\big[{\bf g}^T_{2}({\bf b}-{\bf g}_1\cdot
s({\bf g}_1-\frac{{\bf g}_{2}{\bf g}^T_{2}}{\|{\bf g}_{2}\|^2}{\bf g}_1, {\bf b}-\frac{{\bf g}_{2}{\bf g}^T_{2}}{\|{\bf g}_{2}\|^2}{\bf b})\big]_+}{\|{\bf g}_{2}\|^2}.$$
Since
\begin{eqnarray*}
s({\bf g}_1-\frac{{\bf g}_{2}{\bf g}^T_{2}}{\|{\bf g}_{2}\|^2}{\bf g}_1, {\bf b}-\frac{{\bf g}_{2}{\bf g}^T_{2}}{\|{\bf g}_{2}\|^2}{\bf b})
&=& \frac{\big[\big({\bf g}_1-{\bf g}_{2}\frac{{\bf g}^T_{2}{\bf g}_1}{\|{\bf g}_{2}\|^2}\big)^T\big({\bf b}-{\bf g}_{2}\frac{{\bf g}^T_{2}{\bf b}}{\|
{\bf g}_{2}\|^2}\big)\big]_+}{\left\|{\bf g}_1-{\bf g}_{2}\frac{{\bf g}^T_{2}{\bf g}_1}{\|{\bf g}_{2}\|^2}\right\|^2} \\
&=& \frac{\left[\|{\bf g}_2\|^2{\bf b}^T{\bf g}_1-{\bf b}^T{\bf g}_2\cdot{\bf g}^T_2{\bf g}_1\right]_{+}}
{  \|{\bf g}_1\|^2\|{\bf g}_2\|^2-({\bf g}_1^T{\bf g}_2)^2 },
\end{eqnarray*}
%then
%\begin{eqnarray*}
%\begin{cases}
%y^*_2&=\frac{1}{\|{\bf g}_2\|^2}\left[{\bf b}^T{\bf g}_2-{\bf g}^T_2{\bf g}_1\left[\frac{\|{\bf g}_2\|^2{\bf b}^T{\bf g}_1 -{\bf b}^T{\bf g}_2\cdot{\bf g}^T_2{\bf g}_1}{ \|{\bf g}_1\|^2\|{\bf g}_2\|^2-({\bf g}_1^T{\bf g}_2)^2  }\right]_{+}\right]_{+}\\
%y^*_1&=\frac{1}{\|{\bf g}_1\|^2}[{\bf b}^T{\bf g}_1-({\bf g}^T_2{\bf g}_1)y_2^*]_{+},
%\end{cases}
%\end{eqnarray*}
(\ref{n9}) holds. \end{proof}

\begin{corollary}\label{cor2}
Assume that $G=\left[ \begin{array}{ccc} {\bf g}_1 & {\bf g}_2 &  {\bf g_3} \end{array}\right]\in \mathbb{R}^{m\times 3}$
and $rank(G)=3$. Then the unique solution of the \emph{rank-3 NLS} problem
\begin{align*}
  \left[ \begin{array}{c} y^*_1 \\ y^*_2 \\ y^*_3\end{array}\right]={\rm arg}
  \min_{{\bf y}\geq0}\|G{\bf y}-{\bf b}\|
 = {\rm arg} \min_{\{y_1,y_2,y_3\} \geq0}\|y_1{\bf g}_1+y_2{\bf g}_2+y_3{\bf g}_3-{\bf b}\|
\end{align*}
is given by \begin{align}\label{r3}
\begin{cases}
y^*_3&=\frac{1}{\|{\bf g}_3\|^2}[{\bf b}^T{\bf g}_3-({\bf g}^T_3{\bf g}_2)p - ({\bf g}^T_3{\bf g}_1)\widetilde{p}]_{+}\\
y^*_2&=\frac{1}{\|{\bf g}_2\|^2}\left[{\bf b}^T{\bf g}_2-({\bf g}_3^T{\bf g}_2)y^*_3-{\bf g}^T_2{\bf g}_1\big[\frac{({\bf b}^T{\bf g}_1\|
                  {\bf g}_2\|^2-{\bf b}^T{\bf g}_2\cdot{\bf g}^T_2{\bf g}_1)- ({\bf g}^T_3{\bf g}_1\|{\bf g}_2\|^2
                   -{\bf g}^T_3{\bf g}_2\cdot{\bf g}^T_2{\bf g}_1)y^*_3}{\|{\bf g}_1\|^2\|{\bf g}_2\|^2-({\bf g}^T_2
                   {\bf g}_1)^2}\big]_{+}\right]_{+}\\
y^*_1&=\frac{1}{\|{\bf g}_1\|^2}[{\bf b}^T{\bf g}_1-({\bf g}^T_3{\bf g}_1)y^*_3-({\bf g}^T_2{\bf g}_1)y^*_2]_{+},
\end{cases}
\end{align}
where
\begin{eqnarray*}\label{q1}
p=\left[\frac{{\bf b}^T{\bf g}_2\cdot\|{\bf g}_3\|^2-{\bf b}^T{\bf g}_3\cdot{\bf g}^T_3{\bf g}_2}{\|{\bf g}_2\|^2\|{\bf g}_3\|^2
         -({\bf g}^T_3{\bf g}_2)^2}-\frac{{\bf g}^T_2{\bf g}_1\cdot\|{\bf g}_3\|^2-{\bf g}^T_3{\bf g}_2\cdot{\bf g}^T_3{\bf g}_1}{\|
          {\bf g}_2\|^2\|{\bf g}_3\|^2-({\bf g}^T_3{\bf g}_2)^2}\big[\frac{{\rm det}([{\bf b},{\bf g}_2,{\bf g}_3]^TG)}
           {{\rm det}(G^TG)}\big]_{+}\right]_{+},
\end{eqnarray*}
and
\begin{eqnarray*}\label{q2}
\widetilde{p}=\left[\frac{{\bf b}^T{\bf g}_1\cdot\|{\bf g}_3\|^2
                      -{\bf b}^T{\bf g}_3\cdot{\bf g}^T_3{\bf g}_1}{\|{\bf g}_1\|^2\|{\bf g}_3\|^2-({\bf g}^T_3{\bf g}_1)^2}
                      -\frac{{\bf g}^T_2{\bf g}_1\cdot\|{\bf g}_3\|^2-{\bf g}^T_3{\bf g}_2\cdot{\bf g}^T_3{\bf g}_1}{\|{\bf g}_1\|^2\|
                      {\bf g}_3\|^2-({\bf g}^T_3{\bf g}_1)^2}p\right]_{+}.
\end{eqnarray*}
\end{corollary}

\begin{proof}
Let
$${\bf \widehat{b}}={\bf b}-\frac{{\bf g}_{3}{\bf g}^T_{3}}{\|{\bf g}_{3}\|^2}{\bf b},~~
{\bf \widehat{g}}_1={\bf g}_{1}-\frac{{\bf g}_{3}{\bf g}^T_{3}}{\|{\bf g}_{3}\|^2}{\bf g}_{1},~~
{\bf \widehat{g}}_2={\bf g}_{2}-\frac{{\bf g}_{3}{\bf g}^T_{3}}{\|{\bf g}_{3}\|^2}{\bf g}_{2},~~
 $$
A simple calculation yields that
\begin{align}
{\bf \widehat{g}}^T_2{\bf \widehat{g}}_1
&=({\bf g}_{2}-\frac{{\bf g}^T_{3}{\bf g}_{2}}{\|{\bf g}_{3}\|^2}{\bf g}_{3})^T({\bf g}_{1}
              -\frac{{\bf g}^T_{3}{\bf g}_{1}}{\|{\bf g}_{3}\|^2}{\bf g}_{3})
=\frac{({\bf g}^T_2{\bf g}_1)\|{\bf g}_3\|^2-({\bf g}^T_3{\bf g}_2)({\bf g}^T_3{\bf g}_1)}{\|{\bf g}_{3}\|^2},\label{n18}\\
\|{\bf \widehat{g}}_1\|^2
&=({\bf g}_{1}-\frac{{\bf g}^T_{3}{\bf g}_{1}}{\|{\bf g}_{3}\|^2}{\bf g}_{3})^T({\bf g}_{1}
              -\frac{{\bf g}^T_{3}{\bf g}_{1}}{\|{\bf g}_{3}\|^2}{\bf g}_{3})
=\frac{\|{\bf g}_1\|^2\|{\bf g}_3\|^2-({\bf g}^T_3{\bf g}_1)^2}{\|{\bf g}_{3}\|^2},\label{n21}\\
\|{\bf \widehat{g}}_2\|^2
&=({\bf g}_{2}-\frac{{\bf g}^T_{3}{\bf g}_{2}}{\|{\bf g}_{3}\|^2}{\bf g}_{3})^T({\bf g}_{2}
              -\frac{{\bf g}^T_{3}{\bf g}_{2}}{\|{\bf g}_{3}\|^2}{\bf g}_{3})
=\frac{\|{\bf g}_2\|^2\|{\bf g}_3\|^2-({\bf g}^T_3{\bf g}_2)^2}{\|{\bf g}_{3}\|^2},\label{n16}\\
{\bf \widehat{b}}^T{\bf \widehat{g}}_2
&=({\bf b}-\frac{{\bf g}^T_{3}{\bf b}}{\|{\bf g}_{3}\|^2}{\bf g}_{3})^T({\bf g}_{2}
         -\frac{{\bf g}^T_{3}{\bf g}_{2}}{\|{\bf g}_{3}\|^2}{\bf g}_{3})
=\frac{({\bf b}^T{\bf g}_2)\|{\bf g}_3\|^2-({\bf g}^T_3{\bf g}_2)({\bf g}^T_3{\bf b})}{\|{\bf g}_{3}\|^2},\label{n17}\\
{\bf \widehat{b}}^T{\bf \widehat{g}}_1
&=({\bf b}-\frac{{\bf g}^T_{3}{\bf b}}{\|{\bf g}_{3}\|^2}{\bf g}_{3})^T({\bf g}_{1}
          -\frac{{\bf g}^T_{3}{\bf g}_{1}}{\|{\bf g}_{3}\|^2}{\bf g}_{3})
=\frac{({\bf b}^T{\bf g}_1)\|{\bf g}_3\|^2-({\bf g}^T_3{\bf g}_1)({\bf g}^T_3{\bf b})}{\|{\bf g}_{3}\|^2}.\label{n19}
\end{align}
Let $s(\left[ \begin{array}{cc} {\bf {\widehat g}}_1 & {\bf {\widehat g}}_2 \end{array}\right], {\bf {\widehat b}})
     =\left[
     \begin{array}{c}
          s_1(\left[ \begin{array}{cc} {\bf {\widehat g}}_1 & {\bf {\widehat g}}_2 \end{array}\right], {\bf {\widehat b}}) \\
          s_2(\left[ \begin{array}{cc} {\bf {\widehat g}}_1 & {\bf {\widehat g}}_2 \end{array}\right], {\bf {\widehat b}}) \\
     \end{array}
     \right]$
be the solution to the optimization problem
\[ \min_{y_1,y_2\geq0}\|y_1{\bf {\widehat g}}_1+y_2{\bf {\widehat g}}_2-{\bf {\widehat b}}\|. \]
Then from \eqref{n18}--\eqref{n19} and Corollary \ref{cor1}, we have
\begin{align*}
s_2(\left[ \begin{array}{cc} {\bf \widehat{g}}_1 & {\bf \widehat{g}}_2\end{array}\right], {\bf \widehat{b}})
&=\frac{1}{\|{\bf \widehat{g}}_2\|^2}\left[{\bf \widehat{b}}^T{\bf \widehat{g}}_2
             -{\bf \widehat{g}}^T_2{\bf \widehat{g}}_1\left[\frac{\|{\bf \widehat{g}}_2\|^2{\bf \widehat{b}}^T{\bf \widehat{g}}_1
             -{\bf \widehat{b}}^T{\bf \widehat{g}}_2\cdot{\bf \widehat{g}}^T_2{\bf \widehat{g}}_1}{\|{\bf \widehat{g}}_1\|^2\|{\bf \widehat{g}}_2\|^2-({\bf \widehat{g}}_1^T{\bf \widehat{g}}_2)^2}\right]_{+}\right]_{+}\nonumber\\
&=\left[\frac{{\bf b}^T{\bf g}_2\cdot\|{\bf g}_3\|^2-{\bf b}^T{\bf g}_3\cdot{\bf g}^T_3{\bf g}_2}{\|{\bf g}_2\|^2\|{\bf g}_3\|^2
             -({\bf g}^T_3{\bf g}_2)^2}-\frac{{\bf g}^T_2{\bf g}_1\cdot\|{\bf g}_3\|^2-{\bf g}^T_3{\bf g}_2\cdot{\bf g}^T_3{\bf g}_1}
              {\|{\bf g}_2\|^2\|{\bf g}_3\|^2-({\bf g}^T_3{\bf g}_2)^2}\big[\frac{{\rm det}([{\bf b},{\bf g}_2,{\bf g}_3]^TG)}
               {{\rm det}(G^TG)}\big]_{+}\right]_{+} \nonumber \\
&=p.
\end{align*}
Moreover,
\begin{eqnarray*}
s_1(\left[ \begin{array}{cc} {\bf \widehat{g}}_1 & {\bf \widehat{g}}_2\end{array}\right], {\bf \widehat{b}})
&=& \left[\frac{{\bf b}^T{\bf g}_1\cdot\|{\bf g}_3\|^2-{\bf b}^T{\bf g}_3\cdot{\bf g}^T_3{\bf g}_1}{\|{\bf g}_1\|^2\|{\bf g}_3\|^2
            -({\bf g}^T_3{\bf g}_1)^2}-\frac{{\bf g}^T_2{\bf g}_1\cdot\|{\bf g}_3\|^2-{\bf g}^T_3{\bf g}_2\cdot{\bf g}^T_3{\bf g}_1}
               {\|{\bf g}_1\|^2\|{\bf g}_3\|^2-({\bf g}^T_3{\bf g}_1)^2}
s_2(\left[ \begin{array}{cc} {\bf \widehat{g}}_1 & {\bf \widehat{g}}_2\end{array}\right],
{\bf \widehat{b}})\right]_{+}\\
&=&\widetilde p.
\end{eqnarray*}
Therefore, according to Theorem \ref{thm1} and Corollary \ref{cor1}, the solution to the optimization problem
\[ \min_{\{y_1,y_2,y_3\}\geq0}\|y_1{\bf g}_1+y_2{\bf g}_2+y_3{\bf g}_3-{\bf b}\| \]
is
\begin{eqnarray*}
y^*_3 &=& \frac{1}{\|{\bf g}_{3}\|^2}\big[{\bf g}^T_{3}\big({\bf b}-[{\bf g}_1,{\bf g}_2]\cdot
s( \left[ \begin{array}{cc} {\bf \widehat{g}}_1 & {\bf \widehat{g}}_2\end{array}\right], {\bf \widehat{b}})\big)\big]_+ \\
%&=&\frac{1}{\|{\bf g}_{3}\|^2}\big[{\bf g}^T_{3}{\bf b}-{\bf g}^T_{3}{\bf g}_1\cdot
%s_1(\left[ \begin{array}{cc} {\bf \widehat{g}}_1 & {\bf \widehat{g}}_2\end{array}\right], {\bf \widehat{b}})
%-{\bf g}^T_{3}{\bf g}_2\cdot
%s_2(\left[ \begin{array}{cc} {\bf \widehat{g}}_1 & {\bf \widehat{g}}_2\end{array}\right], {\bf \widehat{b}})\big]_+,\\
&=&\frac{1}{\|{\bf g}_{3}\|^2}\big[{\bf g}^T_{3}{\bf b}-{\bf g}^T_{3}{\bf g}_1\cdot \widetilde{p}
-{\bf g}^T_{3}{\bf g}_2\cdot p\big]_+,
\end{eqnarray*}
with $s_1(\left[ \begin{array}{cc} {\bf \widehat{g}}_1 & {\bf \widehat{g}}_2\end{array}\right], {\bf \widehat{b}})=\widetilde{p}$, $s_2(\left[ \begin{array}{cc} {\bf \widehat{g}}_1 & {\bf \widehat{g}}_2\end{array}\right], {\bf \widehat{b}})=p$,
and
\begin{align*}
\begin{cases}
y^*_2&=s_2( \left[ \begin{array}{cc} {\bf g}_1 & {\bf g}_2 \end{array}\right],  {\bf b}-{\bf g}_{3}y^*_{3})\\
~&=\frac{1}{\|{\bf g}_2\|^2}\left[({\bf b}-{\bf g}_{3}y^*_{3})^T{\bf g}_2-{\bf g}^T_2{\bf g}_1\left[\frac{\|{\bf g}_2\|^2({\bf b}
       -{\bf g}_{3}y^*_{3})^T{\bf g}_1-({\bf b}-{\bf g}_{3}y^*_{3})^T{\bf g}_2\cdot{\bf g}^T_2{\bf g}_1}{\|{\bf g}_1\|^2\|{\bf g}_2\|^2
       -({\bf g}^T_2{\bf g}_1)^2}\right]_{+}\right]_{+}\\
~&=\frac{1}{\|{\bf g}_2\|^2}\left[({\bf b}-{\bf g}_3y^*_3)^T{\bf g}_2-{\bf g}^T_2{\bf g}_1\big[\frac{({\bf b}^T{\bf g}_1\|{\bf g}_2\|^2
       -{\bf b}^T{\bf g}_2\cdot{\bf g}^T_2{\bf g}_1)- ({\bf g}^T_3{\bf g}_1\|{\bf g}_2\|^2
       -{\bf g}^T_3{\bf g}_2\cdot{\bf g}^T_2{\bf g}_1)y^*_3}{\|{\bf g}_1\|^2\|{\bf g}_2\|^2
       -({\bf g}^T_2{\bf g}_1)^2}\big]_{+}\right]_{+}\\
y^*_1&=s_1(\left[ \begin{array}{cc} {\bf g}_1 & {\bf g}_2\end{array}\right], {\bf b}-{\bf g}_{3}y^*_{3})
=\frac{1}{\|{\bf g}_1\|^2}[({\bf b}-{\bf g}_{3}y^*_{3})^T{\bf g}_1-({\bf g}^T_2{\bf g}_1)y^*_2]_{+}.
\end{cases}
\end{align*}
Hence, Corollary \ref{cor2} is proved.
\end{proof}

 %Section 3

\section{ARkNLS with $k=3$}\label{sec4}

Based on Theorem  \ref{thm1-coro} and Colloraries \ref{cor1} and \ref{cor2}, ARkNLS with $k=2$ and $k=3$ can produce practical
numerical methods for NMF. Although there is a closed-form solution for the \emph{rank-k NLS} for any $k$ including when $k\geq 4$, expansion of recursion to obtain a closed-form solution itself gets very complicated and the closed form solution becomes computationally messy. 
Therefore, we will focus on ARkNLS with $k=3$ in the rest of the paper since $k=2$ case can be easily derived in a similar way.
In addition, we will propose methods for handling the possible singularity problem for $k=1, 2$, and $3$.

In the following the closed-form solution of the \emph{rank-k NLS} problem (\ref{V}) with $k=3$ is derived first,
and then a strategy for avoiding rank deficient rank-k NLS in NMF iteration is provided. Finally these closed-form solutions 
and the proposed strategy lead to the efficient algorithm ARkNLS(k=3).

\subsection{The Closed-Form Solution of the \emph{rank-k NLS} Problem (\ref{V}) with $k=3$}

In this subsection, we derive an efficient algorithm to solve
the \emph{rank-k NLS} problem (\ref{V}) (solving the problem (\ref{U})
will be analogous) with $k=3$.
The efficient algorithm is derived generalizing the result
presented in Corollary \ref{cor2} to the case of NLS with multiple
right hand side vectors which are
$(A-\Sigma_{l\neq i}U_lV_l^T)$ in problem (\ref{V}) (and $(A-\Sigma_{l\neq i}U_lV_l^T)^T $in problem (\ref{U})), and exploiting
the special structure of this multiple right hand side vectors,
so that redundant computations are identified and avoided.

Assume $q=r/3$ is an integer and partition $U$ and $V$ into $q$ blocks as
follows:
\[   U= \left[ \begin{array}{ccc} U_1 & \cdots & U_q \end{array}\right], \qquad
                  V= \left[ \begin{array}{ccc} V_1 & \cdots & V_q \end{array}\right], \qquad
                   U_1, \cdots, U_q \in \mathbb{R}^{m\times 3}, \quad V_1, \cdots, V_q\in \mathbb{R}^{n\times 3}.
\]
For notational simplicity, we denote the three columns of the $i$th blocks $U_i$ and $V_i$ as follows, respectively,  without additional subscripts that corresponds to the columns of $U$ and $V$,
\[ U_i=\left[ \begin{array}{ccc} {\bf u}_1 & {\bf u}_2 & {\bf u}_{3}\end{array}\right] \quad \mbox{and} 
   \quad V_i=\left[ \begin{array}{ccc} {\bf v}_1 & {\bf v}_2 & {\bf v}_{3}\end{array}\right].
\]

\begin{theorem}\label{thm3}
Assume $U_i \in \mathbb{R}^{m \times 3}$ and $rank(U_i) = 3$.
Then the unique solution of the \emph{rank-3 NLS} problem
\begin{align}\label{Vi}
 V_i^*=\left[ \begin{array}{ccc} {\bf v}^*_1 & {\bf v}^*_2 & {\bf v}^*_{3}\end{array}\right]
 =\arg \min_{V_i\in R^{n\times 3},V_i\geq{\bf 0}}\|U_iV_i^T-(A-\Sigma_{l\neq i}U_lV_l^T)\|_F^2
\end{align}
is given by
\begin{align}\label{t2}
\begin{cases}
{\bf v}^*_{3}&=\left[{\bf v}_{3}+\frac{{\bf r}_{3}}{\|{\bf u}_{3}\|^2}+\frac{{\bf u}^T_{3}{\bf u}_{1}}{\|{\bf u}_{3}\|^2}({\bf v}_{1}
                  -{\bf \widetilde{p}})+\frac{{\bf u}^T_{3}{\bf u}_{2}}{\|{\bf u}_{3}\|^2}({\bf v}_{2}-{\bf p})\right]_{+},\\
{\bf v}^*_{2}&=\left[{\bf v}_{2}+\frac{{\bf r}_{2}}{\|{\bf u}_{2}\|^2}+\frac{{\bf u}^T_{2}{\bf u}_{1}}{\|{\bf u}_{2}\|^2}({\bf v}_{1}-{\bf z})
                  +\frac{{\bf u}^T_{3}{\bf u}_{2}}{\|{\bf u}_{2}\|^2}({\bf v}_{3}-{\bf v}^*_{3})\right]_{+},\\
{\bf v}^*_{1}&=\left[{\bf v}_{1}+\frac{{\bf r}_{1}}{\|{\bf u}_{1}\|^2}+\frac{{\bf u}^T_{2}{\bf u}_{1}}{\|{\bf u}_{1}\|^2}({\bf v}_{2}
                 -{\bf v}^*_{2})+\frac{{\bf u}^T_{3}{\bf u}_{1}}{\|{\bf u}_{1}\|^2}({\bf v}_{3}-{\bf v}^*_{3})\right]_{+},
\end{cases}
\end{align}
where
\begin{eqnarray}
&& \left[ \begin{array}{ccc} {\bf r}_1 & {\bf r}_2 & {\bf r}_{3}\end{array}\right] =  A^T U_i - V U^T U_i, \label{q32}\\
&& a={\bf u}^T_{2}{\bf u}_{1}\cdot\|{\bf u}_{3}\|^2-{\bf u}^T_{3}{\bf u}_{2}\cdot{\bf u}^T_{3}{\bf u}_{1},\label{q18}\\
&& b={\bf u}^T_{3}{\bf u}_{1}\cdot\|{\bf u}_{2}\|^2-{\bf u}^T_{3}{\bf u}_{2}\cdot{\bf u}^T_{2}{\bf u}_{1},\label{q19}\\
&& d_{12}=\|{\bf u}_{1}\|^2\|{\bf u}_{2}\|^2-({\bf u}^T_{2}{\bf u}_{1})^2,\label{q91}\\
&& d_{13}=\|{\bf u}_{1}\|^2\|{\bf u}_{3}\|^2-({\bf u}^T_{3}{\bf u}_{1})^2,\label{q37}\\
&& d_{23}=\|{\bf u}_{2}\|^2\|{\bf u}_{3}\|^2-({\bf u}^T_{3}{\bf u}_{2})^2,\label{q17}\\
&& {\bf p}=\left[{\bf v}_{2}+\frac{\|{\bf u}_{3}\|^2 {\bf r}_{2}-{\bf u}^T_{3}{\bf u}_{2}\cdot {\bf r}_{3}}{d_{23}}+\frac{a}{d_{23}}\left({\bf v}_{1}-\big[\frac{d_{23} {\bf r} _{1}-a {\bf r} _{2}-b {\bf r}_{3}}{{\rm det}( U_i^TU_i)}+{\bf v}_{1}\big]_{+}\right)\right]_{+},\label{q40}\\
&& {\bf \widetilde{p}}=\left[{\bf v}_{1}+\frac{\|{\bf u}_{3}\|^2 {\bf r}_{1}-{\bf u}^T_{3}{\bf u}_{1}\cdot {\bf r}_{3}}{d_{13}}
+\frac{a}{d_{13}}({\bf v}_{2}-{\bf p})\right]_{+},\label{q101}\\
&& {\bf z}=\left[{\bf v}_{1}+\frac{\|{\bf u}_{2}\|^2 {\bf r}_{1}-{\bf u}_{2}^T{\bf u}_{1}\cdot {\bf r}_{2}}{d_{12}}+\frac{b}{d_{12}}({\bf v}_{3}
                  -{\bf v}^*_{3})\right]_{+}.\label{n30}
\end{eqnarray}
\end{theorem}

\begin{proof} Recall that
\[   A-\Sigma_{l\neq i}U_lV_l^T=A-UV^T+U_{i}V_{i}^T. \]
Then we have
\begin{eqnarray}\label{c02}
(A-UV^T+U_{i}V_i^T)^T{\bf u}_{1}\nonumber\\
&=&A^T{\bf u}_{1}-VU^T{\bf u}_{1}+\|{\bf u}_{1}\|^2{\bf v}_{1}+{\bf u}^T_{2}{\bf u}_{1}\cdot{\bf v}_{2}
           +{\bf u}^T_{3}{\bf u}_{1}\cdot{\bf v}_{3}\nonumber\\
&=&{\bf r}_{1}+\|{\bf u}_{1}\|^2{\bf v}_{1}+{\bf u}^T_{2}{\bf u}_{1}\cdot{\bf v}_{2}+{\bf u}^T_{3}{\bf u}_{1}\cdot{\bf v}_{3},\\
%\end{eqnarray}
%where ${\bf r}_1$ is defined in Eqn. \eqref{q32}.
%Similarly,
%\begin{eqnarray}\label{c1}
(A-UV^T+U_{i}V_{i}^T)^T{\bf u}_{2} &=&{\bf r}_{2}+{\bf u}^T_{2}{\bf u}_{1}\cdot{\bf v}_{1}+\|{\bf u}_{2}\|^2{\bf v}_{2}
            +{\bf u}^T_{3}{\bf u}_{2}\cdot{\bf v}_{3},\\
(A-UV^T+U_{i}V_{i}^T)^T{\bf u}_{3}&=&{\bf r}_{3}+{\bf u}^T_{3}{\bf u}_{1}\cdot{\bf v}_{1}+{\bf u}^T_{3}{\bf u}_{2}\cdot{\bf v}_{2}
            +\|{\bf u}_{3}\|^2{\bf v}_{3},
\end{eqnarray}
where ${\bf r}_i$, for $i=1, 2, 3$, are defined in Eqn. \eqref{q32}.
In addition,
\begin{eqnarray}\label{c04}
(A-UV^T\hspace{-0.2in}&&+U_{i}V_{i}^T)^T{\bf u}_{2}\cdot\|{\bf u}_{3}\|^2-(A-UV^T+U_{i}V_{i}^T)^T{\bf u}_{3}\cdot{\bf u}^T_{3}{\bf u}_{2} \nonumber \\
&=&\|{\bf u}_{3}\|^2{\bf r}_{2}-{\bf u}^T_{3}{\bf u}_{2}{\bf r}_{3}+{\bf u}^T_{2}{\bf u}_{1}\|{\bf u}_{3}\|^2\cdot{\bf v}_{1}
       +\|{\bf u}_{2}\|^2\|{\bf u}_{3}\|^2\cdot{\bf v}_{2}+{\bf u}^T_{3}{\bf u}_{2}\|{\bf u}_{3}\|^2\cdot{\bf v}_{3} \nonumber \\
& & \qquad \qquad -{\bf u}^T_{3}{\bf u}_{1}\cdot{\bf u}^T_{3}{\bf u}_{2}\cdot{\bf v}_{1}-({\bf u}^T_{3}{\bf u}_{2})^2{\bf v}_{2}
       -\|{\bf u}_{3}\|^2{\bf u}^T_{3}{\bf u}_{2}\cdot{\bf v}_{3}    \nonumber\\
&=&\|{\bf u}_{3}\|^2{\bf r}_{2}-{\bf u}^T_{3}{\bf u}_{2}{\bf r}_{3}
       +({\bf u}^T_{2}{\bf u}_{1}\|{\bf u}_{3}\|^2-{\bf u}^T_{3}{\bf u}_{1}\cdot{\bf u}^T_{3}{\bf u}_{2}){\bf v}_{1}
       +(\|{\bf u}_{2}\|^2\|{\bf u}_{3}\|^2-({\bf u}^T_{3}{\bf u}_{2})^2){\bf v}_{2}\nonumber\\
&=&  d_{23}{\bf v}_{2}+\|{\bf u}_{3}\|^2{\bf r}_{2}-{\bf u}^T_{3}{\bf u}_{2}{\bf r}_{3}+a{\bf v}_{1},\\
%\end{eqnarray}
%and analogously,
%\begin{eqnarray}
(A-UV^T\hspace{-0.2in}&&+U_{i}V_{i}^T)^T{\bf u}_{1}\cdot\|{\bf u}_{3}\|^2-(A-UV^T+U_{i}V_{i}^T)^T{\bf u}_{3}\cdot{\bf u}^T_{3}{\bf u}_{1} \nonumber \\
&=& d_{13}{\bf v}_{1}+\|{\bf u}_{3}\|^2{\bf r}_{1}-{\bf u}^T_{3}{\bf u}_{1}{\bf r}_{3}+a{\bf v}_{2},\label{c05}\\
(A-UV^T\hspace{-0.2in}&&+U_{i}V_{i}^T)^T{\bf u}_{1}\|{\bf u}_{2}\|^2-(A-UV^T+U_{i}V_{i}^T)^T{\bf u}_{2}\cdot{\bf u}^T_{2}{\bf u}_{1} \nonumber \\
&=& d_{12}{\bf v}_{1}+\|{\bf u}_{2}\|^2{\bf r}_{1}-{\bf u}^T_{2}{\bf u}_{1}{\bf r}_{2}+b {\bf v}_{3},\label{c06}\\
(A-UV^T\hspace{-0.2in}&&+U_{i}V_{i}^T)^T{\bf u}_{1}\cdot d_{23}-(A-UV^T+U_{i}V_{i}^T)^T{\bf u}_{2}\cdot a-(A-UV^T
        +U_{i}V_{i})^T{\bf u}_{3}\cdot  b \nonumber \\
&=& d_{23}{\bf r}_{1}-a{\bf r}_{2}-b{\bf r}_{3}+{\rm det}\big(U_{i}^TU_{i}\big){\bf v}_{1},   \label{c2}
\end{eqnarray}
where $a, b, d_{23}, d_{13}$, and $d_{12}$ are defined in Eqns. \eqref{q18}--\eqref{q91}.
% and \eqref{q17}.
Also from Eqns. \eqref{c02}--\eqref{c2}, we have
\begin{eqnarray*}
&&{\bf p}=\left[\frac{(A-UV^T+U_{i}V_{i}^T)^T{\bf u}_{2}\cdot\|{\bf u}_{3}\|^2-(A-UV^T
                 +U_{i}V_{i}^T)^T{\bf u}_{3}\cdot{\bf u}^T_{3} {\bf u}_{2}}{d_{23}} \right.\\
& & \qquad -\left. \frac{a}{d_{23}}\big[\frac{(A-UV^T+U_{i}V_{i}^T)^T{\bf u}_{1}\cdot d_{23}-(A-UV^T+U_{i}V_{i}^T)^T{\bf u}_{2} \cdot a
           -(A-UV^T+U_{i}V_{i}^T)^T {\bf u}_{3}\cdot b}{{\rm det}(U_{i}^TU_{i})}\big]_{+}\right]_{+}\nonumber\\
&&~~~=\left[{\bf v}_{2}+\frac{\|{\bf u}_{3}\|^2{\bf r}_{2}-{\bf u}^T_{3}{\bf u}_{2}\cdot {\bf r}_{3}}{d_{23}}
           +\frac{a}{d_{23}}\left({\bf v}_{1}-\big[\frac{d_{23}{\bf r}_{1}-a{\bf r}_{2}-b{\bf r}_{3}}{{\rm det}(U_i^TU_i)}+{\bf v}_{1}\big]_{+}\right)\right]_{+},\\
&&{\bf \widetilde{p}}=\left[\frac{(A-UV^T+U_{i}V_{i}^T)^T{\bf u}_{1}\cdot\|{\bf u}_{3}\|^2
           -(A-UV^T+U_{i}V_{i}^T)^T{\bf u}_{3}\cdot{\bf u}^T_{3}{\bf u}_{1}}{d_{13}}-\frac{a}{d_{13}}{\bf p}\right]_{+}\nonumber\\
&&~~~=\left[{\bf v}_{1}+\frac{\|{\bf u}_{3}\|^2{\bf r}_{1}-{\bf u}^T_{3}{\bf u}_{1}\cdot {\bf r}_{3}}{d_{13}}
           +\frac{a}{d_{13}}({\bf v}_{2}-{\bf p})\right]_{+}.
\end{eqnarray*}
Then, according to Corollary \ref{cor2}, the solution ${\bf v}^*_{3}$ of Problem \eqref{V} is
\begin{align*}
{\bf v}^*_{3}&=\frac{1}{\|{\bf u}_{3}\|^2}[(A-UV^T+U_{i}V_{i}^T)^T{\bf u}_{3}-({\bf u}^T_{3}{\bf u}_{1}){\bf \widetilde{p}}
               -({\bf u}^T_{3}{\bf u}_{2}){\bf p}]_{+}\\
&=\left[{\bf v}_{3}+\frac{{\bf r}_{3}}{\|{\bf u}_{3}\|^2}+\frac{{\bf u}^T_{3}{\bf u}_{1}}{\|{\bf u}_{3}\|^2}({\bf v}_{1}
           -{\bf \widetilde{p}})+\frac{{\bf u}^T_{3}{\bf u}_{2}}{\|{\bf u}_{3}\|^2}({\bf v}_{2}-{\bf p})\right]_{+}.
\end{align*}
Furthermore, letting
\begin{align*}
{\bf z}&=\left[\frac{((A-UV^T+U_{i}V_{i}^T)^T{\bf u}_{1}\|{\bf u}_{2}\|^2
                 -(A-UV^T+U_{i}V_{i}^T)^T{\bf u}_{2}\cdot{\bf u}^T_{2}{\bf u}_{1})}{d_{12}}-\frac{b}{d_{12}}{\bf v}^*_{3}\right]_+\\
&=\left[{\bf v}_{1}+\frac{\|{\bf u}_{2}\|^2{\bf r}_{1}-{\bf u}_{2}^T{\bf u}_{1}\cdot {\bf r}_{2}}{d_{12}}
                 +\frac{b}{d_{12}}({\bf v}_{3}-{\bf v}^*_{3})\right]_{+},
\end{align*}
the solution ${\bf v}^*_{2}$ of Problem \eqref{V} is
\begin{align*}
{\bf v}^*_{2}&=\frac{1}{\|{\bf u}_{2}\|^2}\left[(A-UV^T+U_{i}V_{i}^T)^T{\bf u}_{2}
                      -({\bf u}_{3}^T{\bf u}_{2}){\bf v}^*_{3}-{\bf u}^T_{2} {\bf u}_{1}{\bf z}\right]_{+}\\
&=\left[{\bf v}_{2}+\frac{{\bf r}_{2}}{\|{\bf u}_{2}\|^2}+\frac{{\bf u}^T_{2}{\bf u}_{1}}{\|{\bf u}_{2}\|^2}({\bf v}_{1}
                 -{\bf z})+\frac{{\bf u}^T_{3}{\bf u}_{2}}{\|{\bf u}_{2}\|^2}({\bf v}_{3}-{\bf v}^*_{3})\right]_{+},
\end{align*}
and the solution ${\bf v}^*_{1}$ of Problem \eqref{V} is
\begin{align*}
{\bf v}^*_{1}&=\frac{1}{\|{\bf u}_{1}\|^2}[(A-UV^T+U_{i}V_{i}^T)^T{\bf u}_{1}
                      -({\bf u}^T_{3}{\bf u}_{1}){\bf v}^*_{3}-({\bf u}^T_{2}{\bf u}_{1}) {\bf v}^*_{2}]_{+}\\
&=\left[{\bf v}_{1}+\frac{{\bf r}_{1}}{\|{\bf u}_{1}\|^2}+\frac{{\bf u}^T_{2}{\bf u}_{1}}{\|{\bf u}_{1}\|^2}({\bf v}_{2}
                   -{\bf v}^*_{2})+\frac{{\bf u}^T_{3}{\bf u}_{1}}{\|{\bf u}_{1}\|^2}({\bf v}_{3}-{\bf v}^*_{3})\right]_{+}.
\end{align*}
\end{proof}

\begin{rem}\label{rank-deficient}
Theorem \ref{thm3} holds only when the matrix $U_i$ has full column rank.
When $rank(U_i) \neq 3$, then one or more of the following
values which occur in the denominators will be zero:\\
\begin{itemize}
\item
$\|{\bf u}_1\|$, $\|{\bf u}_2\|$, or $\|{\bf u}_3\|$ when a column of $U_i$ is zero,
\item  $d_{12}=\|{\bf u}_{1}\|^2\|{\bf u}_{3}\|^2-({\bf u}^T_{3}{\bf u}_{1})^2$,
$d_{13}=\|{\bf u}_{1}\|^2\|{\bf u}_{2}\|^2-({\bf u}^T_{2}{\bf u}_{1})^2$, or
$d_{23}=\|{\bf u}_{2}\|^2\|{\bf u}_{3}\|^2-({\bf u}^T_{3}{\bf u}_{2})^2$, 
when two of the columns of $U_i$ are linearly dependent,
\item ${\rm det}(U_i^TU_i)$.
\end{itemize}
\end{rem}
Any of the above cases will make some of the operations not valid.
In the next subsection, we propose a remedy to handle these cases.

\begin{rem}
Theorem \ref{thm3} is based on the assumption that  $r/3$ is an integer.
If $r/3$ is not an integer, the following methods can be used.
Let $q=[{r\over 3}]$ and denote
\[ \begin{cases}  V=\left[ \begin{array}{ccc} {\bf v}_1 & \cdots & {\bf v}_r \end{array}\right]=
             \left[ \begin{array}{cccc} V_1 & \cdots & V_q & {\bf v}_r \end{array}\right], \qquad \qquad\quad
                    V_1, \cdots, V_q\in \mathbb{R}^{n\times 3}, \
                               \quad {\rm if \ } r=3q+1, \\
                  V=\left[ \begin{array}{ccc} {\bf v}_1 & \cdots & {\bf v}_r \end{array}\right]=
                  \left[ \begin{array}{ccccc} V_1 & \cdots & V_q & {\bf v}_{r-1} & {\bf v}_r \end{array}\right], \qquad
                  V_1, \cdots, V_q\in \mathbb{R}^{n\times 3}, \
                  \quad {\rm if \ } r=3q+2,
    \end{cases}
\]
where $[x]$ means the nearest integer less than or equal to $x$. Then $V_1, \cdots, V_q$ are computed by Theorem \ref{thm3}.
For the computation of ${\bf v}_r$ when $r=3q+1$ or ${\bf v}_{r-1}$ and ${\bf v}_r$, there are two choices:
\begin{itemize}
\item[(a)] Let \[ V_{q+1}=\left[ \begin{array}{ccc} {\bf v}_{r-2} & {\bf v}_{r-1} & {\bf v}_r \end{array}\right]. \]
Then $V_{q+1}$ is computed by Theorem \ref{thm3};
\item[(b)] ${\bf v}_r$ is computed by HALS/RRI when $r=3q+1$, and ${\bf v}_{r-1}$ and ${\bf v}_r$ are computed when $r = 3q+2$ by ARkNLS with $k=2$
when $r=3q+2$.
\end{itemize}
\end{rem}
In our implementation, we adopted the choice (a) above due to uniformness of singularity checking. However, even for choice (b), we can easily implement the proposed method to avoid singularity which is discussed in the next section.
 % Section 4 and 4.1

\subsection{Avoiding Rank Deficient ARkNLS in NMF Iteration}\label{singularity}
%\smallskip

The low rank factor matrices $U$ and $V$ in NMF (\ref{nmf}) play important roles in applications.  For example,
in the blind source separation, the matrix $A$ stands for the observation matrix, the matrix $U$ plays the role of mixing
matrix and the matrix $V$ expresses source signals. In topic modeling \cite{DKDP}, where $A$ is a term-document matrix,
the normalized columns of $U$ can be interpreted as topics, and the corresponding columns of $V$ provides the topic distribution
for the documents. If the matrix $V$ has zero columns, this implies that some of source signals may be lost through the process.
 In addition, if any of the columns of these factor matrices are
computed as zeros, then not only the interpretation of the result becomes difficult but also the computed reduced
rank becomes lower than the pre specified reduced rank $r$ which may make the approximation less accurate.
More importantly, ARkNLS iteration assumes that the matrix $U_i$ or $V_i$ that
plays the role of the coefficient matrix in each iteration of NLS has full column rank, and when this is not the case, the algorithm will breakdown. Hence this singularity problem must be overcome for more meaningful solutions as well as for more robust algorithms.
This situation is well known especially in
HALS/RRI for NMF (\ref{nmf}) which is based on the \emph{rank-1}
residue iteration method
where a typical problem of NLS is
\begin{equation}
\min_{{\bf v}_i\in \mathbb{R}^{n\times 1},{\bf v}_i\geq{\bf 0}}\|{\bf u}_i{\bf v}_i^T-
(A-\Sigma_{l\neq i}{\bf u}_l{\bf v}_l^T)\|_F^2.
\end{equation}
The optimal solution vector ${\bf v}_i$ is given by
${\bf v}_i={ {[(A-\Sigma_{l\neq i}{\bf u}_l{\bf v}_l^T)^T {\bf u}_i]_{+}}\over { \|{\bf u}_i\|^2}}$ and thus
it  will be zero if $(A-\Sigma_{l\neq i}{\bf u}_l{\bf v}_l^T)^T {\bf u}_i\leq 0$,
%${\bf u}_i \in null (A-\Sigma_{l\neq i}{\bf u}_l{\bf v}_l^T)^T$ \cite{KHP},
and when this zero vector becomes the coefficient 'matrix' in a later step,
the iteration will break down.
When $k \geq 1$, due to an analogous rank deficiency of the coefficient matrix
$U_i$ or $V_i$, the results in Theorem \ref{thm3} cannot be applied to solve
the NLS problem (\ref{V}). When the NLS problem in the ARkNLS context involves
a rank deficient matrix, a  singularity problem.

Fortunately, in the context of NMF, we can modify the involved NLS problem during rank-k NLS iteration, to avoid such a singularity problem.
This is achieved by
adjusting the columns of $U_i$ and $V_i$ such that the ``new" $U_i$ and $V_i$ satisfy that
$U_i$ is of full column rank and  the value of
\[ U_iV_i^T={\bf u}_1{\bf v}_1^T+{\bf u}_2{\bf v}_2^T+{\bf u}_3{\bf v}_3^T \]
remains unchanged.
The proposed adjustment is motivated by the monotonicity property of our algorithm ARkNLS with $k=3$.
Denote the new
$U_i$ and $V_i$ by $\mathcal{U}_i$ and $\mathcal{V}_i$,
respectively, satisfying
$U_i V_i= \mathcal{U}_i\mathcal{V}_i$. Define $\mathcal{V}_i^*$ as
\[ \mathcal{V}_i^* =\arg \min_{\mathcal{V}\in \mathbb{R}^{n\times 3},\mathcal{V}\geq{\bf 0}}
     \|\mathcal{U}_i\mathcal{V}^T-(A-\sum_{l\neq i}U_lV_l^T)\|_F^2.
\]
Then we have
\[ \|   \mathcal{U}_{i} (\mathcal{V}_i^*)^T-(A-\sum_{l\neq i}U_lV_l^T)\|_F
     \leq \| \mathcal{U}_i \mathcal{V}_i^T-(A-\sum_{l\neq i}U_lV_l^T)\|_F
     =  \| U_i V_i^T-(A-\sum_{l\neq i}U_lV_l^T)\|_F, \]
which preserves the monotonicity property of our algorithm ARkNLS with $k=3$.

Assume that we are concerned with the blocks $U_i$ and $V_i$ in the iteration.
Let
\[ U_i=\left[ \begin{array}{ccc} {\bf u}_1 & {\bf u}_2 & {\bf u}_{3}\end{array}\right]\qquad \mbox{and} \qquad
   V_i=\left[ \begin{array}{ccc} {\bf v}_1 & {\bf v}_2 & {\bf v}_{3}\end{array}\right]. \]
%    H=A^TU_i,  \qquad  M=U^TU_i. \]
Remark \ref{rank-deficient} in the previous subsection listed the cases when $U_i$ is rank deficient.
%Also note that for solving (\ref{Vi}) via \eqref{t2}, only two matrix-matrix products
%$H=A^TU_i$ and $M=U^TU_i$ are needed, and all of ${\bf r}_{1}$, ${\bf r}_{2}$,
%${\bf r}_{3}$, $a$, $b$, $d_{12}$, $d_{13}$, and $d_{23}$ can be invoked via $H$ and $M$
%as shown in Theorem \ref{thm3}. Therefore, we will show how we can detect singularity in terms of matrices $H$ and $M$
%which will have to be computed anyway for other parts of the computation.
%
%\hp{I am not sure whether we want to introduce $H$ and $M$ in this section, and keep track of indexing regaring the entire
%problem, i.e., regarindg $H$ and $M$. It makes expressions complicated.
%If we simply state the key information without showing all details of indexing regarding $H$ and $M$, it will be much easier
%to follow. For example, the values like $d_{ij}$, etc. defined in Theorem 4.1 can be used here. Then we can simply say by computing matrices
%$H$ and $M$, we can obtain all the values we need for singularity problem, and for programming, we can take care of all these
%indexing. What do you think?
%Also I feel Algorithm in the next page is difficult to follow due to detailed indexing.
%How about describing many of those in high level descriptions? At least we need a lot more
%comments in the algorithm so that it is easy to see what is going in major steps in high level}
%
In the following, we show how the factors $U_i$ and $V_i$ can be adjusted
when $U_i$ is rank deficient.
Here, we are adjusting the matrix $U_i$ checking the linear independence of its columns from left and right, and adjust $V_i$ as needed.
\begin{itemize}
\item [Step~i)] In the first step, if ${\bf u}_1 =0$, then we replace this ${\bf u}_1$ with a nonzero vector while keeping the value of
${\bf u}_{1}{\bf v}^T_{1}$ unchanged.
Assume ${\bf u}_1 = 0$, and so ${\bf u}_1{\bf v}_1^T=0$.
Then set $ {\bf u}_{1}(3i-2)=1, \ {\bf v}_{1}={\bf 0}$. Since
the first column ${\bf u}_1$ of $U_i$ is the $(3i-2)$th column of $U$,
${\bf u}_{1}({3i-2})$ is the $(3i-2)$th element of ${\bf u}_{1}$.
Now elements of ${\bf u}_1$ are all zeros except it has only one nonzero element
${\bf u}_{1}(3i-2)=1$. Obviously, ${\bf u}_1\not= 0$ now and the value of
${\bf u}_{1}{\bf v}^T_{1}+{\bf u}_{2}{\bf v}^T_{2}+{\bf u}_{3}{\bf v}^T_{3}$ remains unchanged.

%We can check whether ${\bf u}_1 = 0$ from the relationship ${\bf u}_1^T {\bf u}_1=M(3i-2, 1)=0$.
%Meanwhile, the elements of $H$ and $M$ should also be adjusted:
%\[ H(:,1)=A(3i-2,:)^T,~~~M(:, 1)=U(3i-2,:)^T, ~~~M(3i-2, :)=U_i(3i-2, :).  \]

\item[Step~ii)] Now we have ${\bf u}_1\not= 0$.
But if ${\bf u}_1$ and ${\bf u}_2$ are linearly dependent, then we have
\[ {\bf u}_{2}=\alpha{\bf u}_{1}, \qquad {\bf u}_{1}{\bf v}^T_{1}+{\bf u}_{2}{\bf v}^T_{2}
={\bf u}_{1}({\bf v}^T_{1}+\alpha{\bf v}^T_{2})={\bf u}_{1}({\bf v}_{1}+\alpha{\bf v}_{2})^T, \]
where $\alpha=\frac{\|{\bf u}_2\|}{\|{\bf u}_1\|}$.
%We can check the linear dependence of ${\bf u}_1$ and ${\bf u}_2$ from the relationship
%\[ {\rm det}(\left[ \begin{array}{cc} {\bf u}_1 & {\bf u}_2 \end{array}\right]^T
%              \left[ \begin{array}{cc} {\bf u}_1 & {\bf u}_2 \end{array}\right])={\rm det}(M(3i-2:3i-1, 1:2))=0. \]
%and
%\begin{align}\label{90}
%\alpha=\frac{\|{\bf u}_2\|}{\|{\bf u}_1\|} =\sqrt{\frac{M(3i-1,2)}{M(3i-2,1)}},
%\end{align}
We set
\begin{align}\label{f90}
{\bf v}_{1}={\bf v}_{1}+\alpha{\bf v}_{2},~~ {\bf v}_{2}={\bf 0},~~ {\rm and}~~{\bf u}_{2}={\bf 0}.
\end{align}
Note that the second column ${\bf u}_2$ of $U_i$ is the $(3i-1)$th column of $U$,
we further adjust ${\bf u}_{2}$ to ensure ${\bf u}_{1}$ and ${\bf u}_{2}$ are linearly independent:
\begin{equation} \label{case2-1}
 {\rm if \ }  {\bf u}_{1}(3i-2)\neq 0, {\rm \ set \ }
 {\bf u}_{2}(3i-1)=1; %~H(:,2)=A(3i-1,:)^T, ~M(:, 2)=U(3i-1,:)^T, ~M(3i-1,:)=U_i(3i-1, :);
\end{equation}
\begin{equation}\label{case2-2}
{\rm otherwise, \ set \ }  {\bf u}_{2}(3i-2)=1.
\end{equation}
% ~H(:,2)=A(3i-2,:)^T, ~M(:, 2)=U(3i-2,:)^T, ~M(3i-1,:)=U_i(3i-2, :). \]
Clearly, the adjusted ${\bf u}_{1}$ and ${\bf u}_{2}$ are linearly independent and
${\bf u}_{1}{\bf v}^T_{1}+{\bf u}_{2}{\bf v}^T_{2}$ remains unchanged,
and therefore
${\bf u}_{1}{\bf v}^T_{1}+{\bf u}_{2}{\bf v}^T_{2}+{\bf u}_{3}{\bf v}^T_{3}$ remains unchanged.

\item[Step~iii)] Now ${\bf u}_1 $ and ${\bf u}_2$ are linearly independent.
Finally, if ${\bf u}_1$, ${\bf u}_2$ and  ${\bf u}_3$ are linearly dependent in the following way:
\begin{equation}\label{rank-con1}  {\rm rank}(\left[ \begin{array}{cc} {\bf u}_1 & {\bf u}_2 \end{array}\right])
   ={\rm rank}(\left[ \begin{array}{ccc} {\bf u}_1 & {\bf u}_2 & {\bf u}_3 \end{array}\right])=2,
\end{equation}
then we have
\[ {\bf u}_{3}=\widetilde{\alpha}{\bf u}_{1}+\widetilde{\beta}{\bf u}_{2}, \]
%Note that we can check the condition (\ref{rank-con1}) from the relationship
%\[ {\rm det}(\left[ \begin{array}{ccc} {\bf u}_1 & {\bf u}_2 & {\bf u}_3 \end{array}\right]^T
%              \left[ \begin{array}{ccc} {\bf u}_1 & {\bf u}_2 & {\bf u}_3 \end{array}\right])=
%      {\rm det}(M(3i-2:3i,1:3))=0,  \]
%and furthermore,
where
\begin{eqnarray}\label{z1}
\widetilde{\alpha}=\frac{{\rm det}(\left[ \begin{array}{cc} {\bf u}_1 & {\bf u}_2 \end{array}\right]^T
                                    \left[ \begin{array}{cc} {\bf u}_3 & {\bf u}_2 \end{array}\right])}
                        {{\rm det}(\left[ \begin{array}{cc} {\bf u}_1 & {\bf u}_2 \end{array}\right]^T
                                    \left[ \begin{array}{cc} {\bf u}_1 & {\bf u}_2 \end{array}\right]) }
%=\frac{M(3i-1,2)M(3i-2,3)-M(3i-1,3)M(3i-2,2)}{M(3i-1,2)M(3i-2,1)-(M(3i-1,1))^2},
\end{eqnarray}
and
\begin{eqnarray}\label{z2}
\widetilde{\beta}=\frac{{\rm det}(\left[ \begin{array}{cc} {\bf u}_1 & {\bf u}_2 \end{array}\right]^T
                                    \left[ \begin{array}{cc} {\bf u}_1 & {\bf u}_3 \end{array}\right])}
                        {{\rm det}(\left[ \begin{array}{cc} {\bf u}_1 & {\bf u}_2 \end{array}\right]^T
                                    \left[ \begin{array}{cc} {\bf u}_1 & {\bf u}_2 \end{array}\right]) }.
%=\frac{M(3i-2,1)M(3i-1,3)-M(3i-2,2)M(3i-2,3)}{M(3i-1,2)M(3i-2,1)-(M(3i-1,1))^2}.
\end{eqnarray}
Actually, $\widetilde{\alpha}$ and $\widetilde{\beta}$ cannot be both negative as ${\bf u}_1$, ${\bf u}_2$ and ${\bf u}_{3}$ are all
nonnegative, ${\bf u}_1\not=0$ and ${\bf u}_2\not=0$. But when only one of them is negative,
we need to permute the index list $\left[ \begin{array}{ccc} 1 & 2 & 3 \end{array}\right]$ to the list
$\mathcal{I}=\left[ \begin{array}{ccc} \mathcal{I}(1) & \mathcal{I}(2) & \mathcal{I}(3) \end{array}\right]$
and change the values of $\widetilde{\alpha}$ and $\widetilde{\beta}$ such that
\[ {\bf u}_{\mathcal{I}(3)}=\widetilde{\alpha} {\bf u}_{\mathcal{I}(1)}+\widetilde{\beta} {\bf u}_{\mathcal{I}(2)},
   \quad \widetilde{\alpha}\geq 0, \quad \widetilde{\beta}\geq 0. \]
At the same time, ${\bf u}_{\mathcal{I}(1)}$ and ${\bf u}_{\mathcal{I}(2)}$ must be linearly independent. So we
can adjust ${\bf v}_1$, ${\bf v}_2$ and ${\bf v}_3$ easily and meanwhile we only need to
adjust ${\bf u}_{\mathcal{I}(3)}$ such that the adjusted ${\bf u}_1$, ${\bf u}_2$ and ${\bf u}_3$ are linearly independent,
and the value of $U_iV_i^T={\bf u}_1{\bf v}_1+{\bf u}_2{\bf v}_2+{\bf u}_3{\bf v}_3$ remains unchanged.
The permuted index list $\mathcal{I}=\left[ \begin{array}{ccc} \mathcal{I}(1) & \mathcal{I}(2) & \mathcal{I}(3) \end{array}\right]$
can be obtained as follows where $\mathcal{J}$ denotes the index list in the entire matrix $U$ or $V$:
\begin{itemize}
\item if $\widetilde{\alpha}\widetilde{\beta}\geq 0$,
then $\widetilde{\alpha}\geq 0$, $\widetilde{\beta}\geq 0$, and
${\bf u}_{3}=\widetilde{\alpha}{\bf u}_{1}+\widetilde{\beta}{\bf u}_{2}$. Let
\[ \mathcal{I}=\left[ \begin{array}{ccc} 1 & 2 & 3\end{array}\right], \quad
   \mathcal{J}=\left[ \begin{array}{ccc} 3i-2 & 3i-1 & 3i \end{array}\right]; \]
\item if $\widetilde{\alpha}<0$, $\widetilde{\beta}>0$, then ${\bf u}_{2}
          =-\frac{\widetilde{\alpha}}{\widetilde{\beta}}{\bf u}_{1}+\frac{1}{\widetilde{\beta}}{\bf u}_{3}$
with $-\frac{\widetilde{\alpha}}{\widetilde{\beta}}>0$ and $\frac{1}{\widetilde{\beta}}>0$.
Let
\[ \widetilde{\alpha}={-\widetilde{\alpha}}/{\widetilde{\beta}}>0,
    ~~~\widetilde{\beta}={1}/{\widetilde{\beta}}>0,~~~
    \mathcal{I}=\left[ \begin{array}{ccc} 1 & 3 & 2\end{array}\right], \quad
    \mathcal{J}=\left[ \begin{array}{ccc} 3i-2 & 3i & 3i-1\end{array}\right]; \]
\item if $\widetilde{\alpha}>0$, $\widetilde{\beta}<0$, then
       ${\bf u}_{1}=-\frac{\widetilde{\beta}}{\widetilde{\alpha}}{\bf u}_{2}+\frac{1}{\widetilde{\alpha}}{\bf u}_{3}$
       with $-\frac{\widetilde{\beta}}{\widetilde{\alpha}}>0$ and $\frac{1}{\widetilde{\alpha}}>0$.
Let
\[ \widetilde{\alpha}=-{\widetilde{\beta}}/{\widetilde{\alpha}}>0,~~~
     \widetilde{\beta}={1}/{\widetilde{\alpha}}>0,~~~
     \mathcal{I}=\left[ \begin{array}{ccc} 2 & 3 & 1\end{array}\right], \quad
     \mathcal{J}=\left[ \begin{array}{ccc} 3i-1 & 3i & 3i-2\end{array}\right].\]
\end{itemize}
Now we have
\[ {\bf u}_{\mathcal{I}(3)}=\widetilde{\alpha} {\bf u}_{\mathcal{I}(1)}+\widetilde{\beta} {\bf u}_{\mathcal{I}(2)},
   \quad \widetilde{\alpha}\geq 0, \ \widetilde{\beta}\geq 0, \]
and ${\bf u}_{\mathcal{I}(j)}$ is the $\mathcal{J}(j)$-th column of $U$, $j=1, 2, 3$.
Moreover, we also have
\begin{eqnarray*}
 {\bf u}_1{\bf v}_1^T+ {\bf u}_2{\bf v}_2^T+ {\bf u}_3{\bf v}_3^T
  &=& {\bf u}_{\mathcal{I}(1)} {\bf v}^T_{\mathcal{I}(1)}+
   {\bf u}_{\mathcal{I}(2)}{\bf v}^T_{\mathcal{I}(2)} +
   {\bf u}_{\mathcal{I}(3)} {\bf v}^T_{\mathcal{I}(3)}  \\
   &=& {\bf u}_{\mathcal{I}(1)}( {\bf v}_{\mathcal{I}(1)} + \widetilde{\alpha} {\bf v}_{\mathcal{I}(3)})^T
   +{\bf u}_{\mathcal{I}(2)}( {\bf v}_{\mathcal{I}(2)} + \widetilde{\beta} {\bf v}_{\mathcal{I}(3)})^T.
\end{eqnarray*}
Thus, we set
\[ {\bf v}_{\mathcal{I}(1)}={\bf v}_{\mathcal{I}(1)}+\widetilde{\alpha}{\bf v}_{\mathcal{I}(3)},~~~
           {\bf v}_{\mathcal{I}(2)}={\bf v}_{\mathcal{I}(2)}+\widetilde{\beta}{\bf v}_{\mathcal{I}(3)},
           ~~~{\bf v}_{\mathcal{I}(3)}={\bf 0},~~~{\bf u}_{\mathcal{I}(3)}={\bf 0}.  \]
then ${\bf v}_{\mathcal{I}(1)}$ and ${\bf v}_{\mathcal{I}(2)}$ are both nonnegative vectors,
${\bf u}_{\mathcal{I}(1)}$ and ${\bf u}_{\mathcal{I}(2)}$ are linearly independent, and the value of
${\bf u}_{1}{\bf v}^T_{1}+{\bf u}_{2}{\bf v}^T_{2}+{\bf u}_{3}{\bf v}^T_{3}$ remains unchanged.

Now we adjust ${\bf u}_{\mathcal{I}(3)}$ so that ${\bf u}_{1},{\bf u}_{2}$, ${\bf u}_{3}$ are linearly independent
and the value of ${\bf u}_{1}{\bf v}^T_{1}+{\bf u}_{2}{\bf v}^T_{2}+{\bf u}_{3}{\bf v}^T_{3}$ remains unchanged:
\begin{equation}\label{case3-1}
 {\rm If \ } {\bf u}_{\mathcal{I}(1)}(\mathcal{J}(1)){\bf u}_{\mathcal{I}(2)}(\mathcal{J}(2))
     -{\bf u}_{\mathcal{I}(1)}(\mathcal{J}(2)){\bf u}_{\mathcal{I}(2)}(\mathcal{J}(1))\neq 0,
     {\rm \ set \ }
         {\bf u}_{\mathcal{I}(3)}(\mathcal{J}(3))=1;
\end{equation}
% \quad H(:,\mathcal{I}(3))=A(\mathcal{J}(3),:)^T,\quad
%             M(:,\mathcal{I}(3))=U(\mathcal{J}(3),:)^T,
%        \quad M(\mathcal{J}(3),:)=U_i(\mathcal{J}(3), :); \]
\begin{equation}\label{case3-2}
{\rm  Else,\  if \ } {\bf u}_{\mathcal{I}(1)}(\mathcal{J}(1))+{\bf u}_{\mathcal{I}(2)}(\mathcal{J}(1))=0,
 {\rm \ set \ }
     {\bf u}_{\mathcal{I}(3)}(\mathcal{J}(1))=1;
\end{equation}
%\quad H(:,\mathcal{I}(3))=A(\mathcal{J}(1),:)^T,\quad
%                M(:,\mathcal{I}(3))=U(\mathcal{J}(1),:)^T, \quad
%                M(\mathcal{J}(3),:)=U_i(\mathcal{J}(1), :);\]
\begin{equation}\label{case3-3}
\quad {\rm else \ if \ }  {\bf u}_{\mathcal{I}(1)}(\mathcal{J}(1))+{\bf u}_{\mathcal{I}(2)}(\mathcal{J}(1))\neq 0,
   {\rm \ set \ }
 {\bf u}_{\mathcal{I}(3)}(\mathcal{J}(2))=1.
\end{equation}
% \quad H(:,\mathcal{I}(3))=A(\mathcal{J}(2),:)^T, \quad
%   M(:,\mathcal{I}(3))=U(\mathcal{J}(2),:)^T, \quad
%    M(\mathcal{J}(3),:)=U_i(\mathcal{J}(2), :). \]
\end{itemize}

Through the above three steps, $U_i$ is of full column rank and so Theorem \ref{thm3} can be applied.

\subsection{Algorithm ARkNLS with $k=3$}
Denote
\[    H=A^TU_i,  \qquad  M=U^TU. \]
Note that for solving (\ref{Vi}) via \eqref{t2}, only two matrix-matrix products $H$ and $M$ are needed, and
all of ${\bf r}_{1}$, ${\bf r}_{2}$, ${\bf r}_{3}$, $a$, $b$, $d_{12}$, $d_{13}$, and $d_{23}$ can be invoked via $H$ and $M$
as shown in Theorem \ref{thm3}. Moreover, the singularity problem discussed in the subsection above can be detected via
matrices $H$ and $M$ which will have to be computed anyway for other parts of the computation.

In the following we illustrate the implementation of adjustment of vectors ${\bf u}_1$, ${\bf u}_2$ and ${\bf u}_3$
first and then present Algorithm ARkNLS(k=3).

\begin{itemize}
\item In Step i), whether ${\bf u}_1 = 0$ can be checked by checking whether
${\bf u}_1^T {\bf u}_1=M(3i-2, 3i-2)$
is zero. After ${\bf u}_1$ is adjusted, the elements of $H$ and $M$ should also be adjusted:
\[ H(:,1)=A(3i-2,:)^T,~~~M(:, 3i-2)=U(3i-2,:)^T, ~~~M(3i-2, :)=U(3i-2, :).  \]

\item In Step ii), the linear dependence of ${\bf u}_1$ and ${\bf u}_2$ can be checked by checking whether
${\rm det}(\left[ \begin{array}{cc} {\bf u}_1 & {\bf u}_2 \end{array}\right]^T
              \left[ \begin{array}{cc} {\bf u}_1 & {\bf u}_2 \end{array}\right])={\rm det}(M(3i-2:3i-1, 3i-2:3i-1))$ is zero.
Moreover,
\[ \alpha=\frac{\|{\bf u}_2\|}{\|{\bf u}_1\|} =\sqrt{\frac{M(3i-1,3i-1)}{M(3i-2,3i-2)}}.
\]
Furthermore, After ${\bf u}_2$ is adjusted, the elements of $H$ and $M$ should also be adjusted:
\begin{itemize}
\item Corresponding to (\ref{case2-1}),
\[ H(:,2)=A(3i-1,:)^T, ~M(:, 3i-1)=U(3i-1,:)^T, ~M(3i-1,:)=U(3i-1, :); \]
\item Corresponding to (\ref{case2-2}),
\[ H(:,2)=A(3i-2,:)^T, ~M(:, 3i-1)=U(3i-2,:)^T, ~M(3i-1,:)=U(3i-2, :). \]
\end{itemize}

\item In Step iii), the condition (\ref{rank-con1}) can be checked from the relationship
\[ {\rm det}(\left[ \begin{array}{ccc} {\bf u}_1 & {\bf u}_2 & {\bf u}_3 \end{array}\right]^T
              \left[ \begin{array}{ccc} {\bf u}_1 & {\bf u}_2 & {\bf u}_3 \end{array}\right])=
      {\rm det}(M(3i-2:3i,3i-2:3i))=0.  \]
In addition,
\[
\widetilde{\alpha}=\frac{{\rm det}(\left[ \begin{array}{cc} {\bf u}_1 & {\bf u}_2 \end{array}\right]^T
                                    \left[ \begin{array}{cc} {\bf u}_3 & {\bf u}_2 \end{array}\right])}
                        {{\rm det}(\left[ \begin{array}{cc} {\bf u}_1 & {\bf u}_2 \end{array}\right]^T
                                    \left[ \begin{array}{cc} {\bf u}_1 & {\bf u}_2 \end{array}\right]) }
=\frac{M(3i-1,3i-1)M(3i-2,3i)-M(3i-1,3i)M(3i-2,3i-1)}{M(3i-1,3i-1)M(3i-2,3i-2)-(M(3i-1,3i-2))^2},
\]
and
\[
\widetilde{\beta}=\frac{{\rm det}(\left[ \begin{array}{cc} {\bf u}_1 & {\bf u}_2 \end{array}\right]^T
                                    \left[ \begin{array}{cc} {\bf u}_1 & {\bf u}_3 \end{array}\right])}
                        {{\rm det}(\left[ \begin{array}{cc} {\bf u}_1 & {\bf u}_2 \end{array}\right]^T
                                    \left[ \begin{array}{cc} {\bf u}_1 & {\bf u}_2 \end{array}\right]) }.
=\frac{M(3i-2,3i-2)M(3i-1,3i)-M(3i-1,3i-2)M(3i-2,3i)}{M(3i-1,3i-1)M(3i-2,3i-2)-(M(3i-1,3i-2))^2}.
\]
After ${\bf u}_{\mathcal{I}(3)}$ is adjusted, the elements of $H$ and $M$ should also be adjusted:
\begin{itemize}
\item Corresponding to (\ref{case3-1}),
\[     H(:,\mathcal{I}(3))=A(\mathcal{J}(3),:)^T,\quad
       M(:,\mathcal{J}(3))=U(\mathcal{J}(3),:)^T, \quad
       M(\mathcal{J}(3),:)=U(\mathcal{J}(3), :); \]

\item Corresponding to (\ref{case3-2}),
\[      H(:,\mathcal{I}(3))=A(\mathcal{J}(1),:)^T,\quad
        M(:,\mathcal{J}(3))=U(\mathcal{J}(1),:)^T, \quad
        M(\mathcal{J}(3),:)=U(\mathcal{J}(1), :);   \]

\item Corresponding to  (\ref{case3-3}),
\[      H(:,\mathcal{I}(3))=A(\mathcal{J}(2),:)^T, \quad
        M(:,\mathcal{J}(3))=U(\mathcal{J}(2),:)^T, \quad
        M(\mathcal{J}(3),:)=U(\mathcal{J}(2), :). \]
\end{itemize}

\end{itemize}

Theorem \ref{thm3} and discussions above lead to Algorithm-ARkNLS(k=3) for NMF problem, which is summarized in Algorithm \ref{alg:alg2}.

\begin{algorithm}%[H]
\caption{Alternating \emph{Rank-3 NLS} for NMF (ARkNLS(k=3))} \label{alg:alg2}
\begin{algorithmic}
\STATE 1.~Given a $m$-by-$n$ nonnegative matrix $A$. Initialize $U\in \mathbb{R}^{m\times r}$, $V\in \mathbb{R}^{n\times r}$
with nonnegative elements and normalized the columns of $U$. Let $q={\tt floor}(\frac{r}{3})$, where ${\tt floor}(\frac{r}{3})$
rounds $\frac{r}{3}$ to the nearest integer less than or equal to $\frac{r}{3}$
\STATE 2.~\textbf{Repeat}
\STATE 3.~\textbf{For} $i=1:q$ \textbf{do} \qquad \qquad\quad
                       \  ~\% $U_i=\left[ \begin{array}{ccc} {\bf u}_1 & {\bf u}_2 & {\bf u}_3 \end{array}\right]$,
                         ${\bf u}_1=U(:, 3i-2), ~{\bf u}_2=U(:, 3i-1), ~{\bf u}_3=U(:, 3i)$.
\STATE    \hskip 5cm  \%   $V_i=\left[ \begin{array}{ccc} {\bf v}_1 & {\bf v}_2 & {\bf v}_3 \end{array}\right]$,
                         ${\bf v}_1=V(:, 3i-2), ~{\bf v}_2=V(:, 3i-1), ~{\bf v}_3=V(:, 3i)$.
%\STATE 4.~Let $H=A^TU_i$, $M=U^TU_i$. \ \% For the implementation of (\ref{Vi}) with \eqref{t2}, it only needs to compute
%\STATE    \hskip 5cm  \% the two matrix-matrix products $H=A^TU_i$ and $M=U^TU_i$,
%\STATE    \hskip 5cm  \% then all of $r_{1}$, $r_{2}$, $r_{3}$, $a$, $b$, $d_{12}$, $d_{13}$, and $d_{23}$ can be invoked via $H$ and $M$.

\STATE 4. Adjust the columns of $U_i$ and $V_i$ such that the adjusted $U_i$ is of full column rank and the value of $U_iV_i^T$
\STATE \qquad remains unchanged.
\STATE 5.~Compute $V_i$ for (\ref{Vi}) by  \eqref{t2}.
\STATE 6.~\textbf{End for}

\STATE 7.~\textbf{If} ${\tt mod}(r,3)=1 {\rm ~or~} 2$ \textbf{then} compute ${{\bf v}}_{r-2}$, ${{\bf v}}_{r-1}$, ${{\bf v}}_{r}$
               via the above Steps 4--5. \textbf{End if}
\STATE       \hskip 5cm \% ${\tt mod}(r,3)$ returns the remainder after division of $r$ by $3$.
\STATE
\STATE       \hskip 5cm \% In the case ${\tt mod}(r,3)=1$, ${\bf u}_{r-2}$ and ${\bf u}_{r-1}$ must be linearly independent
\STATE       \hskip 5cm \%   after Step 5, we only need to check whether ${\bf u}_{r-2}$, ${\bf u}_{r-1}$ and ${\bf u}_r$ are
\STATE       \hskip 5cm \%   linearly independent before  updating the last three columns of $V$.
%\STATE 35.~~~~set $3i-2:=r-2$, $3i-1:=r-1$, $3i:=r$, compute ${{\bf v}}_{r-2}$, ${{\bf v}}_{r-1}$, ${{\bf v}}_{r}$
%               via the above Steps 16--32.
%\STATE 36.~\textbf{end if}
%\STATE 37.~\textbf{If} ${\tt mod}(r,3)=2$, \textbf{then}

\STATE
\STATE        \hskip 5cm \% In  the case ${\tt mod}(r,3)=2$, ${\bf u}_{r-2}\not= 0$ after Step 5, we only need to check
\STATE        \hskip 5cm \%   whether ${\bf u}_{r-2}$ and ${\bf u}_{r-1}$, and  ${\bf u}_{r-2}$, ${\bf u}_{r-1}$ and ${\bf u}_r$
                            are linearly independent
\STATE        \hskip 5cm \%   before updating the last three columns of $V$.
%\STATE 38.~~~~set $3i-2:=r-2$, $3i-1:=r-1$, $3i:=r$, compute ${{\bf v}}_{r-2}$, ${{\bf v}}_{r-1}$, ${{\bf v}}_{r}$
%                via the above Steps 9--32.
%\STATE 39.~\textbf{end if}
\STATE 8.~Replace $U$ with $V$ and $V$ with $U$, $A$ with $A^T$, repeat Steps 3--7.
\STATE 9.~\textbf{Until} a stopping criterion is satisfied
\end{algorithmic}
\end{algorithm}

\subsection{Computational complexity.}
\label{sec:itercomp}
We briefly discuss the per iteration computational complexity of~\cref{alg:alg2}. The matrix multiplications needed to compute $M = U^TU$ and $H = A^T U$ cost $2mr^2$ flops and $2mnr$ flops respectively. Adjusting the rank deficiency of $U$ requires only $\approx 6n$ of flops and data movement per $U_i$ to replace entries of $M$ and $H$. Solving for $V_i$ using~\cref{thm3} involves calculating $A^TU_i - VU^TU_i$ which can be computed using $H$ and $M$ using $7mr$ flops. Finally solving for $V_i$ involves a constant number of vector operations taking $\approx 50m$ flops. These solves are updated for every $V_i$ block giving us a total time for updating $V$ to be $2mnr + 2mr^2 + \frac{r}{3}(7mr+50m + 6n) = 2mnr + O(mr^2)$. Performing a similar analysis for updating $U$ we get the overall per iteration computational complexity of~\cref{alg:alg2} to be $4mnr + O\left((m+n)r^2\right)$ flops. If $r$ is small the computation is dominated by the matrix multiplication operations involving $A$.% Section 4.2
\section{Numerical Experiments}\label{sec5}
%Machine parameters
In this section we provide numerical experiments on synthetic data sets and real world text and image data sets. All methods were implemented in MATLAB (version R2017a) and the experiments were conducted on a server with 2 Intel(R) Xeon(R) CPU ES-2680 v3 CPUs and 377GB RAM.
We compared the following algorithms for NMF.
\begin{enumerate}
    \item (ARk) ARkNLS with $k=3$, the method proposed in this paper.
    \item (HALS) Cichocki and Phan's hierarchical alternating least squares algorithm~\cite{CP}.
    \item (BPP) Kim and Park's block principal pivoting method~\cite{KP}.
    \item (RTRI) Liu and Zhou's rank-two residual iteration method~\cite{LZ}. 
\end{enumerate}
% \hp{explain why we tested only the following algos}
Prior work~\cite{KHP,CP,CZPA} has shown that HALS and BPP are two of the most effective methods for NMF. RTRI is an algorithm with a similar style as our proposed method and thus serve as good benchmark as well. We implemented the RTRI algorithm and utilized the BPP and HALS implementation provide by Kim et al.~\cite{KHP}. In all our experiments $A \in \mathbb{R}_+^{m \times n}$ refers to the input matrix with $r$ being the approximation rank.

\subsection{Data sets}
\label{sec:datasets}

\begin{table}[tb]
\caption{Data sets}\label{tbl:rwdata}
\begin{center}
\begin{tabular}{ccc}
\hline
Dataset & Size & Sparsity\\
\hline
    TDT2          & 36,771 $\times$ 9,394  & 99.65\%\\
    Reuters       & 18,933 $\times$ 8,293  & 99.75\%\\
    20Newsgroups  & 26,214 $\times$ 18,846 & 99.66\%\\
    ORL           & 10,304 $\times$ 400    & 0.01\%\\
    Facescurb     & 9,216  $\times$ 22,631 & 0.39\%\\
    YaleB         & 10,000 $\times$ 2,432  & 2.70\%\\
    Caltech256    & 9,216  $\times$ 30,607 & 1.00\%\\
\hline
\end{tabular}
\end{center}
\end{table}
%\hp{Srini, I switched rows and columns sizes for Caltech256, correct?}
We use 7 real-world data sets in our experiments. Three sparse text data sets~\footnote{The text data sets are available at \url{http://www.cad.zju.edu.cn/home/dengcai/}} TDT2, Reuters21578, and 20Newsgroups and 4 dense image datasets ORL, Facescurb, YaleB, and Caltech256.
A summary of their characteristics can be found in~\Cref{tbl:rwdata}. 
The text data is represented as a term-document matrix and images are represented as a vector of pixels. In addition, we test the methods on various synthetic data sets as described later in  \Cref{sec:synexp}. Detailed description of the real-world data sets are as follows:

\begin{enumerate}
    \item TDT2: The NIST Topic Detection and Tracking corpus consists of news articles collected during 1998 and taken from various sources including television programs, radio programs, and news wires. The articles are classified into 96 categories. Documents appearing in multiple categories are pruned leaving us with 9,394 documents and 30 classes of documents in total.
    \item Reuters: We use the ModApte version of the Reuters21578 corpus. It consists of articles appearing in the 1987 Reuters news wire. Retaining documents with only single labels leaves us with 8,293 documents in 65 categories.
    \item 20Newsgroups: It is a collection of newsgroup documents partitioned across 20 different categories with 18,846 documents.
    \item ORL\footnote{\url{https://github.com/fengbingchun/NN\_Test}}: AT\&T Laboratories Cambridge collected 400 facial images of 40 different people with different expressions and postures. Each image has $92\times112$ pixels resulting in matrix of dimension $10,304\times400$.
    \item Facescrub: This database contains 106,863 photos of 530 celebrities, 265 whom are male, and 265 female~\cite{NW}. The initial images that make up this dataset were procured using Google Image Search. Subsequently, they were processed using the Haarcascade-based face detector from OpenCV 2.4.7 on the images to obtain a set of faces for each celebrity name, with the requirement that a face must be at least $96\times 96$ pixels. In our experiment, we use photos from 256 male, and scale the images to $96\times 96$ pixels.
    \item YaleB~\footnote{\url{http://vision.ucsd.edu/~iskwak/ExtYaleDatabase/Yale\%20Face\%20Database.htm}}: 5,760 single light source images of 10 subjects were collected under 576 viewing conditions. The images have normal, sleepy, sad and surprising expressions. A subset of 38 persons with 64 images per people, i.e., 2,432 images are used in this paper. We scale the images to $100\times 100$ pixels each.
    \item Caltech256\footnote{\url{http://www.vision.caltech.edu/Image\_Datasets/Caltech256/}}: This corpus is a set of 256 object categories containing a total of 30,607 images. They were collected by choosing a set of object categories, downloading examples from Google Images and then manually screening out all images that did not fit the category. Images to $96\times 96$ pixels in our experiments.
\end{enumerate}

We also use synthetic data for additional benchmarks. The details of these sets can be found in~\Cref{sec:kselect,sec:sparse}.

%$k=2$ versus $k=3$
\subsection{ARkNLS with $k=2$ vs $k=3$}
\label{sec:kselect}
\begin{figure}[tb]
\begin{subfigure}{.32\textwidth}
\centering\includegraphics[width=\linewidth]{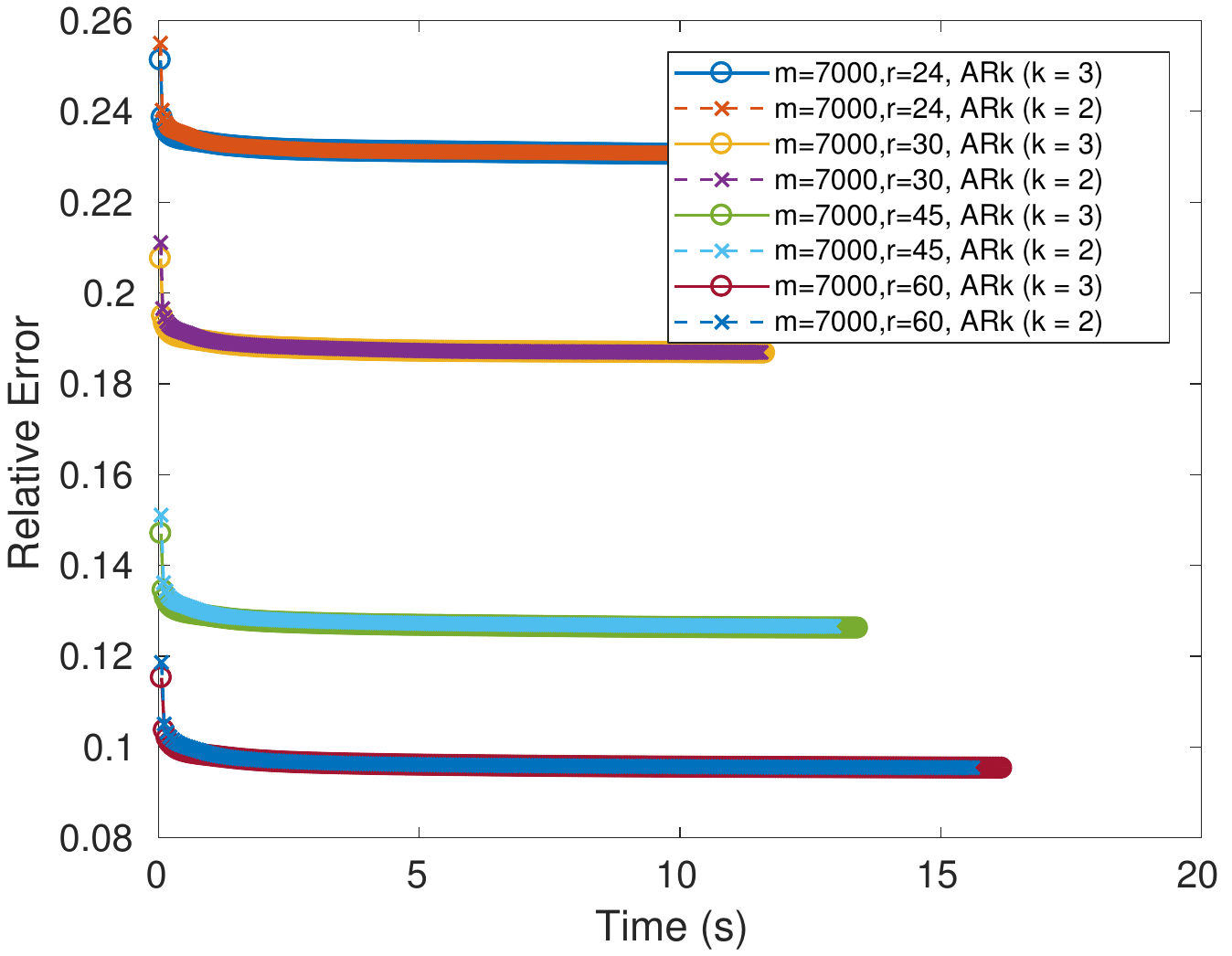}
\caption{Varying $r$ for $m = 7,000$}\label{fig:kselectallr}
\end{subfigure}\hfill
\begin{subfigure}{.32\textwidth}
\centering\includegraphics[width=\linewidth]{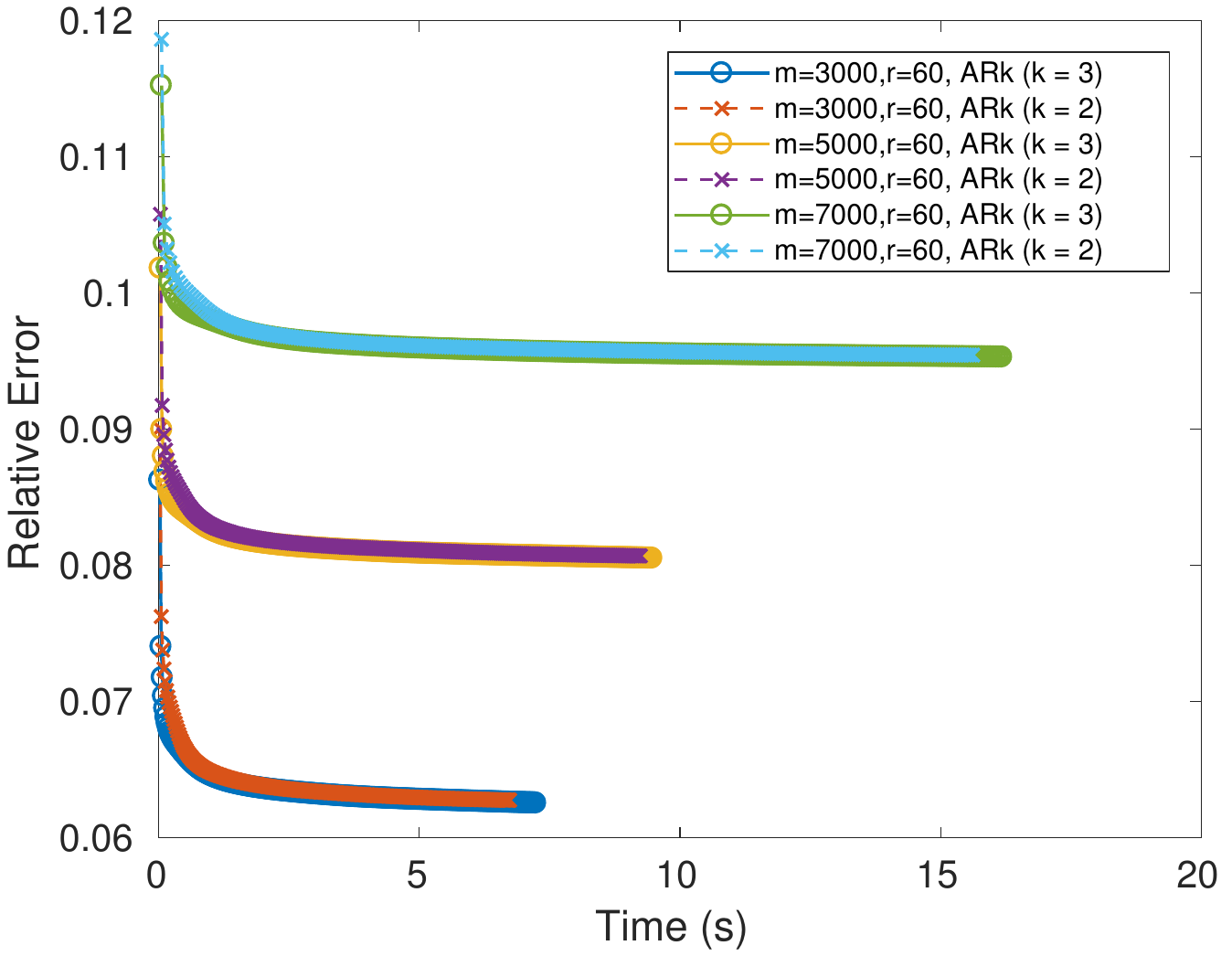}
\caption{Varying $m$ for $r=60$}\label{fig:kselectallm}
\end{subfigure}\hfill
\begin{subfigure}{.32\textwidth}
\centering\includegraphics[width=\linewidth]{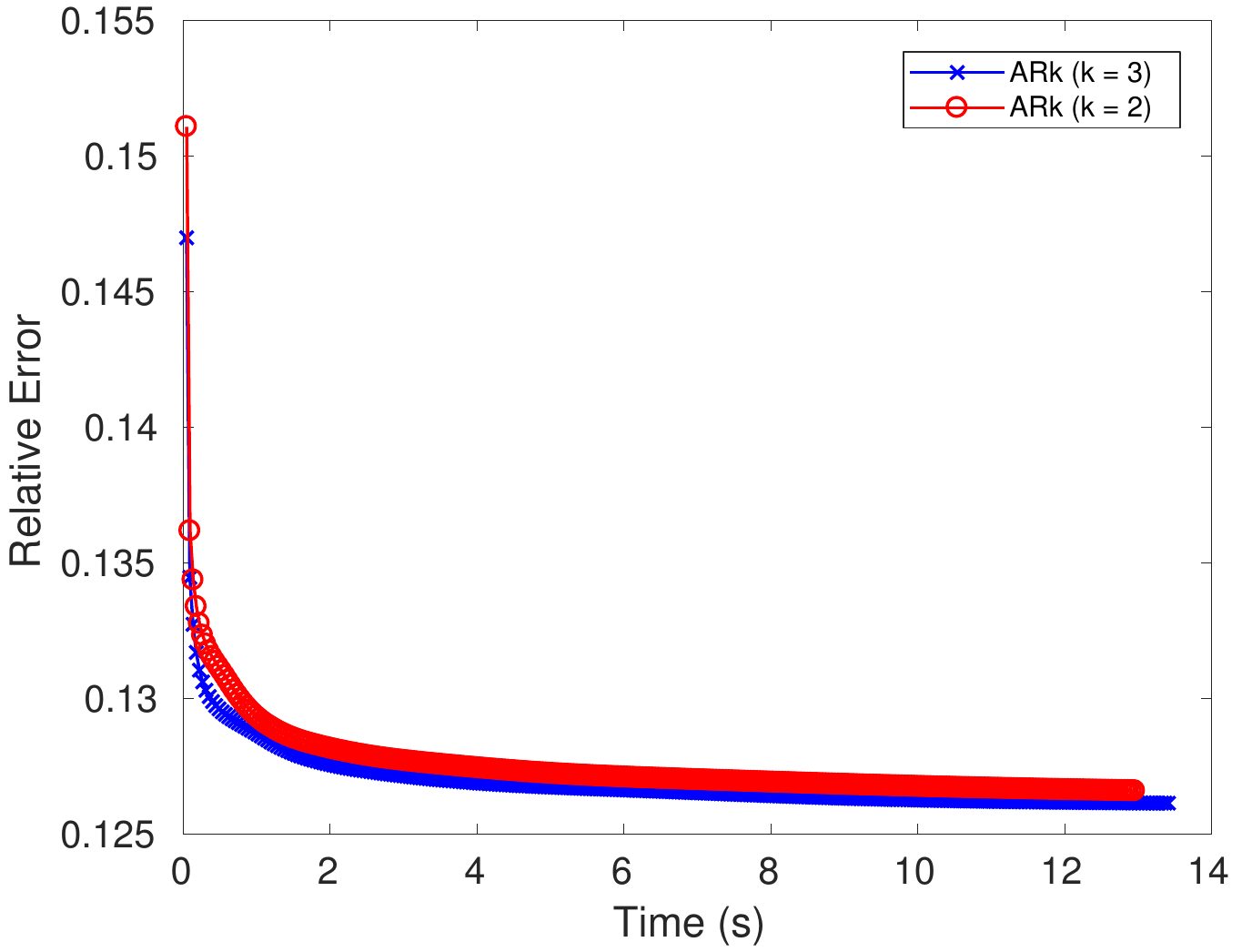}
\caption{$m=7,000$ and $r=45$}\label{fig:kselectcase}
\end{subfigure}
\caption{Synthetic experiments for selecting $k$ for ARk. Both $k=2$ and $k=3$ perform similarly with $k=3$ obtaining marginally better approximations in slightly longer time.}
\label{fig:kselect}
\end{figure}

As stated earlier, although ARkNLS can be developed for any integer $k$, we focus on 
%One of the parameters needed for ARkNLS is the base case $k$ needed in the NLS solver developed. We need to be mindful of the computational cost for solving \hp{for solving what?} (see~\Cref{sec4}) we test 
the choices of $k=2$ and $k=3$ for efficiency of closed form solutions. We test the cases of ARkNLS for $k=2$ and $k=3$ on synthetic low-rank matrices. 
The synthetic matrices are created as $A = WH^T + N$ where $W \in \mathbb{R}_+^{m\times r}$ and $H \in \mathbb{R}_+^{n\times r}$ are random nonnegative matrices where the columns of $W$ have unit norm and 
$N \in \mathbb{R}^{m\times n}$ is random Gaussian matrix with 0 mean and 0.03 standard deviation. 
We ensure that $A$ is nonnegative by replacing negative values with 0. We fix the number of columns $n$ to 5,000 and vary the number of rows $m$. We run ARkNLS for 300 iterations of updating every column of $U$ and $V$. The results of these experiments can be seen in~\Cref{fig:kselect}.

From~\cref{fig:kselectallr,fig:kselectallm} we can see that the choices of $k$ do not affect the running time and convergence characteristics of ARkNLS too much. Relative residual is measured as $\lVert A-UV^T\rVert_F/\lVert A \rVert_F$. We show a particular case of the runs in~\cref{fig:kselectcase}. Here we can see that $k=3$ achieves slightly higher accuracy and runs at about the same time. This trend is true for all configurations of $m$ and $r$. This makes sense intuitively since $k=3$ updates one more column per block than $k=2$. The major computational bottleneck both variations comes from matrix multiplications involving $A$ (see~\cref{sec:itercomp}) and does not vary with $k$. Hence we focus only on the $k=3$ setting for the rest of our experiments.
%\hp{Srini, state precisely this many flops are needed for what},
%\se{Should we have a section on computational complexity of ARkNLS in section 4 to refer?}
%\hp{Yes, can you add them?}

\subsection{Experiments on Synthetic Data}
\label{sec:synexp}
The convergence behavior of the different NMF algorithms is compared on synthetic matrices in the following experiments.

%Synthetic Experiments - Dense
\subsubsection{Experiments on Dense Synthetic Data}
The dense synthetic matrices are created in the same manner as described in~\Cref{sec:kselect}. Defining a stopping criteria for iterative algorithms like NMF is often a tricky task. Many options exist~\cite{KHP} but for comparison purposes in this section we run all algorithms for 100 iterations and observe their convergence behavior. We ran our tests with $n=15,000$ and vary $m$ from 5,000 to 30,000 in increments of 5,000. $r$ is fixed as one of 30,60, and 90. We split the input matrices into Short-Fat ($m<n$), Square ($m=n$), and the Tall-Skinny ($m>n$) cases. Only particular instances of each case is shown in~\Cref{fig:synsf,fig:synsq,fig:synts} with the results being similar in the other experiments. 

\begin{figure}[bt]
\begin{subfigure}{.32\textwidth}
\centering\includegraphics[width=\linewidth]{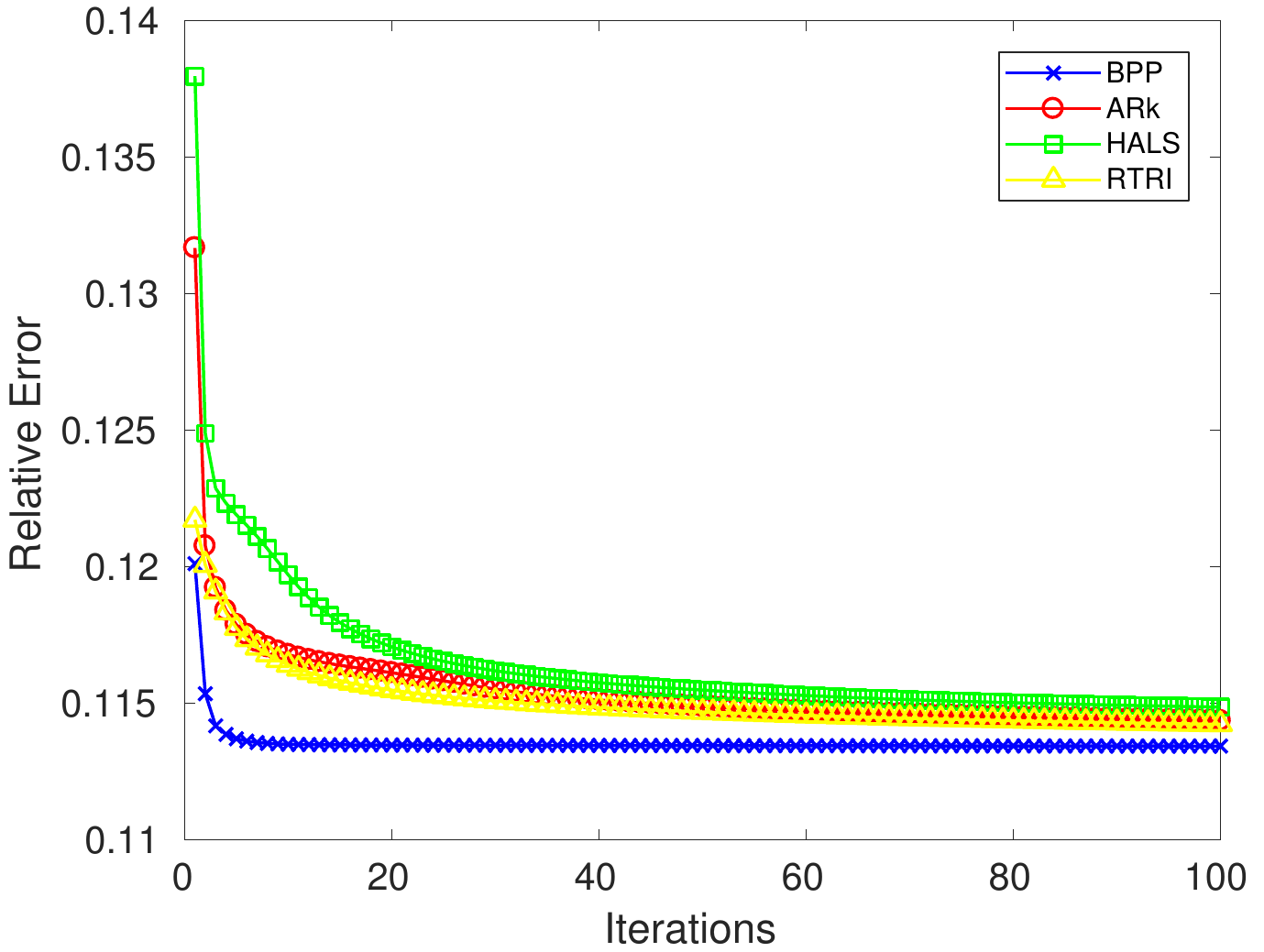}
\caption{Residual versus iterations}\label{fig:sfiters}
\end{subfigure}\hfill
\begin{subfigure}{.32\textwidth}
\centering\includegraphics[width=\linewidth]{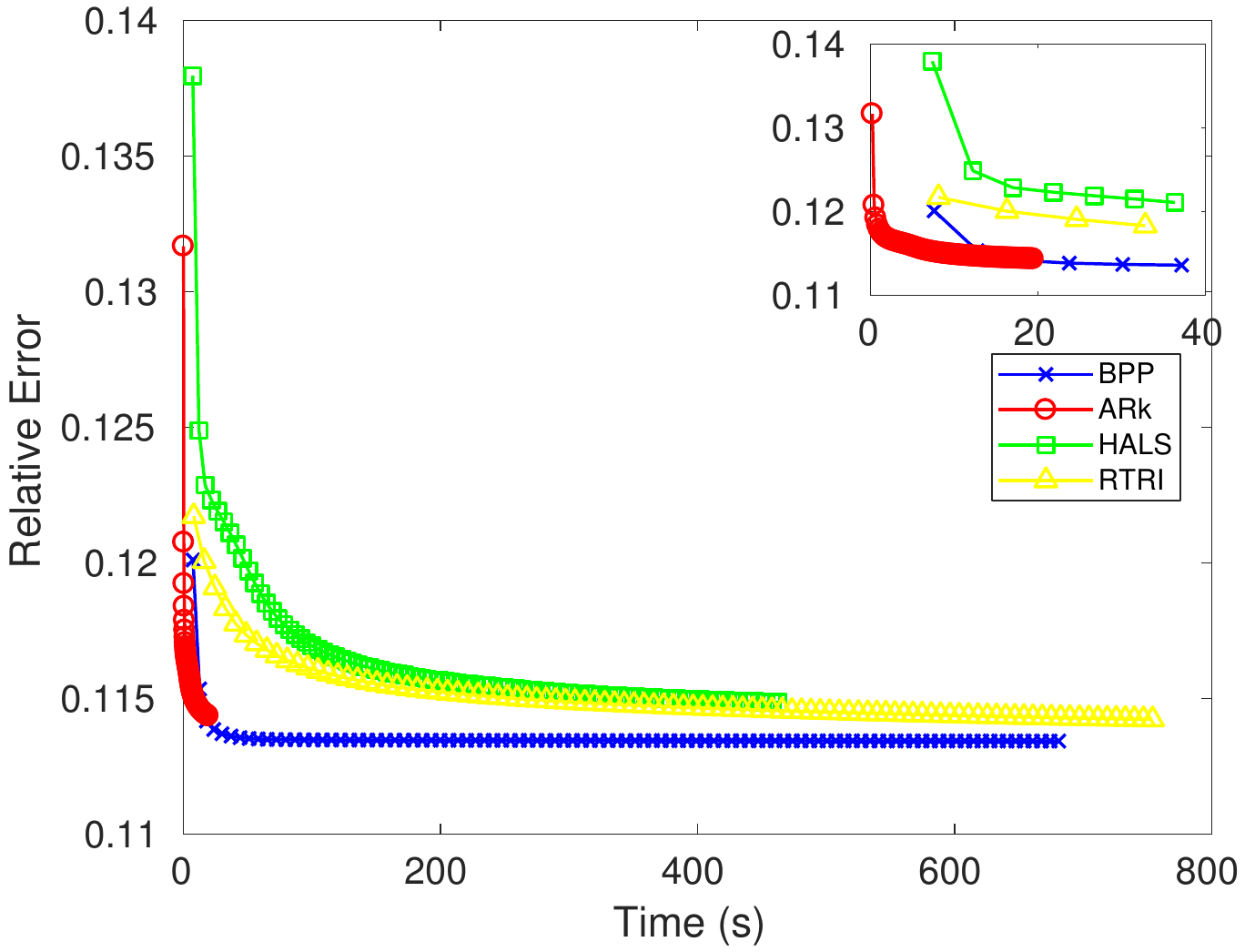}
\caption{Residual versus time}\label{fig:sftime}
\end{subfigure}\hfill
\begin{subfigure}{.32\textwidth}
\centering\includegraphics[width=\linewidth]{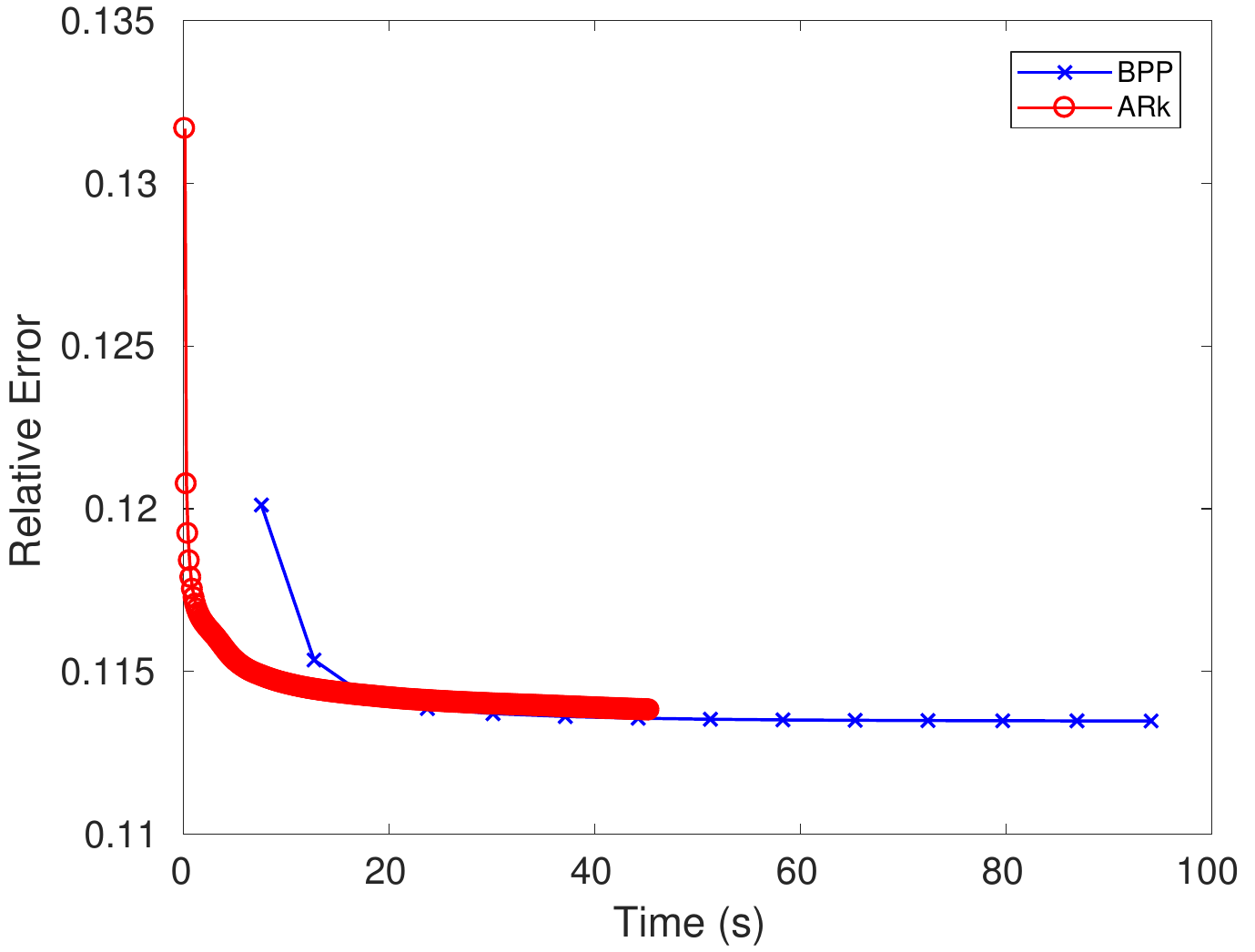}
\caption{Running ARk for 300 iterations}\label{fig:sfark}
\end{subfigure}
\caption{Short-Fat case ($ m < n$): $A \in \mathbb{R}_+^{10,000 \times 15,000}$ with $r=60$. BPP performs the best in terms of residual with ARk being next best and reaching similar residuals much faster than the other methods.}
\label{fig:synsf}
\end{figure}

\Cref{fig:synsf} shows the Short-Fat case with $m=10,000$ and $r=60$. BPP achieves the lowest residual while HALS, RTRI, and ARk perform slightly worse as seen in~\cref{fig:sfiters}.~\Cref{fig:sftime} shows the convergence with respect to time. 
All algorithms converge very quickly and BPP and ARk show the fastest drops in residual. It can be clearly seen that ARk is the most efficient of the algorithms often completing over half of its iterations before the others complete their first iteration. From~\Cref{fig:sfiters,fig:sftime} it looks like ARk might reach a lower residual if we allow it to run for a few more iterations and so allow it to run till 300 iterations in~\Cref{fig:sfark} and compare it to BPP. BPP is still better but the difference is marginal.  

\begin{figure}[tb]
\begin{subfigure}{.32\textwidth}
\centering\includegraphics[width=\linewidth]{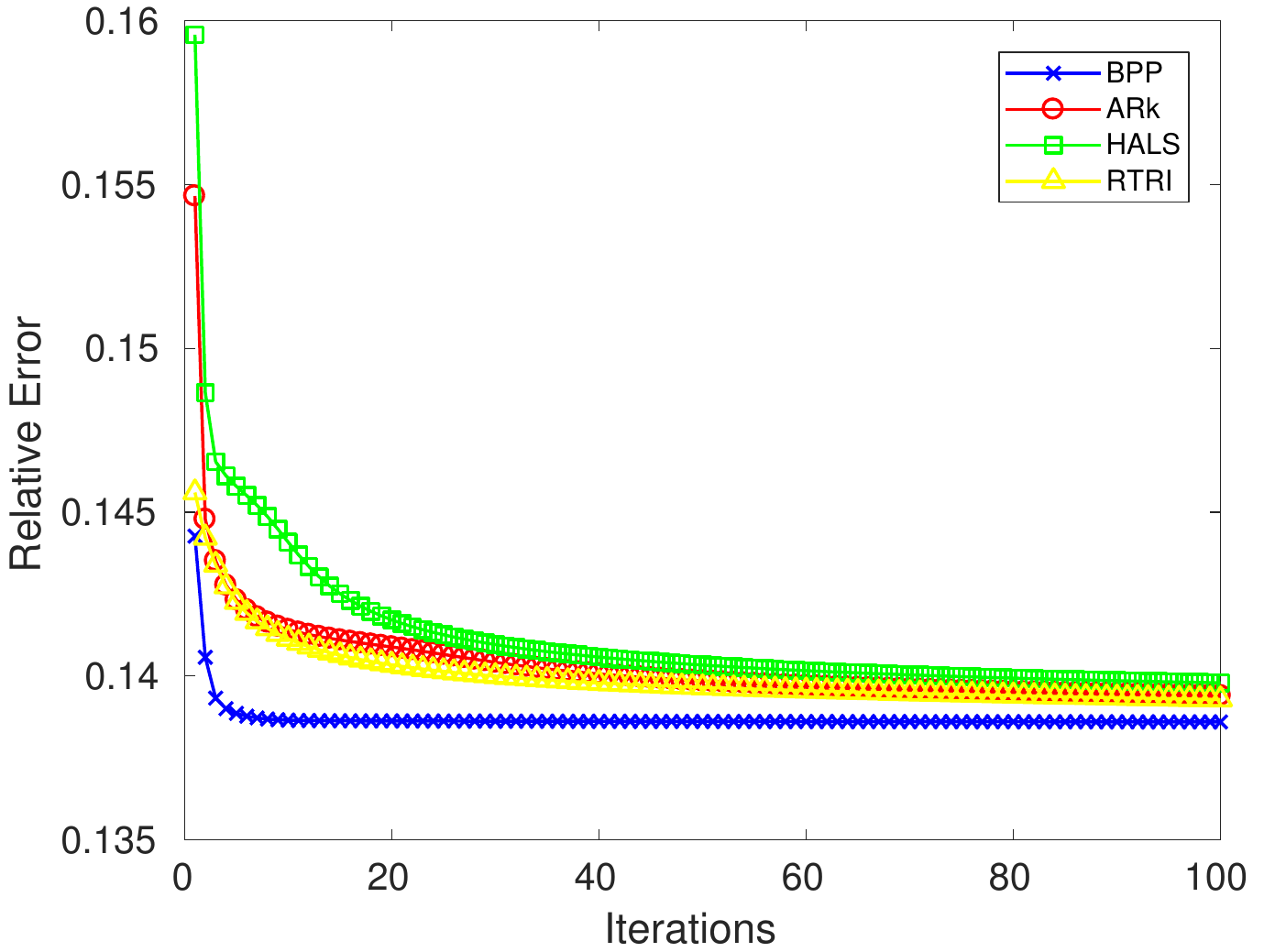}
\caption{Residual versus iterations}\label{fig:sqiters}
\end{subfigure}\hfill
\begin{subfigure}{.32\textwidth}
\centering\includegraphics[width=\linewidth]{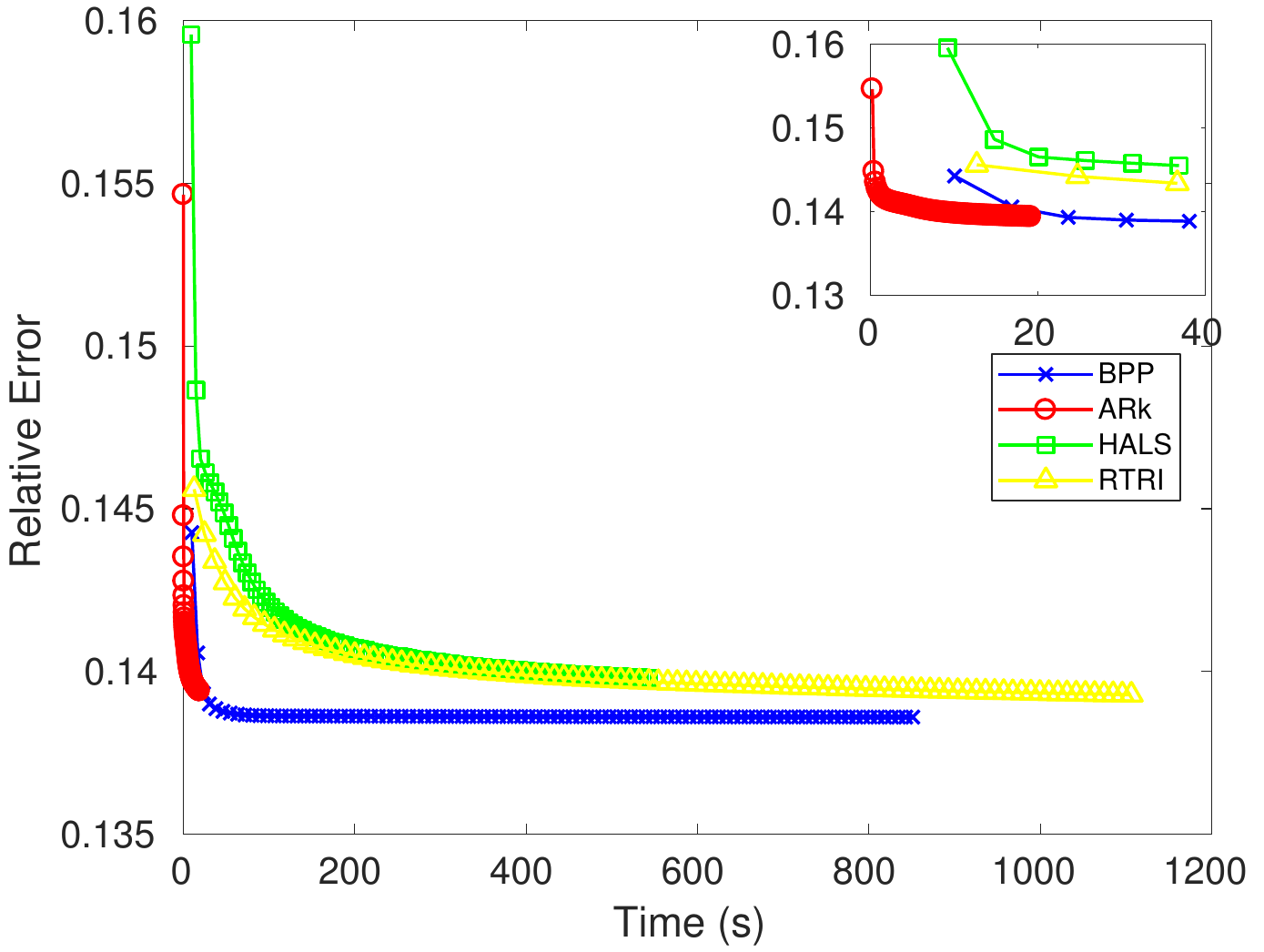}
\caption{Residual versus time}\label{fig:sqtime}
\end{subfigure}\hfill
\begin{subfigure}{.32\textwidth}
\centering\includegraphics[width=\linewidth]{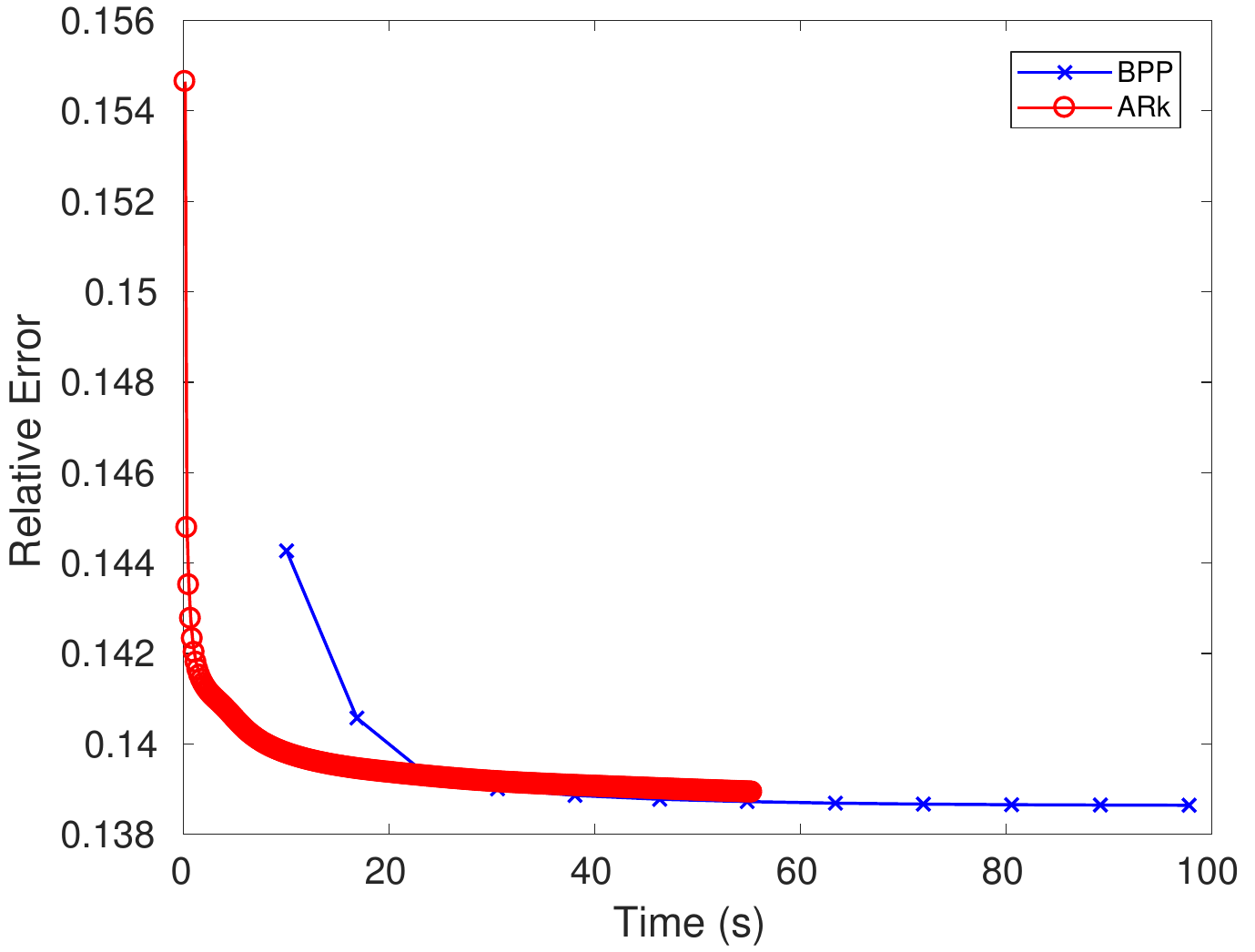}
\caption{Running ARk for 300 iterations}\label{fig:sqark}
\end{subfigure}
\caption{Square Case ($m=n$): $A \in \mathbb{R}_+^{15,000 \times 15,000}$ with $r=60$. BPP performs the best in terms of residual with ARk being next best and reaching similar residuals much faster than the other methods.}
\label{fig:synsq}
\end{figure}

\begin{figure}[htb]
\begin{subfigure}{.32\textwidth}
\centering\includegraphics[width=\linewidth]{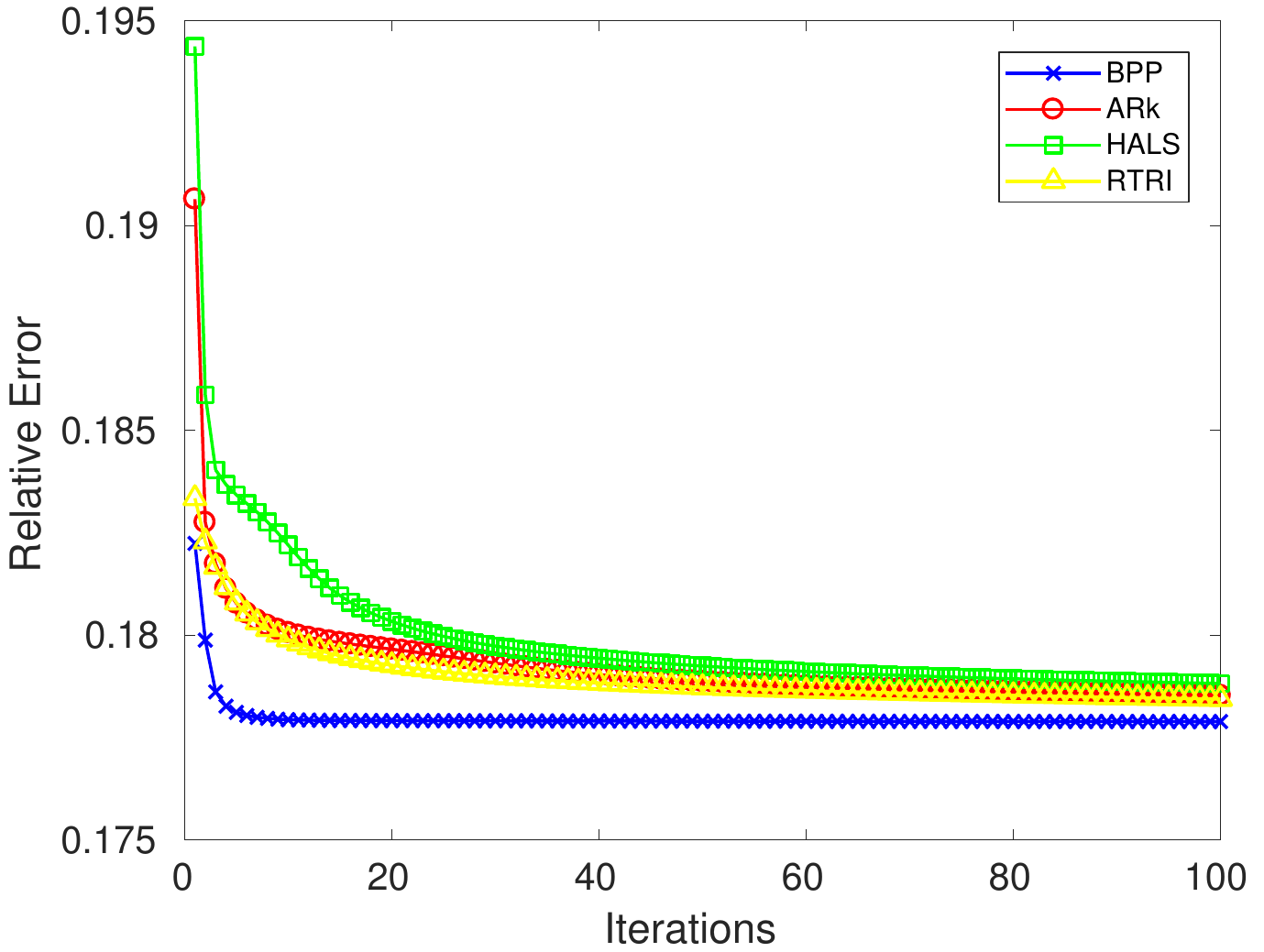}
\caption{Residual versus iterations}\label{fig:tsiters}
\end{subfigure}\hfill
\begin{subfigure}{.32\textwidth}
\centering\includegraphics[width=\linewidth]{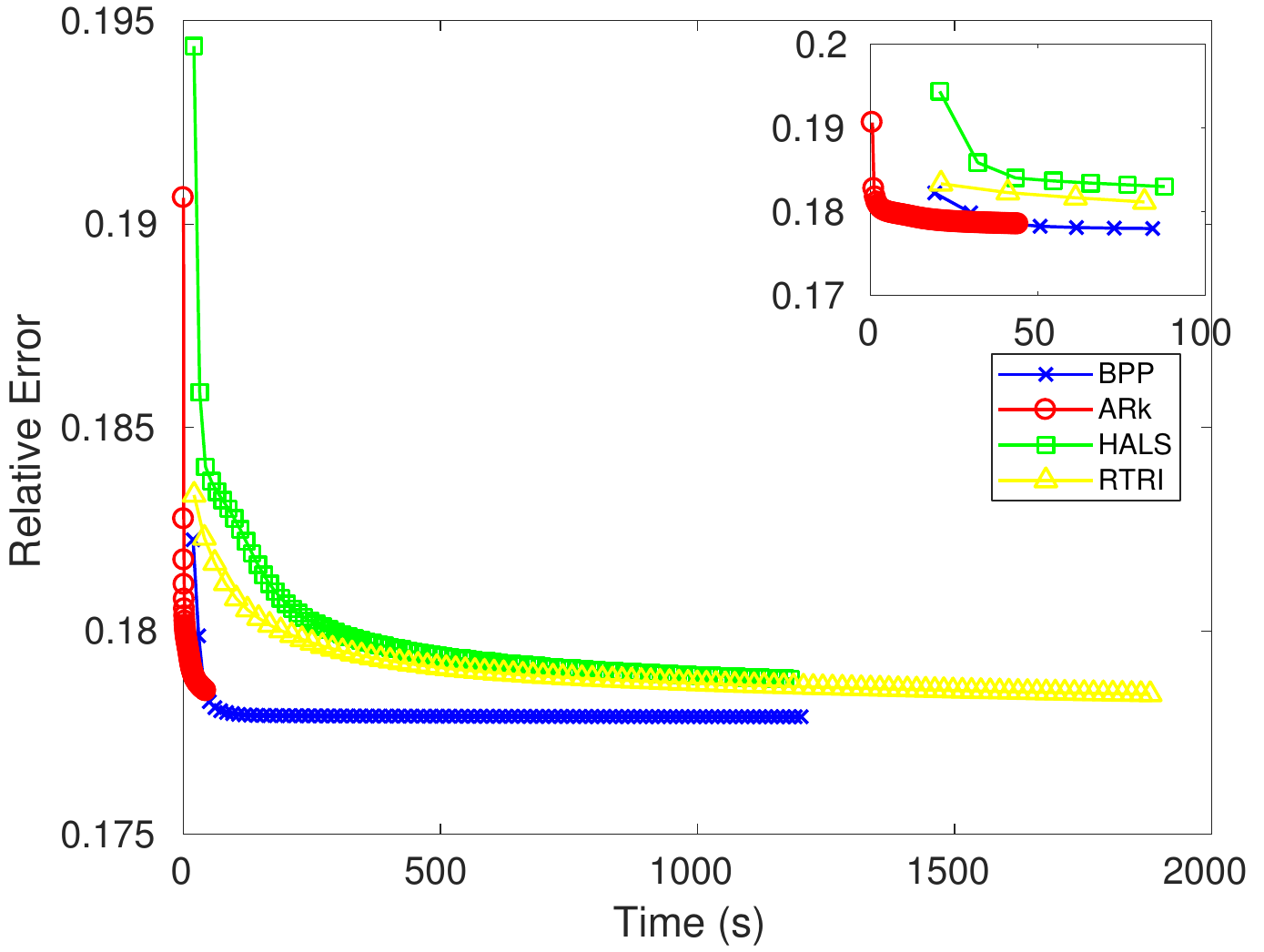}
\caption{Residual versus time}\label{fig:tstime}
\end{subfigure}\hfill
\begin{subfigure}{.32\textwidth}
\centering\includegraphics[width=\linewidth]{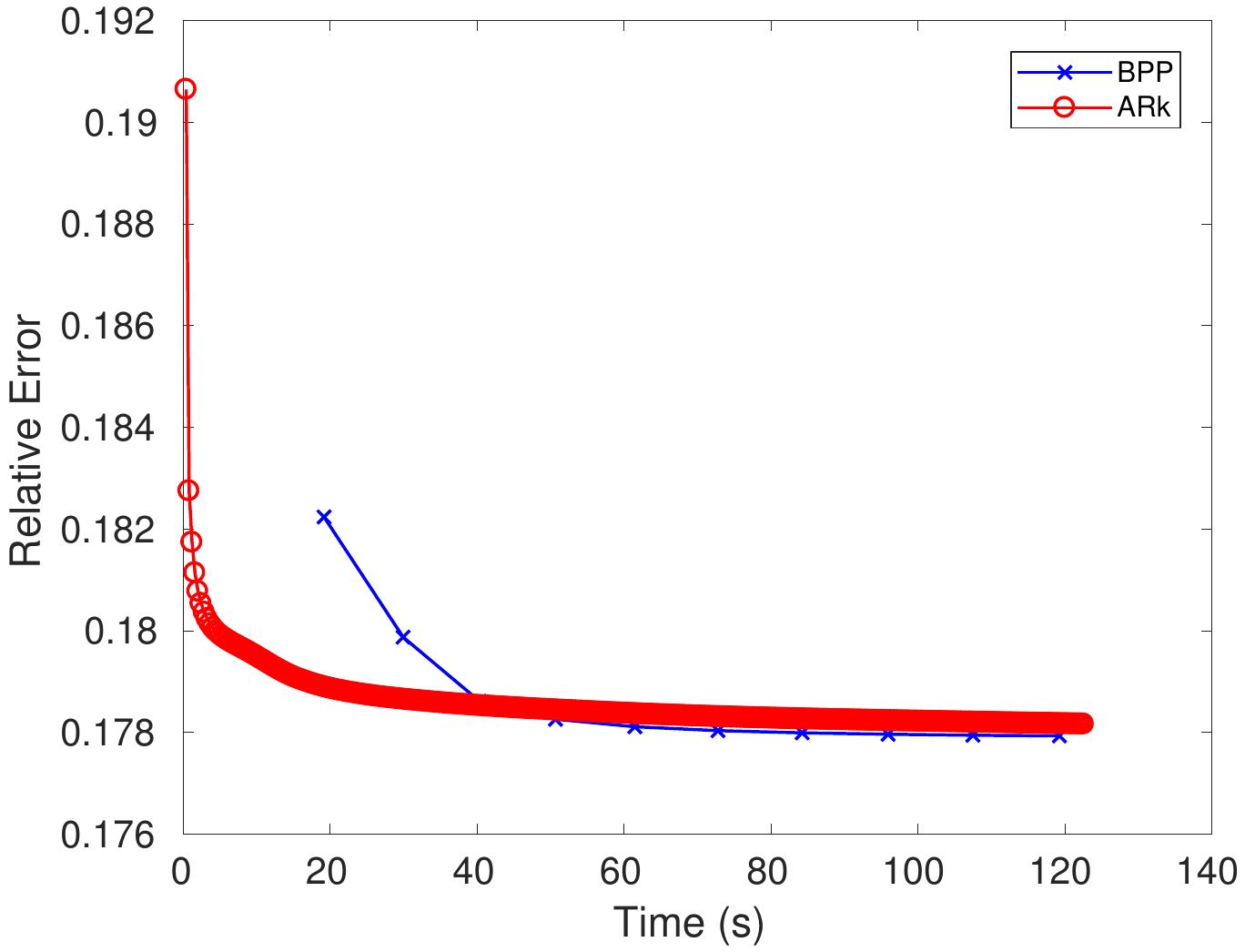}
\caption{Running ARk for 300 iterations}\label{fig:tsark}
\end{subfigure}
\caption{Tall-Skinny case ($m > n$) :$A \in \mathbb{R}_+^{25,000 \times 10,000}$ with $r=60$. BPP performs the best in terms of residual with ARk being next best and reaching similar residuals much faster than the other methods.}
\label{fig:synts}
\end{figure}

\Cref{fig:synsq,fig:synts} show the Square and Tall-Skinny cases respectively. The observations from the Short-Fat case can be carried forward to these as well. Similar results were obtained for the other choices of $r$ and $m$ and we omit them from this section for ease of presentation.
%\se{Should we add an appendix with all results or is that overkill?}\hp{no need}

%Synthetic Experiments - Rank-sweep
\subsubsection{Experiments with Various Ranks}

\begin{figure}[ht]
\begin{subfigure}{.48\textwidth}
\centering\includegraphics[width=0.7\linewidth]{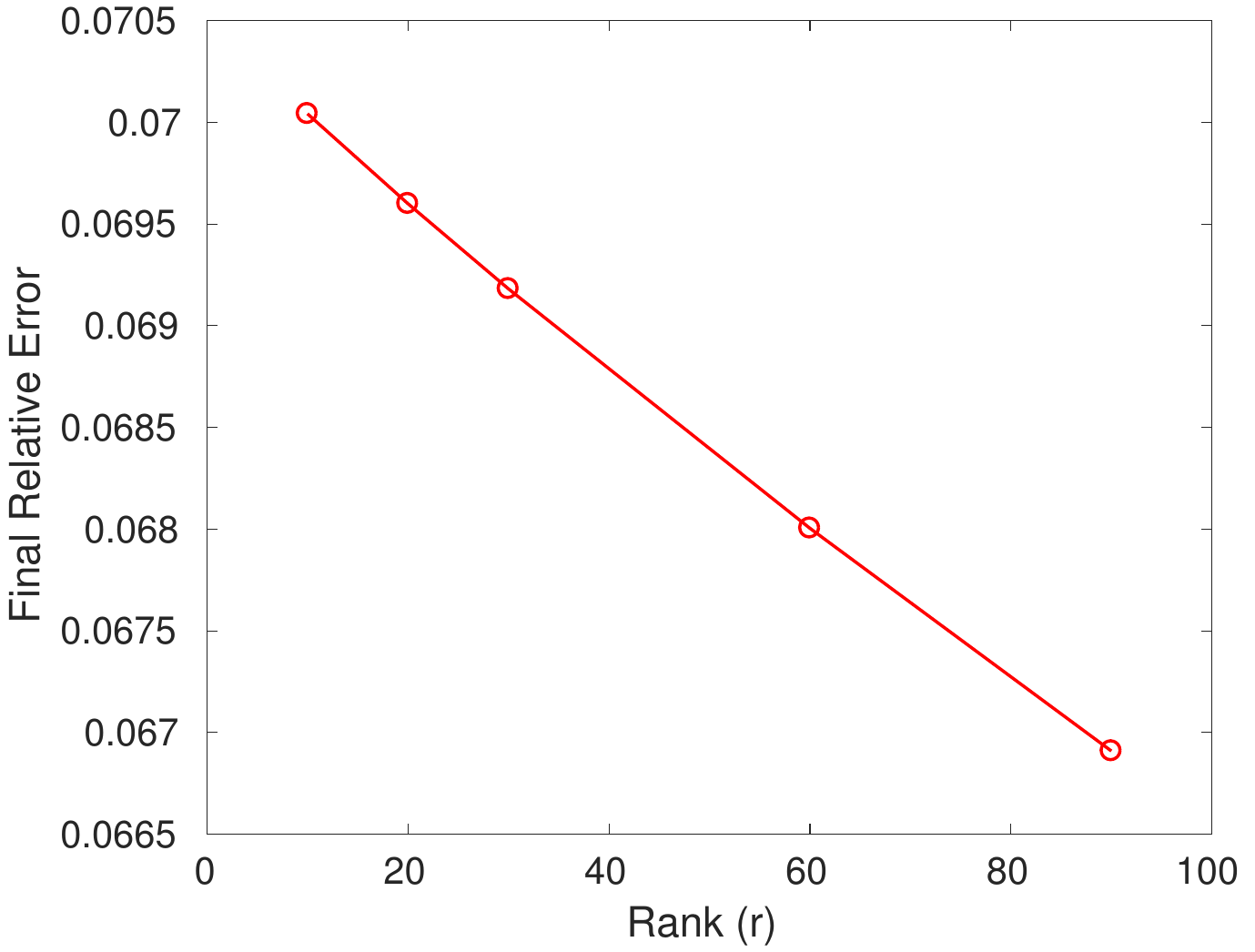}
\caption{ARk approximation}\label{fig:rsweepark}
\end{subfigure}\hfill
\begin{subfigure}{.48\textwidth}
\centering\includegraphics[width=0.7\linewidth]{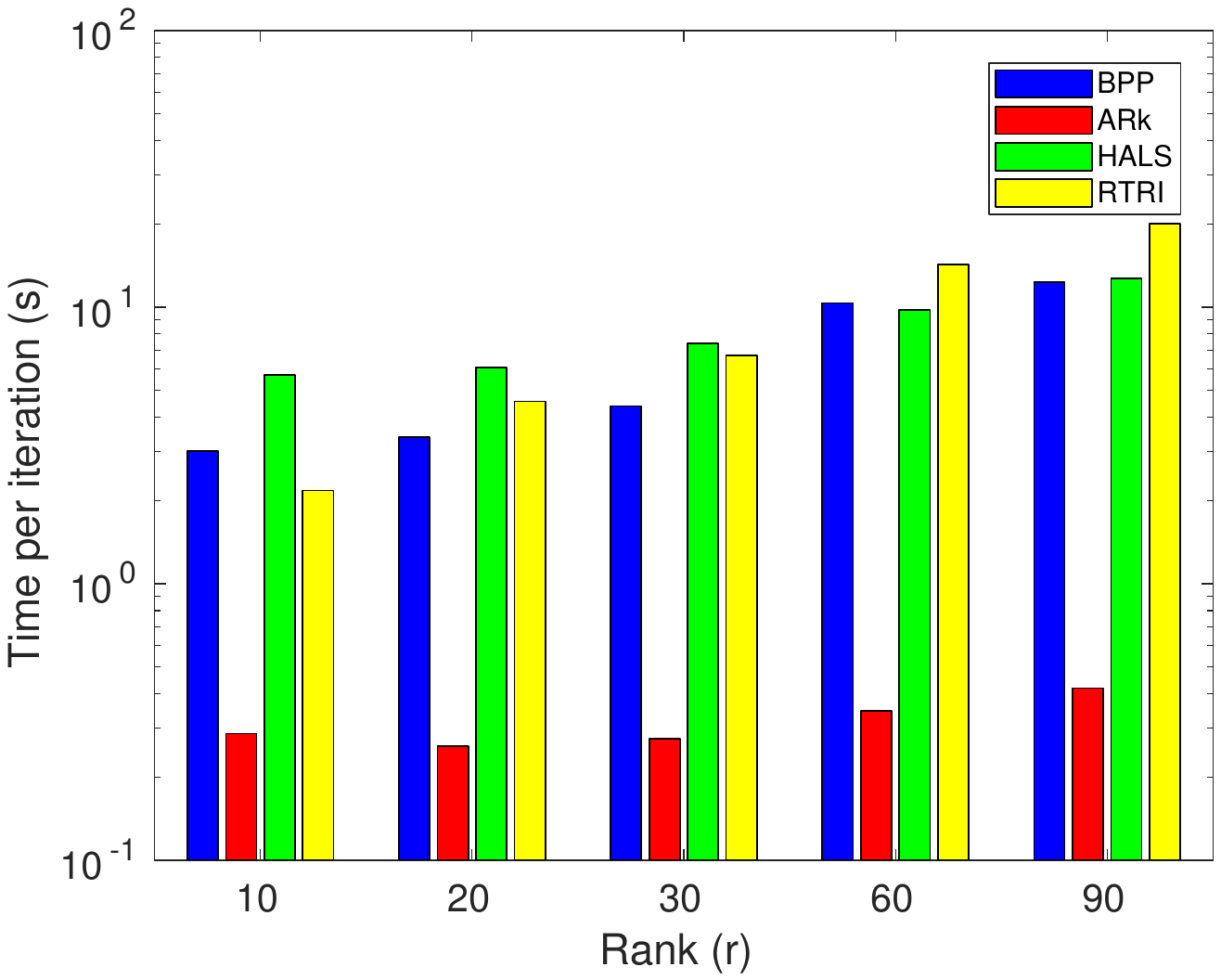}
\caption{Time per iteration}\label{fig:rsweeptime}
\end{subfigure}
\caption{Rank sweep experiments on $A \in \mathbb{R}_+^{20,000 \times 15,000}$. ARk achieves better approximations with increased $r$ as expected and maintains its computational efficiency over the other methods.}
\label{fig:rsweep}
\end{figure}
%\hp{Srini, state what we mean exactly by one iteration. Strange HALS took so much longer}
We test the effect of increasing the approximation rank ($r$) on the different algorithms. First we check the approximation quality of ARk when the $r$ in increased in~\Cref{fig:rsweepark} on a synthetic matrix $A \in \mathbb{R}_+^{20,000 \times 15,000}$ with a low-rank of 150. We can see that ARk achieves a good approximation even at $r=10$ and progressively gets better when $r$ is increased. Next we see the effects on running time when $r$ increases.~\Cref{fig:rsweeptime} shows the time per iteration, that is the time taken to update all columns of $U$ and $V$ once, of the different algorithms as $r$ increases on a matrix with $m=20,000$ and $n=15,000$. All algorithms show a moderate increase in time as $r$ increases. ARk maintains about 10 times faster $10\times$ computational speed over the other algorithms for all $r$.

%Synthetic Experiments - Sparse
\subsubsection{Experiments with Sparse Synthetic Data}
\label{sec:sparse}
The sparse synthetic matrices are created in the following manner. We first generate a dense low-rank nonnegative matrix $L \in \mathbb{R}_+^{m \times n}$ as shown in~\Cref{sec:kselect}. Then we generate a uniform random sparse matrix $X \in \mathbb{R}_+^{m \times n}$ with the desired sparsity $\rho$ and element-wise multiply it with $L$ to obtain our synthetic matrix $A = X * L$. Here $*$ denotes the element-wise product of matrices. It must be noted that this matrix is not truly low-rank due to the element-wise product.

\begin{figure}[ht]
\begin{subfigure}{.24\textwidth}
\centering\includegraphics[width=\linewidth]{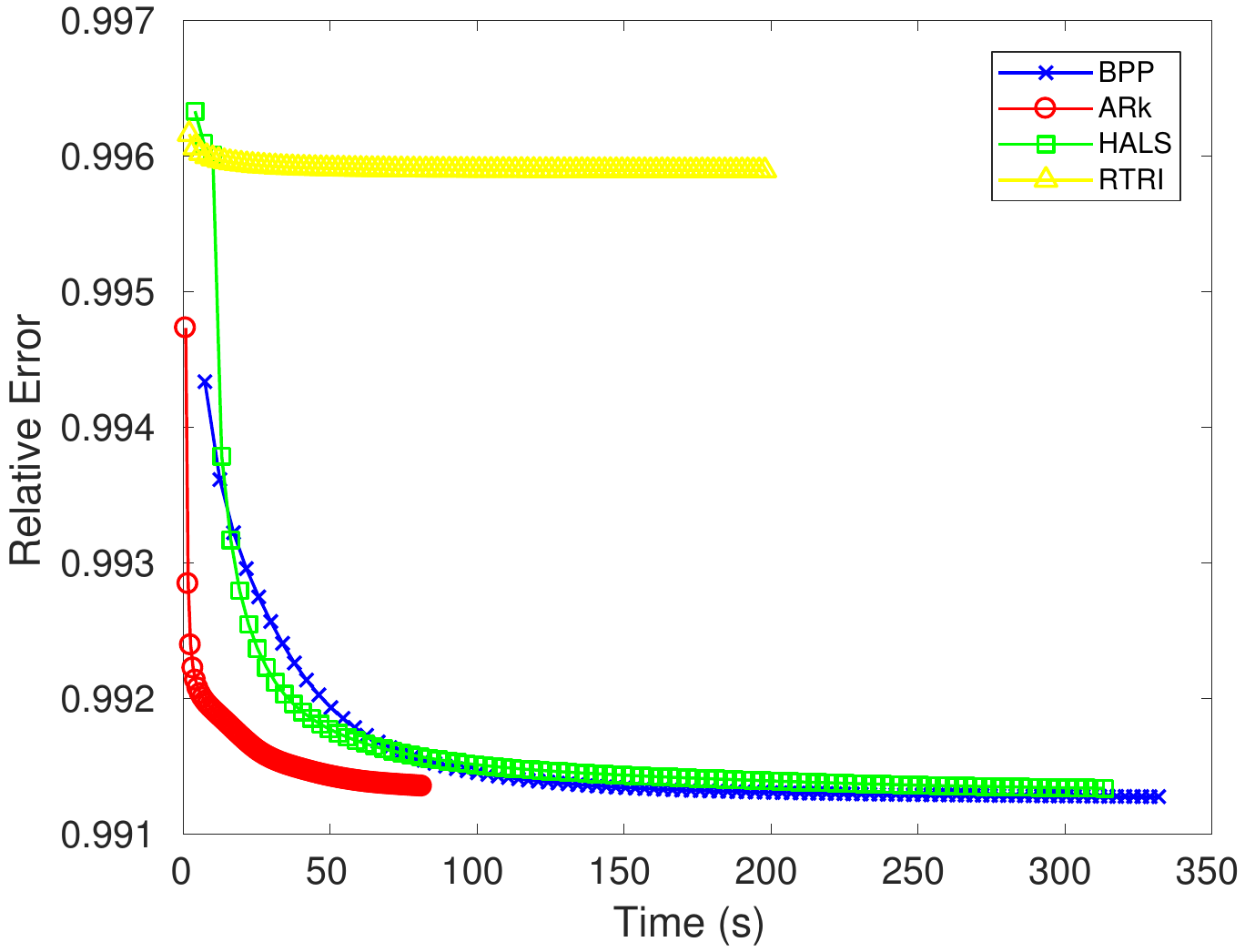}
\caption{Sparsity 0.01}\label{fig:sp1}
\end{subfigure}\hfill
\begin{subfigure}{.24\textwidth}
\centering\includegraphics[width=\linewidth]{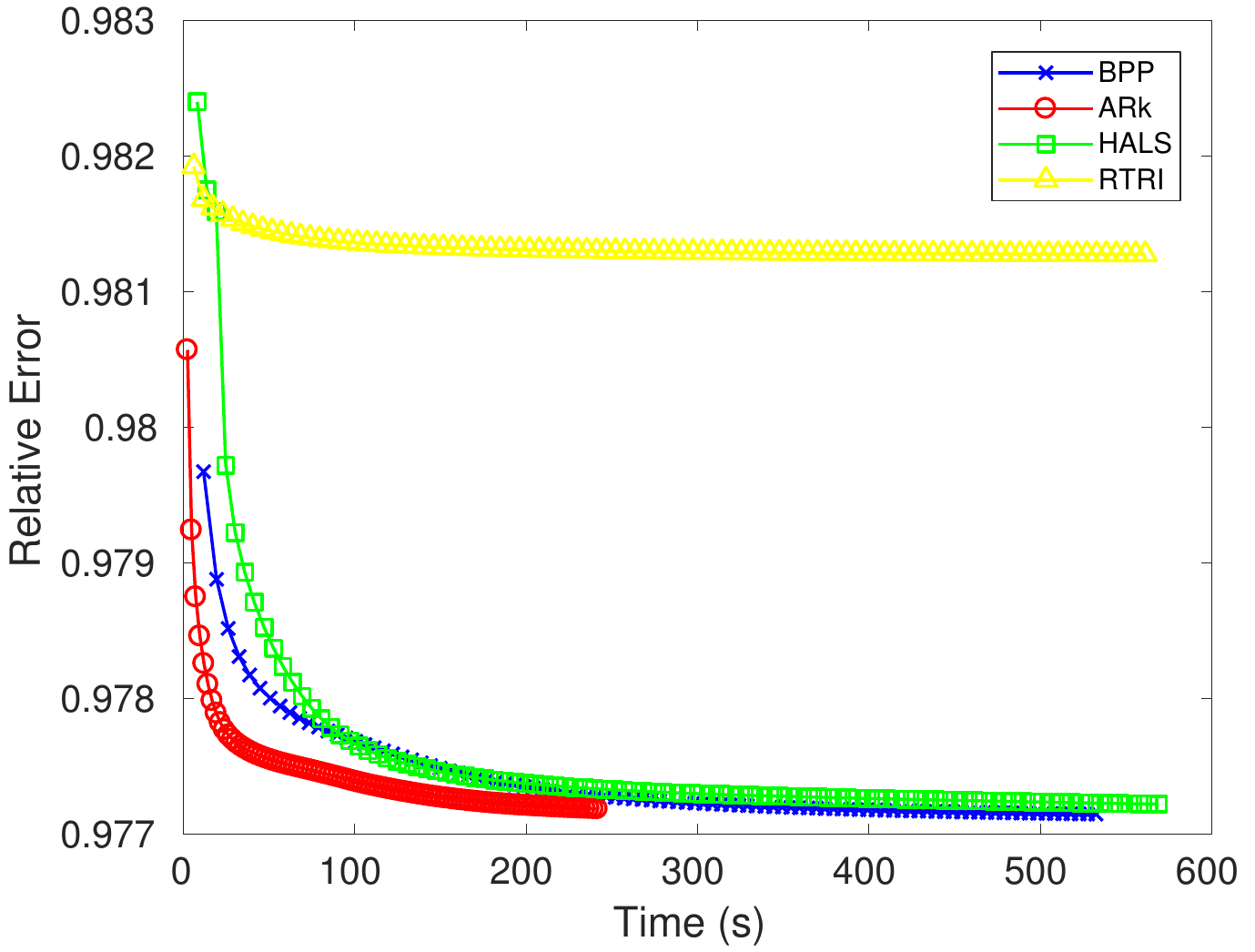}
\caption{Sparsity 0.05}\label{fig:sp2}
\end{subfigure}\hfill
\begin{subfigure}{.24\textwidth}
\centering\includegraphics[width=\linewidth]{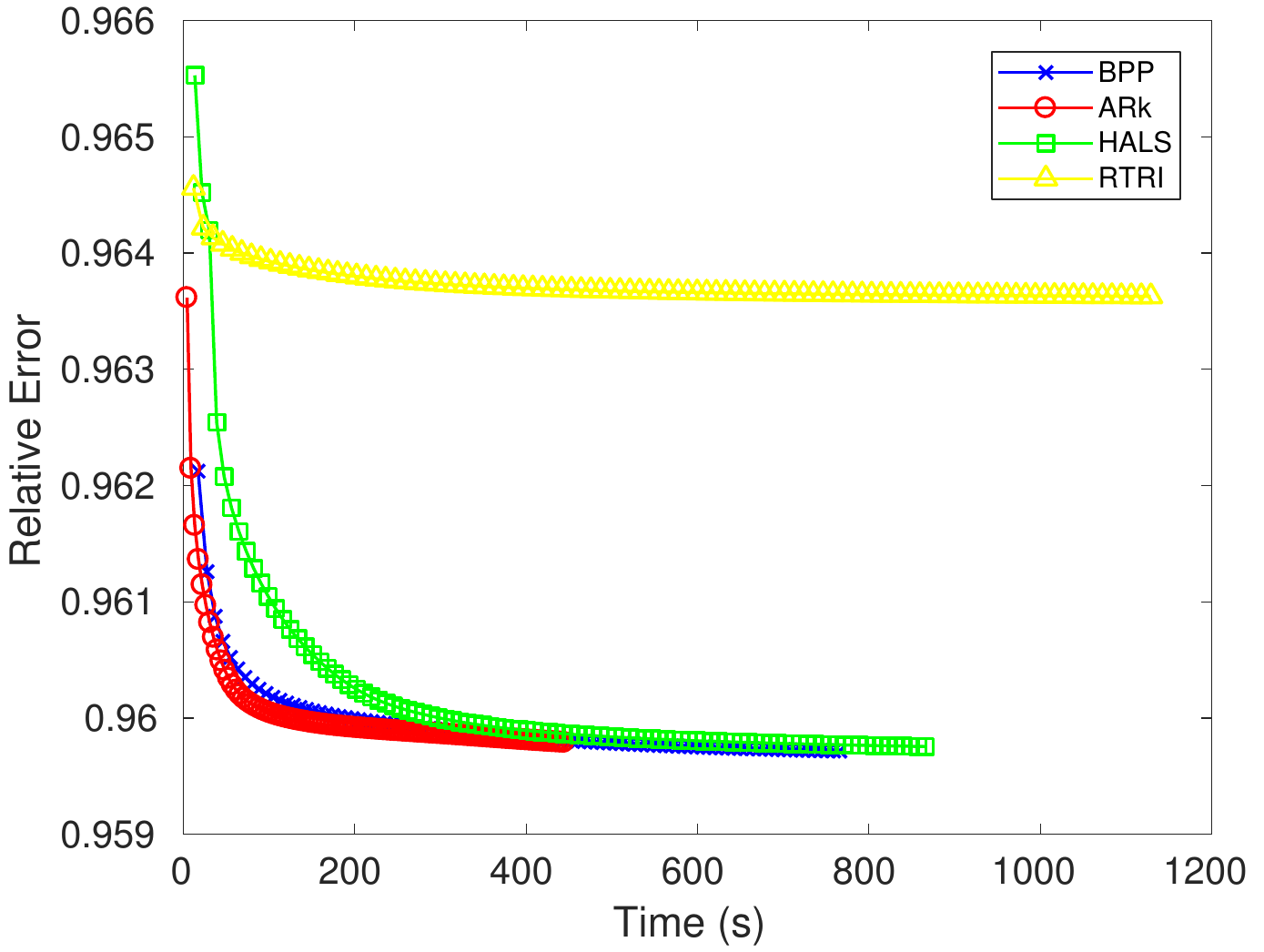}
\caption{Sparsity 0.1}\label{fig:sp3}
\end{subfigure}\hfill
\begin{subfigure}{.24\textwidth}
\centering\includegraphics[width=\linewidth]{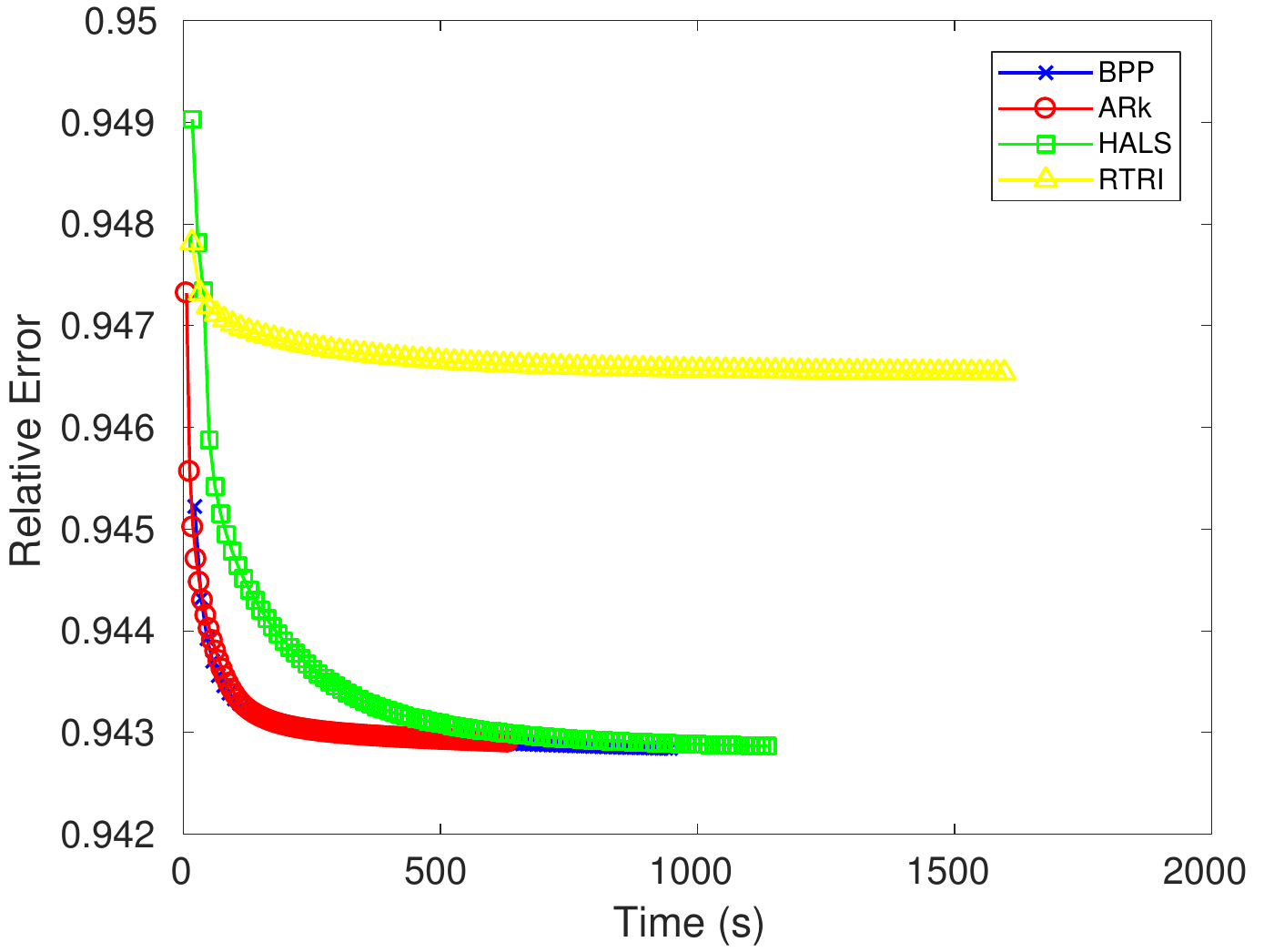}
\caption{Sparsity 0.15}\label{fig:sp4}
\end{subfigure}
\caption{Sparse case: $A \in \mathbb{R}_+^{20,000\times 15,000}$ with varying sparsities. ARk achieves the lowest residual errors in the shortest computational time.}
\label{fig:spcase}
\end{figure}

\Cref{fig:spcase} shows the performance of our four algorithms on a sparse $20,000\times 15,000$ matrix with $r=60$ and varying sparsity. We can see that BPP, ARk, and HALS achieve similar approximation errors while RTRI produces larger residual values. ARk is able to achieve the best relative error within the shortest time. Its relative speedup over the other algorithms is less than the dense case but it still is faster by a factor of $2-3$ compared to the other methods. 

\subsection{Experiments on Real World Data}
% Raw data can be found here: https://docs.google.com/spreadsheets/d/1FFNqFM95tg2juzAmSxi9JP9hg3OpVV8_bvqg7c3Jh60/edit?usp=sharing
\begin{table}[ht]
\caption{Convergence results on real world data. All methods are run till the maximum time specfied and the final relative error is captured.}\label{tbl:rwres}
\begin{tabular}{ccccccccccccc}
\hline
\multirow{2}{*}{Dataset}                       & \multirow{2}{*}{Rank} & \multirow{2}{*}{Time (s)} & \multirow{2}{*}{AE (\%)} & \multicolumn{4}{c}{Final relative error} \\
                  & & &  & BPP  & ARk & HALS  & RTRI \\
\hline
\multirow{4}{*}{TDT2}         & 60   & 31   & 1                & 0.7847~$\pm$~0.0006 & 0.7788~$\pm$~0.0005  & 0.7841~$\pm$~0.0020   & 0.8585~$\pm$~0.0012 \\
                              & 90   & 42   & 1                & 0.7548~$\pm$~0.0004 & 0.7476~$\pm$~0.0004  & 0.7567~$\pm$~0.0022  & 0.8584~$\pm$~0.0040  \\
                              & 120  & 60   & 1                & 0.7312~$\pm$~0.0012 & 0.7247~$\pm$~0.0003  & 0.7340~$\pm$~0.0009   & 0.8603~$\pm$~0.0039 \\
                              & 150  & 84   & 1                & 0.7123~$\pm$~0.0007 & 0.7062~$\pm$~0.0003  & 0.7161~$\pm$~0.0014  & 0.8575~$\pm$~0.0023 \\
                              \hline
\multirow{4}{*}{Reuters} & 60   & 24   & 1                & 0.6995~$\pm$~0.0016 & 0.6935~$\pm$~0.0006  & 0.7006~$\pm$~0.0012  & 0.7907~$\pm$~0.0025 \\
                              & 90   & 38   & 1                & 0.6693~$\pm$~0.0002 & 0.6633~$\pm$~0.0003  & 0.6715~$\pm$~0.0008  & 0.7939~$\pm$~0.0058 \\
                              & 120  & 56   & 1                & 0.6459~$\pm$~0.0011 & 0.6402~$\pm$~0.0006  & 0.6482~$\pm$~0.0015  & 0.7951~$\pm$~0.0018 \\
                              & 150  & 72   & 1                & 0.6258~$\pm$~0.0015 & 0.6204~$\pm$~0.0002  & 0.6293~$\pm$~0.0012  & 0.7931~$\pm$~0.0037 \\
                              \hline                              
\multirow{4}{*}{20Newsgroups} & 60   & 40   & 1                & 0.5943~$\pm$~0.0017 & 0.5909~$\pm$~0.0015  & 0.5965~$\pm$~0.0017  & 0.6412~$\pm$~0.0054 \\
                              & 90   & 70   & 1                & 0.5626~$\pm$~0.0014 & 0.5582~$\pm$~0.0010   & 0.5642~$\pm$~0.0006  & 0.6385~$\pm$~0.0025 \\
                              & 120  & 89   & 1                & 0.5394~$\pm$~0.0006 & 0.5349~$\pm$~0.0005  & 0.5417~$\pm$~0.0011  & 0.6379~$\pm$~0.0066 \\
                              & 150  & 108  & 1                & 0.5212~$\pm$~0.0017 & 0.5172~$\pm$~0.0010   & 0.5236~$\pm$~0.0013  & 0.6354~$\pm$~0.0049 \\
                              \hline
\multirow{4}{*}{ORL}          & 60   & 58   & 1                & 0.1393~$\pm$~0.0001 & 0.1369~$\pm$~0.0001  & 0.1423~$\pm$~0.0004  & 0.1381~$\pm$~0.0001 \\
                              & 90   & 99   & 1                & 0.1241~$\pm$~0.0002 & 0.1212~$\pm$~0.0001  & 0.1263~$\pm$~0.0003  & 0.1229~$\pm$~0.0001 \\
                              & 120  & 162  & 1                & 0.1129~$\pm$~0.0001 & 0.1092~$\pm$~0.0001  & 0.1139~$\pm$~0.0001  & 0.1129~$\pm$~0.0004 \\
                              & 150  & 305  & 1                & 0.1038~$\pm$~0.0001 & 0.0988~$\pm$~0.0001  & 0.1030~$\pm$~0.0002   & 0.1062~$\pm$~0.0003 \\
                              \hline                              
\multirow{4}{*}{Facescrub}    & 60   & 454  & 1                & 0.1593~$\pm$~0.0001 & 0.1584~$\pm$~0.0001  & 0.1630~$\pm$~0.0003   & 0.2370~$\pm$~0.0080   \\
                              & 90   & 701  & 1                & 0.1435~$\pm$~0.0001 & 0.1420~$\pm$~0.0000  & 0.1470~$\pm$~0.0004   & 0.2319~$\pm$~0.0028 \\
                              & 120  & 540  & 5                & 0.1351~$\pm$~0.0004 & 0.1308~$\pm$~0.0001  & 0.1400~$\pm$~0.0003    & 0.2304~$\pm$~0.0045 \\
                              & 150  & 746  & 15               & 0.1369~$\pm$~0.0005 & 0.1223~$\pm$~0.0000 & 0.1317~$\pm$~0.0002  & 0.2278~$\pm$~0.0036 \\
                              \hline
\multirow{4}{*}{YaleB}        & 60   & 70   & 1                & 0.1679~$\pm$~0.0003 & 0.1655~$\pm$~0.0001  & 0.1751~$\pm$~0.0005  & 0.1731~$\pm$~0.0010  \\
                              & 90   & 126  & 1                & 0.1464~$\pm$~0.0003 & 0.1437~$\pm$~0.0000 & 0.1541~$\pm$~0.0007  & 0.1515~$\pm$~0.0008 \\
                              & 120  & 194  & 1                & 0.1309~$\pm$~0.0002 & 0.1283~$\pm$~0.0001  & 0.1376~$\pm$~0.0004  & 0.1396~$\pm$~0.0009 \\
                              & 150  & 364  & 1                & 0.1194~$\pm$~0.0002 & 0.1160~$\pm$~0.0001   & 0.1245~$\pm$~0.0004  & 0.1316~$\pm$~0.0008 \\
                              \hline
\multirow{4}{*}{Caltech256}   & 60   & 495  & 1                & 0.2067~$\pm$~0.0002 & 0.2054~$\pm$~0.0000 & 0.2091~$\pm$~0.0004  & 0.2812~$\pm$~0.0039 \\
                              & 90   & 763  & 1                & 0.1910~$\pm$~0.0001  & 0.1899~$\pm$~0.0001  & 0.1939~$\pm$~0.0000 & 0.2785~$\pm$~0.0041 \\
                              & 120  & 567  & 5                & 0.1845~$\pm$~0.0002 & 0.1793~$\pm$~0.0001  & 0.1868~$\pm$~0.0002  & 0.2756~$\pm$~0.0058 \\
                              & 150  & 732  & 15               & 0.1944~$\pm$~0.0003 & 0.1708~$\pm$~0.0001  & 0.1787~$\pm$~0.0002  & 0.2716~$\pm$~0.0050  \\
                              \hline                              
\end{tabular}
\end{table}

We run the four algorithms on the real world data described in~\Cref{tbl:rwdata}. We test a suite of approximation rank $r$, varying from 60 to 150 in increments of 30. In these experiments we want to measure how the algorithms converge over time and use an upper bound on time as the stopping criteria. The maximum time is selected as follows. From~\ref{sec:synexp} we can see that generally BPP achieves the lowest residual after 100 iterations. We run BPP for 100 iterations on all the datasets which then select the time needed to reach within 1\%,5\%,10\% or 15\% of that final error. This approximation is chosen to keep the overall running time of each algorithm under 15 minutes. For most the datasets we are able to approximate upto 1\% of error except for some instances of the larger Caltech and Facescrub runs. The results of these displayed in~\cref{tbl:rwres}. The table contains the final residual of all the algorithms averaged over 5 runs with different initialisations. AE is the approximation percentange of the error with respect to running BPP for 100 iterations.

We can observe that ARk is th best performing algorithm followed by BPP, HALS, and RTRI in that order. RTRI performs particularly poorly for the sparse inputs. The main reason for ARk performing well on these experiments is due to its computational efficiency. ARk is able to perform thousands of iterations in the same time as it takes BPP or HALS to perform tens. This can be clearly seen in~\cref{fig:rwcase} where we can see that ARk has converged very quickly and before BPP can complete a single iteration. It is true that given enough time the other algorithms, especially BPP, might find a better solution but this could be prohibitively expensive especially for larger values $r$. From these experiments we can conclude that ARk strikes a good balance between accuracy and computational efficiency and discovers good approximations much faster than the the other methods compared in this work.

\begin{figure}[tb]
\begin{subfigure}{.48\textwidth}
\centering\includegraphics[width=0.7\linewidth]{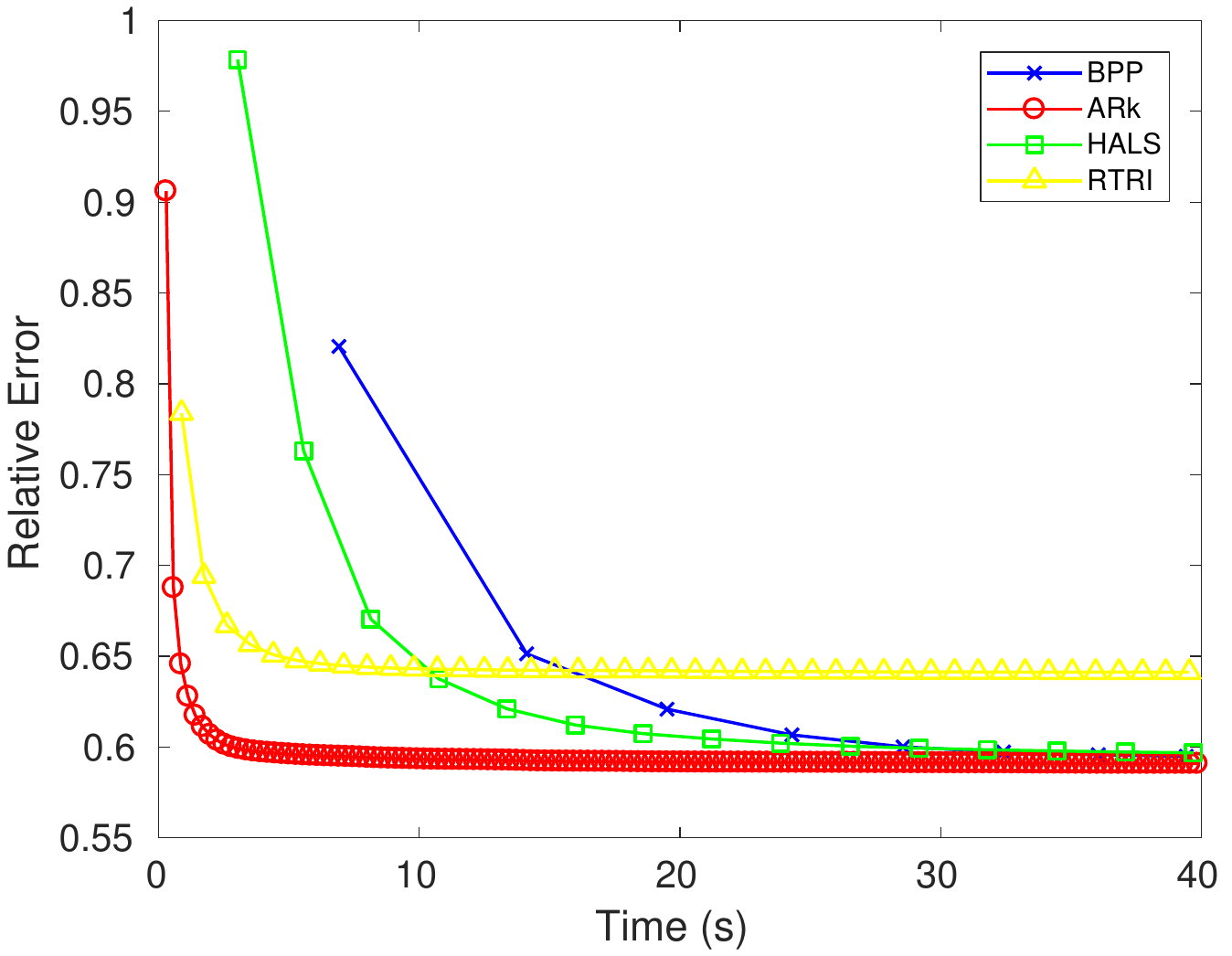}
\caption{20Newsgroups with $r=60$}\label{fig:20ng}
\end{subfigure}
\begin{subfigure}{.48\textwidth}
\centering\includegraphics[width=0.7\linewidth]{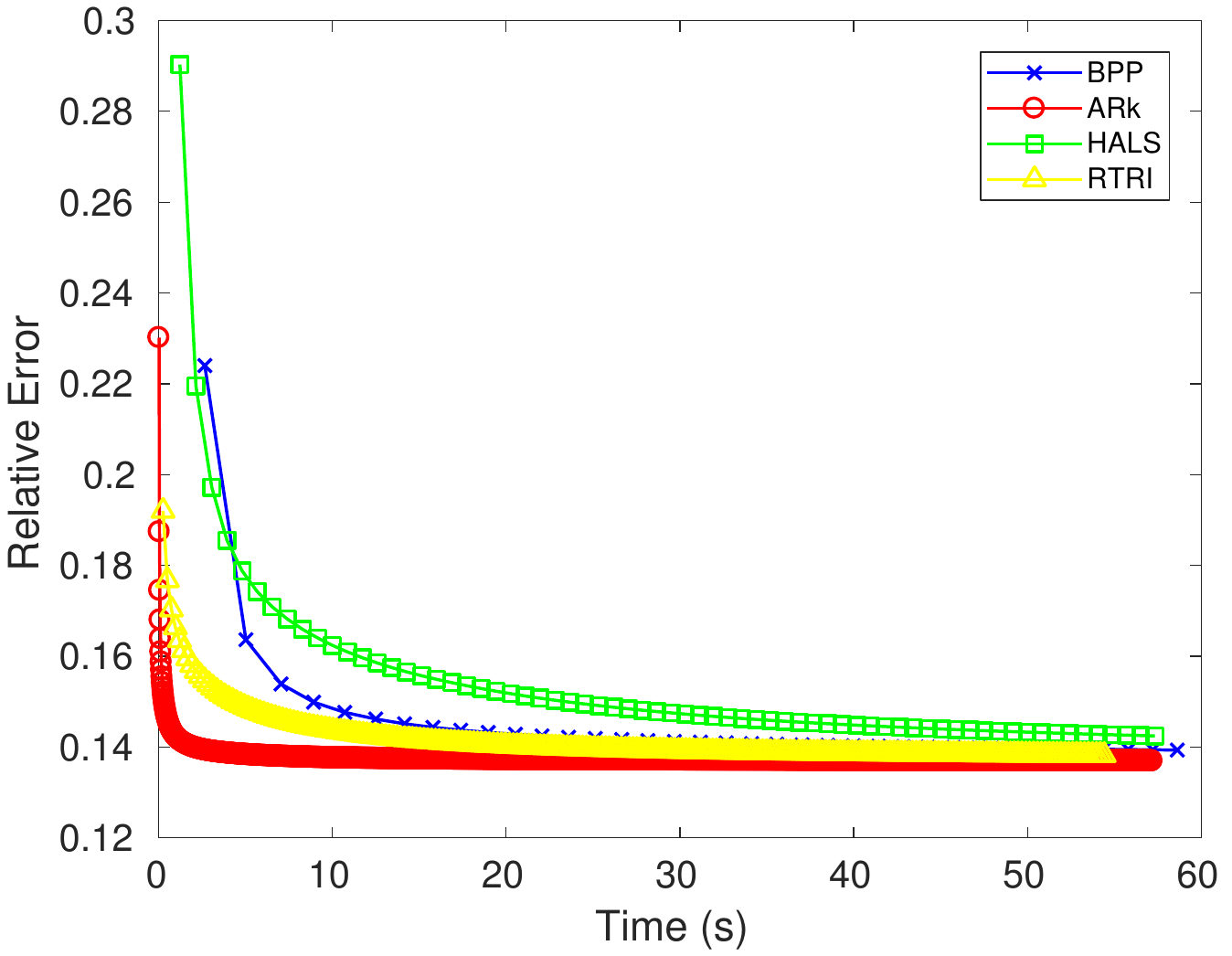}
\caption{ORL with $r=60$}\label{fig:orl}
\end{subfigure}
\caption{Some typical convergence plots for real world world data. ARk is able to achieve good approximation error in a fraction of the time taken by the other methods.}
\label{fig:rwcase}
\end{figure}
%5
\section{ Concluding Remarks}\label{sec6}
In this paper, we have established the recursive formula for the solutions of the \emph{rank-k NLS} and
developed an alternating \emph{rank-k} nonnegative least squares framework ARkNLS for NMF based on this  recursive formula.
We have studied ARkNLS with $k=3$ further which builds upon the \emph{rank-3} residue iteration for NLS that
updates two more columns than HALS per updating step. We have also introduced an new strategy that efficiently
overcomes the potential singularity problem within the context of NMF computation. Extensive numerical comparisons
using real datasets demonstrate that our new algorithm ARkNLS(k=3) provides state-of-the-art performance in terms
of computational accuracy and cpu time.

%6

% Compiles locally but not on server
\bibliographystyle{siamplain}
\bibliography{arknmf}

\end{document}